\documentclass{amsart}
\usepackage{amsmath}
\usepackage{amssymb}
\usepackage{mathabx}
\usepackage[all]{xy}%commutative graph
\usepackage{amsbsy}
\usepackage{graphicx}
\usepackage{appendix}
\usepackage{float}
\usepackage{subfigure}
\usepackage[numbers,sort&compress]{natbib}
\usepackage{graphicx}%pictures
\usepackage{amsthm}
\usepackage{geometry}
\usepackage{fancyhdr}%Heads
\usepackage{color} %colors
\usepackage{overpic}%pictures with text
\usepackage{mathabx}
\usepackage{tikz-cd}
\usepackage{enumerate}
\usepackage{colonequals}
\usepackage{microtype}
 \usepackage{hyperref}

\newtheorem{thm}{Theorem}[section]
\newtheorem{cor}[thm]{Corollary}
\newtheorem{prop}[thm]{Proposition}
\newtheorem{lem}[thm]{Lemma}

\theoremstyle{definition}
\newtheorem{defn}[thm]{Definition}
\newtheorem{exmp}[thm]{Example}
\newtheorem{conj}[thm]{Conjecture}
\newtheorem{cons}[thm]{Construction}
\newtheorem*{convs}{Conventions}
\newtheorem*{conv}{Convention}

\newtheorem*{org}{Organization}
\newtheorem*{ack}{Acknowledgement}

\theoremstyle{remark}
\newtheorem{rem}[thm]{Remark}

\numberwithin{equation}{section}

\newcommand{\beq}{\begin{equation*}\begin{aligned}}
\newcommand{\eeq}{\end{aligned}\end{equation*}}
\newcommand{\bpf}{\begin{proof}}
\newcommand{\epf}{\end{proof}}
\newcommand{\bthm}{\begin{thm}}
\newcommand{\ethm}{\end{thm}}
\newcommand{\bprop}{\begin{prop}}
\newcommand{\eprop}{\end{prop}}
\newcommand{\bcor}{\begin{cor}}
\newcommand{\ecor}{\end{cor}}
\newcommand{\blem}{\begin{lem}}
\newcommand{\elem}{\end{lem}}
\newcommand{\bcons}{\begin{cons}}
\newcommand{\econs}{\end{cons}}
\newcommand{\bdefn}{\begin{defn}}
\newcommand{\edefn}{\end{defn}}
\newcommand{\bexmp}{\begin{exmp}}
\newcommand{\eexmp}{\end{exmp}}
\newcommand{\brem}{\begin{rem}}
\newcommand{\erem}{\end{rem}}
\newcommand{\benu}{\begin{enumerate}[(1)]}
\newcommand{\eenu}{\end{enumerate}}

\newcommand{\bdia}{\begin{displaymath}\xymatrix}
\newcommand{\edia}{\end{displaymath}}

\newcommand{\hfk}{\widehat{HFK}}
\newcommand{\hf}{\widehat{HF}}

\newcommand{\shi}{\underline{\rm SHI}}

\newcommand{\wyk}{(-\widehat{Y},\widehat{K})}

\newcommand{\al}{\alpha}
\newcommand{\be}{\beta}
\newcommand{\ga}{\gamma}
\newcommand{\Ga}{\Gamma}

\newcommand{\ti}{\tilde}
\newcommand{\p}{\prime}
\newcommand{\pp}{{\prime\prime}}

\newcommand{\intg}{\mathbb{Z}}

\newcommand{\ra}{\rightarrow}

\newcommand{\xra}{\xrightarrow}

\begin{document}

\title{Instanton Floer homology, sutures, and Heegaard diagrams}

%    Remove any unused author tags.

%    author one information
\author{Zhenkun Li}
\address{Department of Mathematics, Stanford University}
\curraddr{}
\email{zhenkun@stanford.edu}
\thanks{}

%    author two information
\author{Fan Ye}
\address{Department of Pure Mathematics and Mathematical Statistics, University of Cambridge}
\curraddr{}
\email{fy260@cam.ac.uk}
\thanks{}

\date{}
\dedicatory{}
\maketitle
\begin{abstract}
This paper establishes a new technique that enables us to access some fundamental structural properties of instanton Floer homology. As an application, we establish, for the first time, a relation between the instanton Floer homology of a $3$-manifold or a null-homologous knot inside a $3$-manifold and the Heegaard diagram of that $3$-manifold or knot. We further use this relation to compute the instanton knot homology of some families of $(1,1)$-knots, including all torus knots in $S^3$, which were mostly unknown before. As a second application, we also study the relation between the instanton knot homology $KHI(Y,K)$ and the framed instanton Floer homology $I^\sharp(Y)$. In particular, we prove the inequality $\dim_\mathbb{C} I^\sharp(Y)\le \dim_\mathbb{C}KHI(Y,K)$ for all rationally null-homologous knots $K\subset Y$ and we constructed a new decomposition of the framed instanton Floer homology of Dehn surgeries along $K$ that corresponds to the decomposition along torsion spin${}^c$ decompositions in monopole and Heegaard Floer theory.

\end{abstract}

\tableofcontents%table of contents
%\newpage

%————Start from here————
\section{Introduction}
The instanton homology of closed 3-manifolds and knots in 3-manifolds was introduced by Floer \cite{floer1988instanton,floer1990knot}, which became a powerful tool in the study of 3-dimensional topology. Some related constructions were made by Kronheimer and Mrowka \cite{kronheimer2010knots,kronheimer2011knot,kronheimer2011khovanov}, the first author \cite{li2019direct}, and Daemi and Scaduto \cite{daemi2019equivariant}. Apart from instanton Floer homology, there are three Floer homologies of closed 3-manifolds, knots, and balanced sutured manifolds: Heegaard Floer homology by Ozsv\'ath and Szab\'o \cite{ozsvath2004holomorphic,ozsvath2004holomorphicknot}, Rasmussen \cite{Rasmussen2003}, and Juh\'{a}sz \cite{juhasz2006holomorphic}, monopole Floer homology by Kronheimer and Mrowka \cite{kronheimer2007monopoles,kronheimer2010knots}, and embedded contact homology ($ECH$) by Hutchings \cite{Hutchings2010}, Colin, Ghiggini, Honda, and Hutchings \cite{Colin2010}. For closed 3-manifolds, all these three Floer homologies are isomorphic by work of Kutluhan, Lee, and Taubes \cite{kutluhan2010hf}, or Taubes \cite{taubes2010ech1} combined with Colin, Ghiggini, and Honda \cite{colin+}. However, instanton Floer homology remains isolated from the rest. The following conjecture is still open.
\begin{conj}[{\cite[Conjecture 7.24]{kronheimer2010knots}}]\label{conj: I=HF}
    For a balanced sutured manifold $(M,\ga)$, we have $$SHI(M,\ga)\cong SFH(M,\ga)\otimes \mathbb{C}.$$In particular, for a knot $K$ in a closed 3-manifold $Y$, there are isomorphisms
    $$I^{\sharp}(Y)\cong \widehat{HF}(Y)\otimes\mathbb{C}~{\rm and}~KHI(Y,K)\cong \widehat{HFK}(Y,K)\otimes \mathbb{C}.$$
Here $SHI$ is sutured instanton Floer homology \cite{kronheimer2010knots}, $SFH$ is sutured (Heegaard) Floer homology \cite{juhasz2006holomorphic}, $I^{\sharp}$ is framed instanton Floer homology \cite{kronheimer2011khovanov}, $\widehat{HF}$ is the hat version of Heegaard Floer homology \cite{ozsvath2004holomorphic}, $KHI$ is instanton knot homology \cite{kronheimer2010knots}, and $\widehat{HFK}$ is (Heegaard) knot Floer homology \cite{ozsvath2004holomorphicknot,Rasmussen2003}.
\end{conj}

Instanton Floer homology is closely related to the representations of the fundamental groups and many other topological properties of $3$-manifolds and knots. For example, it is the essential ingredient in proving the property P conjecture \cite{kronheimer2004witten} and the fact that Khovanov homology detects the unknot \cite{kronheimer2011khovanov}. Despite those remarkable applications, many fundamental structural properties of instanton Floer homology remain unknown. We propose a few of them here:
\benu
\item Instanton Floer homology serves as a topological invariant for $3$-manifolds and knots. On the other hand, Heegaard diagrams are one of the most important ways to describes $3$-manifolds and knots and are the basis for Heegaard Floer homology as well. So is it possible to relate instanton Floer homology with the Heegaard diagrams of $3$-manifolds and knots?
\item The monopole and Heegaard Floer homology of closed $3$-manifolds both decompose along spin${}^c$ structures. Non-torsion spin${}^c$ structures have their correspondence in instanton theory by looking at the simultaneous generalized eigenspace decompositions. However, the simultaneous generalized eigenspace decomposition of instanton Floer homology does not distinguish the torsions of the homology group of $3$-manifolds and hence is trivial for all rational homology spheres. So is it possible to obtain a decomposition of instanton Floer homology corresponding to the torsions? This new decomposition would also be the prerequisite for the fourth problem.
\item Can we understand the Euler characteristic of instanton Floer homology? Can we relate it to some other topological invariants of the $3$-manifold?
\item Can we relate the instanton Floer homology of knots and $3$-manifolds? In particular, can we derive a surgery formula for the instanton Floer homology of Dehn surgeries along knots?
\eenu

This paper develops a new technique that enables us to access those fundamental questions listed above. In this paper, we mainly deal with the first and the second and provide some answers to the fourth questions. First, towards answering the first question, we establish the following. 

\begin{thm}\label{thm_1: from heegaard diagram to SHI}
Suppose $Y$ is a rational homology sphere, and $K\subset Y$ is a knot. Suppose $(\Sigma,\al,\be,z,w)$ is a doubly-pointed Heegaard diagram of $(Y,K)$. Then there is a balanced sutured handlebody $(H,\gamma)$ constructed from $(\Sigma,\al,\be,z,w)$ (\textit{c.f.} Subsection \ref{subsec: sutured manifold with tangles}), so that the followings hold
$$\dim_{\mathbb{C}}I^{\sharp}(-Y)\leq \dim_{\mathbb{C}}KHI(-Y,K)\leq \dim_{\mathbb{C}}SHI(-H,-\gamma).$$
\end{thm}
\brem
For most arguments in this paper, there are minus signs before the manifold and the suture, which means that we take the reverse orientation. This is because the proofs are based on contact gluing maps for sutured instanton homology (\textit{c.f.} Subsection \ref{subsec: general bypass}).
\erem
The proof of Theorem \ref{thm_1: from heegaard diagram to SHI} makes use of rationally null-homologous tangles in balanced sutured manifolds. In particular, we proved the following proposition.

\bprop\label{prop: tangle inequality}
Suppose $(M,\ga)$ is a balanced sutured manifold and $T$ is a connected vertical tangle in $(M,\ga)$ (\textit{c.f.} Definition \ref{defn: tangle}). Suppose $M_T=M\backslash N(T)$ and $\ga_T=\ga\cup m_T$, where $m_T$ is the meridian of $T$. If $[T]=0\in H_1(M,\partial M;\mathbb{Q})$, then we have\[\dim_\mathbb{C}SHI(-M,-\ga)\le \dim_\mathbb{C}SHI(-M_T,-\ga_T).\]

\eprop
By Proposition \ref{prop: tangle inequality}, we also prove a generalization of the first inequality in Theorem \ref{thm_1: from heegaard diagram to SHI}, which generalizes the result for null-homologous knots by Wang \cite[Proposition 1.18]{wang2020cosmetic}.
\bprop\label{prop: direct result}
Suppose $Y$ is a closed 3-manifold and $K\subset Y$ is a knot such that
$$[K]=0\in H_1(Y;\mathbb{Q}).$$
Then we have
$$\dim_{\mathbb{C}}I^{\sharp}(-Y)\leq \dim_{\mathbb{C}}KHI(-Y,K).$$
\eprop
In Theorem \ref{thm_1: from heegaard diagram to SHI}, we bound the dimensions of $I^\sharp(-Y)$ and $KHI(-Y,K)$ by the dimension of sutured instanton homology $SHI(-H,-\ga)$, which is still difficult to compute in general.  However, in the case where $H$ is a handlebody, an upper bound of $\dim_\mathbb{C}SHI(H,\ga)$ can be calculated via bypass exact triangles (for bypass exact triangle, \textit{c.f.} {\cite[Theorem 1.20]{baldwin2018khovanov}}, and for the algorithm to obtain an upper bound, \textit{c.f.} \cite[Section 4]{li2019decomposition}). In particular, we apply this theorem to $(1,1)$-knots in lens spaces, whose Heegaard diagrams are described in Proposition \ref{oneonepara}, and obtain the following theorem.

\begin{thm}\label{thm_1: KHI of 1,1 knots}
Suppose $Y$ is a lens space, and $K\subset Y$ is a $(1,1)$-knot. Then we have
$$\dim_{\mathbb{C}}KHI(Y,K)\leq {\rm rk}_{\intg}\widehat{HFK}(Y,K).$$
\end{thm}
Prior to the current paper, there are two main approaches to estimate the dimension of $KHI$. The first is via the spectral sequence from Khovanov homology to instanton knot homology established by Kronheimer and Mrowka \cite{kronheimer2011khovanov}. This bound is sharp for all alternating knots and many other knots. However, Khovanov homology is only defined for knots in $S^3$, so we cannot have any information for knots in other $3$-manifolds. The second way is to study a set of explicit generators of the instanton knot homology and its variances for some special families of knots, and the number of generators bounds the dimension of homology. This idea has been exploited by Hedden, Herald, and Kirk \cite{Hedden2014} and Daemi and Scaduto \cite{daemi2019equivariant}. Our Theorem \ref{thm_1: from heegaard diagram to SHI} and Theorem \ref{thm_1: KHI of 1,1 knots} then offers a totally new way to obtain an upper bound of $\dim_\mathbb{C}KHI$, and the following corollary indicates that this bound is sharp for many examples.
\begin{cor}\label{cor_1: L space knots has sharp bounds}
Suppose $K\subset S^3$ is a $(1,1)$-knot that is also a Heegaard Floer L-space knot. Then
$$\dim_{\mathbb{C}}KHI(S^3,K)={\rm rk}_{\intg}\widehat{HFK}(S^3,K).$$
\end{cor}
\begin{proof}Suppose the Alexander polynomial of $K$ is $\Delta_{K}(t)=\sum_{i\in\intg}c_it^i.$ From Ozsv\'ath and Szab\'o \cite[Theorem 1.2]{Ozsvath2005}, we have
$${\rm rk}_{\mathbb{Z}}\widehat{HFK}(S^3,K)=\sum_{i\in\intg}|c_i|.$$
In instanton theory, the main result of Kronheimer and Mrowka \cite{kronheimer2010instanton}, or Lim \cite{Lim2009}, states that the Euler characteristic of the $i$-th grading of $KHI(S^3,K)$ equals $\pm c_i$. As a result, we have
$$\dim_{\mathbb{C}}KHI(S^3,K)\geq \sum_{i\in\intg}|c_i|.$$
Hence Theorem \ref{thm_1: KHI of 1,1 knots} applies and we conclude the desired equality.
\end{proof}
Corollary \ref{cor_1: L space knots has sharp bounds} would provide many examples whose related spectral sequences from Khovanov homology to instanton knot homology have some nontrivial intermediate pages. In particular, for torus knots, previously there were only partial computations of $KHI$ from the related specatral sequences (\textit{c.f.} \cite{Kronheimer2014,Lobb2020}), while Corollary \ref{cor_1: L space knots has sharp bounds} applies to torus knots directly since torus knots admit lens spaces surgeries (\textit{c.f.} Moser \cite{Moser1971}).
\bcor\label{cor_1: torus knots}
For a torus knot $K=T_{(p,q)}$, suppose its Alexander polynomial is $$\Delta_{K}(t)=t^{-\frac{(p-1)(q-1)}{2}}\frac{(t^{pq}-1)(t-1)}{(t^p-1)(t^q-1)}=\sum_{i=-\frac{(p-1)(q-1)}{2}}^{\frac{(p-1)(q-1)}{2}}c_it^i.$$
Then we have
$$\dim_{\mathbb{C}}KHI(S^3,K,i)=|c_i|,$$
where $i$ denotes the Alexander grading of $KHI(S^3,K)$.
\ecor

In general lens spaces, we may obtain lower bounds from graded Euler characteristics of $SHI$ and provide more examples of Conjecture \ref{conj: I=HF}. For example, the Heegaard Floer homology of constrained knots in lens spaces, introduced by the second author \cite{Ye2020}, is determined by their graded Euler characteristics. However, in the manifold other than $S^3$, it is not known if graded Euler characteristics of $\widehat{HFK}$ and $KHI$ are equal. We will deal with the graded Euler characteristics in a forthcoming paper \cite{LY2021}.

A \textbf{(Heegaard) Floer simple knot} is a knot $K$ in a rational homology sphere $Y$ such that \begin{equation}\label{eq: Floer simple}
    {\rm rk}_\mathbb{Z}\widehat{HFK}(Y,K)={\rm rk}_\mathbb{Z}\widehat{HF}(Y)=|H_1(Y;\mathbb{Z})|,
\end{equation}where $|H_1(Y;\mathbb{Z})|$ is the order of the first homology group of $Y$. A \textbf{(Heegaard Floer) L-space} is a rational homology sphere $Y$ satisfies the latter equality in (\ref{eq: Floer simple}). Examples of Floer simple knots include simple knots in lens spaces (\textit{c.f.} Definition \ref{defn: reduced}; see also \cite[Section 2.1]{Rasmussen2007}). Rasmussen and Rasmussen \cite{Rasmussen2017} proved that for Floer simple knots, there is an interval in $\mathbb{Q}\cup \{\infty\}$ so that a Dehn surgery gives an L-space if and only if the surgery slope is in the interval. In the interior of the interval, the dual knots are also Floer simple knots.

A similar result may hold for instanton Floer homology. A knot $K$ in a rational homology sphere $Y$ is called a \textbf{instanton Floer simple knot} if
\begin{equation}\label{eq: instanton Floer simple}\dim_\mathbb{C}KHI(Y,K)=\dim_\mathbb{C} I^\sharp (Y)=|H_1(Y;\mathbb{Z})|.
\end{equation}
An \textbf{instanton L-space} is a rational homology sphere $Y$ satisfies the latter equality in (\ref{eq: instanton Floer simple}).
\bprop\label{prop: simple knots}
Simple knots in lens spaces are instanton Floer simple knots.
\eprop

\bpf
Suppose $K$ is a simple knot in $Y=L(p,q)$. Combining Theorem \ref{thm_1: from heegaard diagram to SHI}, Theorem \ref{thm_1: KHI of 1,1 knots} and a direct calculation of knot Floer homology, we have
\[\dim_\mathbb{C} I^\sharp(Y)\leq \dim _\mathbb{C}KHI(Y,K)\leq {\rm rk}_\mathbb{Z}\widehat{HFK}(Y,K)=p.\]
By Scaduto \cite[Corollary 1.4]{scaduto2015instanton}, we have
\[\dim_\mathbb{C} I^\sharp(Y)\ge |H_1(Y)|=p.\]Hence we conclude the desired equality.
\epf

Inspired by the result about Floer simple knots by Rasmussen and Rasmussen \cite{Rasmussen2017}, we prove the following theorem for simple knots in lens spaces.

\bthm\label{surgery}
Suppose $K$ is a simple knot in a lens space $Y$. Fixing a framing on $\partial Y(K)$ by picking an arbitrary longitude of $K$. Then there exists an integer $N\ge 0$ so that for any $r\in\mathbb{Q}$ with $|r|\ge N$, the manifold obtained by a surgery of slope $r$ along $K$ is an instanton L-space, and the dual knot is also an instanton Floer simple knot.
\ethm
\brem
In general, framed instanton Floer homology is very difficult to compute. Only some special families of 3-manifolds were studied (\textit{c.f.} \cite{Scaduto2018,lidman2020framed,baldwin2020concordance,alfieri2020framed}). Theorem \ref{surgery} provides many new examples whose framed instanton Floer homology can be computed.
\erem

Towards answering the second question regarding the decomposition of instanton Floer homology, we establish the following.
\bthm\label{thm: torsion spin c decomposition}
Suppose $\widehat{Y}$ is a closed 3-manifold, and $\widehat{K}\subset \widehat{Y}$ is a rationally null-homologous knot. Let $\widehat{Y}(\widehat{K})=\widehat{Y}\backslash {\rm int} (N(\widehat{K}))$ be the knot complement and let $S$ be a Seifert surface of $\widehat{K}$. Suppose futher that $\lambda=\partial S\subset\partial \widehat{Y}(\widehat{K})$ is connected and $|\hat{\mu}\cdot \lambda|=q$, where $\hat{\mu}$ is the meridian of $\widehat{K}$ on $\partial \widehat{Y}(\widehat{K})$ and the dot $\cdot$ denotes the algebraic intersection number. Then there is a decomposition associated to $\widehat{K}$ up to isomorphism:
$$I^{\sharp}(\widehat{Y})\cong \bigoplus_{i=0}^{q-1}I^{\sharp}(\widehat{Y},i).$$
\ethm

When $H_1(\widehat{Y})=\intg_q$ and $\widehat{K}$ represents a generator of $H_1(\widehat{Y})$, we can regard the decomposition in Theorem \ref{thm: torsion spin c decomposition} as an analog of the torsion spin$^c$ decompositions$$\widetilde{HM}(\widehat{Y})=\bigoplus_{\mathfrak{s}\in {\rm Spin}^{c}(\hat{Y})} \widetilde{HM}(\widehat{Y},\mathfrak{s})~{\rm and}~\widehat{HF}(\widehat{Y})=\bigoplus_{\mathfrak{s}\in {\rm Spin}^{c}(\hat{Y})}\widehat{HF}(\widehat{Y},\mathfrak{s}).$$Here $\widetilde{HM}$ is the tilde version of monopole Floer homology \cite{Bloom2009}.

% An evidence is the following.

% \bprop\label{prop: spectral sequence}
% Under the hypothesis and the statement of Theorem \ref{thm: torsion spin c decomposition}, there is a spectral sequence $\{E_r,d_r\}_{r\ge 1}$ from $E_1=KHI(\widehat{Y},\widehat{K})$ to $I^\sharp(\widehat{Y})$. Moreover, this spectral sequence respects the decomposition in Theorem \ref{thm: torsion spin c decomposition}, \textit{i.e.}, there is a decomposition$$(E_r,d_r)=\bigoplus_{i=0}^{q-1}(E_{r,i},d_{r,i})$$so that $E_{r,i}$ converges to $I^\sharp(\widehat{Y},i)$.
% \eprop

To provide some evidence, we will prove the following result in a forthcoming paper \cite{LY2021}.
\bprop
Under the hypothesis and the statement of Theorem \ref{thm: torsion spin c decomposition}, there is a well-defined $\intg_2$ grading on $I^{\sharp}(\widehat{Y},i)$. Suppose $\mu\subset\partial \widehat{Y}(\widehat{K})$ is a simple closed curve so that $|\mu\cdot \lambda|=1$ and suppose $Y$ is the manifold obtained by Dehn filling along $\mu$. For any integer $i\in[0,q-1]$, we have
$$\chi(I^{\sharp}(\widehat{Y},i))=\chi(I^{\sharp}(Y)).$$In particular, if $\widehat{Y}$ is an instanton L-space and $Y=S^3$, then for any integer $i\in[0,q-1]$, we have
$$I^{\sharp}(\widehat{Y},i)\cong\mathbb{C}.$$
\eprop

\begin{rem}\label{rem: defect}
The defect of the decomposition of $I^{\sharp}(\widehat{Y})$ in Theorem \ref{thm: torsion spin c decomposition} is that currently, we do not know if it is independent of the choice of $\widehat{K}$ inside $\widehat{Y}$.
\end{rem}
For integral surgeries on a null-homologous knot, we obtain more results inspired by the large surgery formula in Heegaard Floer homology (\textit{c.f.} \cite[Theorem 4.1]{ozsvath2004holomorphicknot}). For a null-homologous knot $K$ in a closed 3-manifold $Y$, let the basis of $\partial Y(K)$ be formed by the meridian of $K$ and the boundary of a Seifert surface.
\bprop\label{prop: all but 2g-1 summands are trivial}
Suppose $Y$ is a closed 3-manifold and $K\subset Y$ is a null-homologous knot. Suppose $n$ is an integer satisfying $n\geq 2g(K)+1$ and suppose $\widehat{Y}_n$ is obtained from $Y$ by performing the $n$ surgery along $K$. For any integer $i\in[0,n-2g(K)-1]\cup\{n-1\}$, we have
$$I^{\sharp}(\widehat{Y}_n,i)\cong I^{\sharp}(Y).$$For any integer $i\in[0,n-1]$, we have
$$I^{\sharp}(\widehat{Y}_n,i)\cong I^{\sharp}(\widehat{Y}_{n+1},i+1).$$Thus, we have $$\dim_\mathbb{C}I^{\sharp}(\widehat{Y}_{n+1})-\dim I^{\sharp}(\widehat{Y}_{n})=\dim_\mathbb{C}I^\sharp(Y).$$
\eprop
% \bprop\label{prop: stabilization for large surgeries}
% Suppose $Y$ is a closed 3-manifold and $K\subset Y$ is a null-homologous knot. Suppose $\widehat{Y}_q$ is obtained from $Y$ by performing the $q/1$ surgery along $K$. If $q\geq 2g(K)+1$, then 
% \eprop
% \begin{rem}
% Proposition \ref{prop: all but 2g-1 summands are trivial} states that for large integral surgeries, all but $(2g-1)$ summands are isomorphic to the $I^{\sharp}(Y)$. Proposition \ref{prop: stabilization for large surgeries} then states that all those summands will stabilize when performing large integral surgeries with increasing slopes.
% \end{rem}
The proofs of Theorem \ref{thm: torsion spin c decomposition} and Proposition \ref{prop: all but 2g-1 summands are trivial} make use of $SHI(\widehat{Y}(\widehat{K}),\ga)$ for some special suture $\ga\subset \partial \widehat{Y}(\widehat{K})$. In \cite[Section 3]{li2019direct}, the first author constructed a grading, \textit{i.e.}, a decomposition of $SHI$ associated to a properly embedded surface with some admissible conditions. The Seifert surface plays role of this properly embedded surface, which decomposes $SHI(\widehat{Y}(\widehat{K}),\ga)$ for any suture $\ga$. Then we are able to identify some direct summands of the decomposition with $I^\sharp (\widehat{Y})$. Explicitly, we can construct the decomposition by the following proposition.
\bprop\label{prop: I sharp and SHI}
Suppose $Y$ is a closed 3-manifold and $K\subset Y$ is a null-homologous knot. Suppose $\widehat{Y}$ is obtained from $Y$ by performing the $q/p$-surgery along $K$ with $q>0$. Then there is a set $\mathcal{G}$ of sutures on the boundary of the knot complement $Y(K)$, determined by the integer $q$ and the genus $g(K)$ of $K$ so that the followings hold.
\benu
\item For any suture $\ga\in \mathcal{G}$, sutured instanton Floer homology carries a $\intg$ grading, \textit{i.e.},
$$SHI(-Y(K),\ga)=\bigoplus_{i\in\intg}SHI(-Y(K),\ga,i).$$
\item For any suture $\ga\in \mathcal{G}$, there is an integer $i_\ga$ so that there is an isomorphism
$$f_{K,\ga}: I^{\sharp}(-\widehat{Y})\xra{\cong}\bigoplus_{j=i_\ga}^{i_\ga+q-1}SHI(-Y(K),\ga,j).$$
\item For any two sutures $\ga_1,\ga_2\in \mathcal{G}$, there is an isomorphism with respect to gradings (\textit{i.e.} sending the grading $i_{\ga_1}$ to $i_{\ga_2}$ and so on) $$g_{K,\ga_1,\ga_2}: \bigoplus_{j=i_{\ga_1}}^{i_{\ga_1}+q-1}SHI(-Y(K),\ga_1,j)\xra{\cong}\bigoplus_{j=i_{\ga_2}}^{i_{\ga_2}+q-1}SHI(-Y(K),\ga_2,j).$$
\item For any two sutures $\ga_1,\ga_2\in \mathcal{G}$, there is a commutative diagram $$\xymatrix@R=6ex{
\bigoplus_{j=i_{\ga_1}}^{i_{\ga_1}+q-1}SHI(-Y(K),\ga_1,j)\ar[rr]^{g_{K,\ga_1,\ga_2}}&&\bigoplus_{j=i_{\ga_2}}^{i_{\ga_2}+q-1}SHI(-Y(K),\ga_2,j)\\
&I^\sharp(-\widehat{Y})\ar[ul]^{f_{K,\ga_1}}\ar[ur]_{f_{K,\ga_2}}&
}$$
\eenu
\eprop
\brem
Roughly speaking, in Proposition \ref{prop: I sharp and SHI}, the suture $\ga$ in the set $\mathcal{G}$ consists of two parallel copies of simple closed curves with some large slope. The isomorphism $f_{K},\ga$ is constructed from cobordism maps associated to surgeries on curves in the interior of $Y(K)$. The isomorphism $g_{K,\ga_1,\ga_2}$ is constructed from bypass maps. See the end of Subsection \ref{subsec: basic setup for knots} for a sketch of the proof.
\erem
\begin{rem}
Theorem \ref{thm: torsion spin c decomposition} is straightforward from term (2) of Proposition \ref{prop: I sharp and SHI}. However, there are many different choices of the grading $i_{\ga}$. If we regard the decomposition in Theorem \ref{thm: torsion spin c decomposition} as a $\intg_q$-grading on $I^{\sharp}(\widehat{Y})$, then different choices of $i_{\ga}$ lead to shifts of the $\intg_{q}$-grading. So we get a relative $\intg_q$-grading but it is straightforward to upgrade it to a canonical $\intg_q$-grading by fixing the choice of $i_{\ga}$ in term (2) of Proposition \ref{prop: I sharp and SHI}. In Heegaard Floer theory, if $\widehat{Y}$ is obtained from the $q$-surgery on a knot $K$ in an integral homology sphere, then there is a canonical way to identify the set of spin${}^c$ structures on $\widehat{Y}$ with $\intg_q$. In our setup for instanton theory, with some efforts, one could fix a suitable $i_{\ga}$ so that $I^{\sharp}(\widehat{Y},i)$ indeed corresponds to $\widehat{HF}(\widehat{Y},[i])$. Here $\widehat{HF}(\widehat{Y},[i])$ is the hat version of Heegaard Floer homology, of the $3$-manifold $\widehat{Y}$ and the spin${}^c$ structure on $\widehat{Y}$ identified with $[i]\in\intg_q$.
\end{rem}

We would like to make a remark on the developments up to date in answering those questions: In \cite{BLY2020}, Baldwin together with the authors proved a more general inequality that bounds the dimension of instanton Floer homology from above by the number of generators of any Heegaard Floer chain complex of the same manifold. In \cite{LY2021,LY2021enhanced}, based on the techniques developed in the current paper, the authors fully answer the third question and relate the Euler characteristic of instanton Floer homology with the Turaev torsion of $3$-manifolds. In \cite{LY2021surgery}, based on the answer to the second question in this paper, the authors develop a large surgery formula for instanton Floer homology of Dehn surgeries along knots.

\begin{convs}If not mentioned, homology groups and cohomology groups are with $\mathbb{Z}$ coefficients. (Sutured) Heegaard Floer homology is with $\mathbb{Z}$ coefficient, while (sutured) instanton Floer homology is with $\mathbb{C}$ coefficient. We write $\mathbb{Z}_n$ for $\mathbb{Z}/n\mathbb{Z}$.

If it is not mentioned, all manifolds are smooth, connected, and oriented. Suppose $M$ is an oriented manifold. Let $-M$ denote the same manifold with the reverse orientation, called the \textbf{mirror manifold} of $M$. If not mentioned, we do not consider orientations of knots. Suppose $K$ is a knot in a 3-manifold $M$. Then $(-M,K)$ is the \textbf{mirror knot} in the mirror manifold. In $S^3$, the mirror knot is also denoted by $\bar{K}$.

For a manifold $M$, let ${\rm int} (M)$ denote its interior. For a submanifold $A$ in a manifold $Y$, let $N(A)$ denote the tubular neighborhood. The knot complement of $K$ in $Y$ is denoted by $Y(K)=Y\backslash {\rm int} (N(K))$.

For a simple closed curve on a surface, we do not distinguish between its homology class and itself. The algebraic intersection number of two curves $\al$ and $\be$ on a surface is denoted by $\al\cdot\be$, while the number of intersection points between $\al$ and $\be$ is denoted by $|\al\cap \be|$. A basis $(m,l)$ of $H_1(T^2;\mathbb{Z})$ satisfies $m\cdot l=-1$. The \textbf{surgery} means the Dehn surgery and the slope $q/p$ in the basis $(m,l)$ corresponds to the curve $qm+pl$.

A \textbf{rational homology sphere} is a closed 3-manifold whose homology groups with rational coefficients are isomorphic to those of $S^3$. An \textbf{integral homology sphere} is defined similarly. A knot $K\subset Y$ is called \textbf{null-homologous} if it represents the trivial homology class in $H_1(Y;\mathbb{Z})$, while it is called \textbf{rationally null-homologous} if it represents the trivial homology class in $H_1(Y;\mathbb{Q})$.

For $r\in\mathbb{R}$, let $\lceil r\rceil$ and $\lfloor r \rfloor$ be the minimal integer and the maximal integer satisfying $\lceil r\rceil\ge r$ and $\lfloor r \rfloor\le r$, respectively. An argument holds for \textbf{large enough} $n$ if there exists a fixed $N\in\mathbb{Z}$ so that the argument holds for any integer $n>N$. An argument holds for \textbf{small enough} $n$ if there exists a fixed $N\in\mathbb{Z}$ so that the argument holds for any integer $n<N$.
\end{convs}
\begin{org}
The paper is organized as follows.

In Section \ref{sec: preliminaries}, we review some backgrounds, including instanton Floer homology of closed manifolds (Subsection \ref{subsec: instanton Floer homology}), Heegaard Floer homology and instanton Floer homology of balanced sutured manifolds (Subsection \ref{subsec: balanced sutured manifolds}), and bypass attachments on balanced sutured manifolds (Subsection \ref{subsec: general bypass}).

In Section \ref{sec: Instanton theory and Heegaard diagrams}, we construct the sutured handlebody $(-H,-\ga)$ for Theorem \ref{thm_1: from heegaard diagram to SHI} in Subsection \ref{subsec: sutured manifold with tangles}, and prove a generalization of Proposition \ref{prop: tangle inequality} in Subsection \ref{subsec: dimension inequality for tangles}, which leads to the proofs of Theorem \ref{thm_1: from heegaard diagram to SHI} and Proposition \ref{prop: direct result}. Then we deal with $(1,1)$-knots and prove Theorem \ref{thm_1: KHI of 1,1 knots} in Subsection \ref{subsec: $(1,1)$-knots} and Theorem \ref{surgery} in Subsection \ref{subsec: large surgery}.

In Section \ref{sec: Instanton theory and knots}, we state basic setups and sketch the proofs of Proposition \ref{prop: I sharp and SHI} and Theorem \ref{thm: torsion spin c decomposition} in Subsection \ref{subsec: basic setup for knots}, leaving the essential parts to the next two subsections. Then we prove Proposition \ref{prop: all but 2g-1 summands are trivial} in Subsection \ref{subsec: surgeries on knots}.

Finally, in Section \ref{sec: future directions}, we discuss some possible future directions.

\end{org}
\begin{ack}
The first author is partially supported by his Ph.D. advisor Tomasz Mrowka's NSF Grant 1808794. The second author would like to thank his supervisor Jacob Rasmussen for his patient guidance and helpful comments and thank his parents for support and constant encouragement. The authors would also like to thank Ciprian Manolescu for helpful comments on the draft of the paper and thank John A. Baldwin, Qilong Guo, Joshua Wang, Yi Xie, and Ruide Fu for helpful conversations. The second author is also grateful to Yi Liu for inviting him to BICMR, Peking University, and Zekun Chen, Honghuai Fang, and Longke Tang for the company when he was writing this paper.
\end{ack}

% \newpage

\section{Instanton Floer homology and balanced sutured manifolds}\label{sec: preliminaries}
\subsection{Instanton Floer homology}\label{subsec: instanton Floer homology}
\bdefn\label{defn_2: admissible pair}
Suppose $Y$ is a closed 3-manifold and $\omega$ is a closed 1-submanifold in $Y$. Suppose further that there is a closed oriented surface $\Sigma\subset Y$ of genus at least one such that the algebraic intersection number $\omega\cdot\Sigma$ is odd. Then the pair $(Y,\omega)$ is called an \textbf{admissible pair}.
\edefn

For admissible pairs, Floer constructed a homology group from $SO(3)$ connections.
\bthm[\cite{floer1990knot}]
Suppose $(Y,\omega)$ is an admissible pair. Then there is a finite-dimensional complex vector space $I^{\omega}(Y)$ called the \textbf{instanton Floer homology} of $(Y,\omega)$.

Suppose $(Y,\omega)$ and $(Y',\omega')$ are two admissible pairs. Suppose $W$ is a cobordism from $Y$ to $Y'$, \textit{i.e.} $ \partial W=-Y\sqcup Y^\p$, and suppose $\nu\subset W$ is a 2-submanifold with $\partial \nu=(-\omega)\sqcup  \omega'$.Then there exists a complex-linear map
$$I(W,\nu):I^{\omega}(Y)\ra I^{\omega'}(Y'),$$called the \textbf{cobordism map} associated to $(W,\nu)$.
\ethm

\brem
For a fixed $3$-manifold $Y$, $I^{\omega}(Y)$ only depends on the class of $\omega$ in $H_1(Y;\intg_2)$.
\erem
For an admissible pair $(Y,\omega)$, any homology class $\al\in H_*(Y)$ induces a complex-linear action on the instanton Floer homology:
$$\mu(\al):I^{\omega}(Y)\ra I^{\omega}(Y).$$
For any two homology classes $\al_1,\al_2\in H_*(Y)$, we have
$$\mu(\al_1+\al_2)=\mu(\al_1)+\mu(\al_2)~{\rm and}~\mu(\al_1)\mu(\al_2)=(-1)^{{\rm deg}(\al_1){\rm deg}(\al_2)}\mu(\al_2)\mu(\al_1).$$
If $b_2(Y)> 0$, we can pick a basis $\be_1,\dots,\be_n$ of $H_2(Y;\mathbb{Q})$ and consider the simultaneous generalized eigenspaces of all the actions $\mu(\be_1),\dots,\mu(\be_n)$. The simultaneous eigenvalues, as a tuple $(\lambda_1,\dots,\lambda_n)$, can be viewed as a linear map from $H_2(Y;\mathbb{Q})$ to $\mathbb{Q}$. This linear map is the analog of the evaluation of the first Chern classes of spin${}^c$ structures in Heegaard Floer homology. We have the following definition.

\bdefn[{\cite[Definition 7.3]{kronheimer2010knots}}]\label{defn_2: top eigenspaces}
Suppose $(Y,\omega)$ is an admissible pair, $R$ is a closed surface of genus at least one, and $\omega\cdot R$ is odd. Let $I^{\omega}(Y|R)$ be the $(2g(R)-2,2)$-generalized eigenspaces of the pair of actions $((\mu(R),\mu({\rm pt}))$ on $I^{\omega}(Y)$, where ${\rm pt}$ is any fixed basepoint on $Y$.
\edefn
However, if $\al=0\in H_2(Y;\mathbb{Q})$, then $\mu(\al)=0$. Hence for a rational homology sphere $Y$, all $\mu$-actions associated to second homology classes are trivial, and we cannot obtain an effective decomposition from $\mu$.
% For an admissible pair $(Y,\omega)$, there is a canonical $\intg_2$ grading on $I^{\omega}(Y)$ as follows.

% \bthm[\textit{c.f.} {\cite[Section 2.6]{kronheimer2010instanton}} or {\cite[Section 4.6]{scaduto2015instanton}}.]\label{thm_5: mod 2 grading in instanton theory}
% For any admissible pair $(Y,\omega)$, there is a canonical $\intg_2$ grading on $I^{\omega}(Y)$. Furthermore, if $(Y',\omega')$ is another admissible pair, and $(W,\nu)$ is a cobordism from $(Y,\omega)$ to $(Y',\omega')$, then the cobordism map
%  $$I(W,\nu):I^{\omega}(Y)\ra I^{\omega'}(Y')$$
%  is homogeneous, and its degree is fully determined by the diffeomorphism types of $W$, $Y$, and $Y'$.
%  \ethm

% The mod 2 degrees of the maps involved in an exact triangle were described explicitly by Kronheimer and Mrowka in \cite[Section 42.3]{kronheimer2007monopoles}. For the convenience of later usage, we present the discussion from \cite{kronheimer2007monopoles} here.

% Suppose $M$ is a compact oriented 3-manifold with torus boundary. Suppose $\omega\subset M$ is a closed 1-submanifold so that there exists a closed surface $\Sigma$ of genus at least one with $\omega\cdot \Sigma$ odd. Let $\delta$ be a nonzero homology class in the 1-dimensional space ${\rm ker}(i_*)$, where $i:\partial M\ra M$ is the inclusion, and

Suppose $M$ is a compact 3-manifold with torus boundary. Suppose $\omega\subset M$ is a closed 1-submanifold such that there exists a closed surface $\Sigma$ of genus at least one with $\omega\cdot \Sigma$ odd. Let $i:\partial M\ra M$ be the inclusion, and let
\begin{equation}\label{eq_2: induced map by inclusion}
i_*:H_1(\partial M;\mathbb{Q})\ra H_1(M;\mathbb{Q}).
\end{equation}
be the induced map on homology. Let $\ga_1,\ga_2,\ga_3$ be three simple closed curves on $\partial M$ with
$$\ga_1\cdot \ga_2=\ga_2\cdot \ga_3=\ga_3\cdot\ga_1=-1.$$
For $i\in\{1,2,3\}$, let $Y_i$ be the closed 3-manifold obtained by Dehn filling along $\ga_i$:
$$Y_i=M\mathop{\cup}_{\ga_i=\{1\}\times\partial D^2} S^1\times D^2.$$
Then clearly for $i\in\{1,2,3\}$, $(Y_i,\omega)$ are all admissible pairs. Floer proved the following theorem.
\bthm[\cite{floer1990knot}]\label{thm_2: Floer's exact triangle}
There is an exact triangle
\begin{equation}\label{eq_2: Floer's triangle}
    \xymatrix@R=6ex{
    I^{\omega}(Y_1)\ar[rr]^{f_1}&&I^{\omega}(Y_2)\ar[dl]^{f_2}\\
    &I^{\omega}(Y_3)\ar[ul]^{f_3}&
    }
\end{equation}
Furthermore, all maps in the exact triangle (\ref{eq_2: Floer's triangle}) are induced by cobordism maps.
\ethm
\brem
In original construction of Floer \cite{floer1990knot} or Scaduto \cite[Section 2]{scaduto2015instanton}, one has to add some extra component to $\omega$ in one of $Y_1$, $Y_2$, and $Y_3$ to make the exact triangle hold. However, from \cite[Section 2.2]{baldwin2020concordance}, Baldwin and Sivek showed that one could wisely choose some other 1-submanifold $\omega'$ to start with. After adding the extra component coming from the original exact triangle, we finally arrive at a 1-submanifold representing the same homology class as $\omega$ in $H_1(Y;\intg_2)$ for all three 3-manifolds.
\erem

\subsection{Balanced sutured manifolds}\label{subsec: balanced sutured manifolds}
\bdefn[{\cite[Definition 2.2]{juhasz2006holomorphic}}]\label{defn_2: balanced sutured manifold}
A \textbf{balanced sutured manifold} $(M,\ga)$ consists of a compact 3-manifold $M$ with non-empty boundary together with a closed 1-submanifold $\ga$ on $\partial{M}$. Let $A(\ga)=[-1,1]\times\ga$ be an annular neighborhood of $\ga\subset \partial{M}$ and let $R(\ga)=\partial{M}\backslash{\rm int}(A(\ga))$, such that they satisfy the following properties.
\begin{enumerate}[(1)]
    \item Neither $M$ nor $R(\ga)$ has a closed component.
    \item If $\partial{A(\ga)}=-\partial{R(\ga)}$ is oriented in the same way as $\ga$, then we require this orientation of $\partial{R(\ga)}$ induces the orientation on $R(\ga)$, which is called the \textbf{canonical orientation}.
    \item Let $R_+(\ga)$ be the part of $R(\ga)$ for which the canonical orientation coincides with the induced orientation on $\partial{M}$ from $M$, and let $R_-(\ga)=R(\ga)\backslash R_+(\ga)$. We require that $\chi(R_+(\ga))=\chi(R_-(\ga))$. If $\ga$ is clear in the contents, we simply write $R_\pm=R_\pm(\ga)$, respectively.
\end{enumerate}

\edefn
For a balanced sutured manifold $(M,\ga)$, Juh\'{a}sz constructed sutured (Heegaard) Floer homology, and Kronheimer and Mrowka constructed sutured instanton Floer homology.

\bdefn[{\cite[Definition 2.11]{juhasz2006holomorphic}}]\label{defn: balanced diagram}
A \textbf{balanced diagram} is a triple $(\Sigma,\al,\be)$ scuh that the followings hold.
\begin{enumerate}[(1)]
    \item $\Sigma$ is a compact surface with boundary.
    \item $\al = \{ \al_1, \dots, \al_n \}$ and $\be = \{ \be_1,\dots, \be_n \}$ are two sets of pair-wise disjoint simple closed curves in the interior of $\Sigma$. We do not distinguish between the set and the union of curves.
    \item The maps $\pi_0(\partial \Sigma)\to \pi_0(\Sigma\backslash \al)$ and $\pi_0(\partial \Sigma)\to \pi_0(\Sigma\backslash\be)$ are surjective.
\end{enumerate}

Let $M$ be the 3-manifold obtained from $\Sigma\times [-1,1]$ by attaching 3–dimensional 2–handles along $\al_i \times \{-1\}$ and $\be_i \times \{1\}$ for each integer $i\in[1,n]$ and let $\ga=\partial \Sigma\times \{0\}$. A balanced diagram $(\Sigma,\al,\be)$ is called \textbf{compatible} with a balanced sutured manifold $(N,\nu)$ if the sutured manifold $(M,\ga)$ is diffeomorphic to $(N,\nu)$. A sutured manifold $(N,\nu)$ is called a \textbf{product sutured manifold} if it is compatible with $(\Sigma,\emptyset,\emptyset)$ for some $\Sigma$.
\edefn
\bthm[\cite{juhasz2006holomorphic}]\label{thm: SFH}
Suppose $(M,\ga)$ is a balanced sutured manifold. Then there is a balanced diagram compatible with $(M,\ga)$. We can construct a $\mathbb{Z}$-module $SFH(M,\ga)$ from compatible balanced diagrams, which is independent of the choices of balanced diagrams and called the \textbf{sutured (Heegaard) Floer homology} of $(M,\ga)$.
\ethm
\brem\label{rem: SFH}
Sutured Heegaard Floer homology generalizes Heegaard Floer homology \cite{ozsvath2004holomorphic} and knot Floer homology \cite{ozsvath2004holomorphicknot,Rasmussen2003}. Suppose $Y$ is a closed 3-manifold and $K\subset Y$ is a knot. Let $Y(1)$ be obtained from $Y$ by removing a 3-ball and let $\delta$ be a simple closed curve on $\partial Y(1)$. Let $\ga$ consist of two meridians of $K$ with opposite orientations. Then there are canonical isomorphisms \[SFH(Y(1),\delta)\cong \hf(Y)~{\rm and}~ SFH(Y(K),\ga)\cong \hfk(Y,K).\]
\erem
\bthm[{\cite[Section 7.4]{kronheimer2010knots}}]\label{thm_2: KM's definition of SHI}
For a balanced sutured manifold $(M,\ga)$, one can associate a triple $(Y,R,\omega)$, called a \textbf{closure} of $(M,\ga)$, such that the followings hold.
\begin{enumerate}[(1)]
    \item $Y$ is a closed 3-manifold such that $M$ is a submanifold of $Y$.
    \item $R\subset Y$ is a closed surface of genus at least one such that $R_{+}(\ga)$ is a submanifold of $R$ and $R\cap {\rm int}(M)=\emptyset$.
    \item  $\omega\subset Y$ is a simple closed curve such that it intersects $R$ transversely at one point and $\omega\cap {\rm int}(M)=\emptyset$.
\end{enumerate}
Moreover, the isomorphism class of $I^{\omega}(Y|R)$ as in Definition \ref{defn_2: top eigenspaces} is independent of the choices of the triple $(Y,R,\omega)$ and is a topological invariant of $(M,\ga)$.
\ethm

\bdefn\label{defn_2: definition of SHI}
For a balanced sutured manifold $(M,\ga)$, the vector space $I^{\omega}(Y|R)$ for a closure $(Y,R,\omega)$ of $(M,\ga)$ is called the \textbf{sutured instanton Floer homology} of $(M,\ga)$. It is also denoted by $SHI(M,\ga)$ to stress the independence of choices of closures as claimed in Theorem \ref{thm_2: KM's definition of SHI}.
\edefn

One important property of these two sutured Floer homologies is that they detect the tautness of balanced sutured manifolds.
\bdefn[{\cite[Definition 2.6]{juhasz2006holomorphic}}]
A balanced sutured manifold $(M,\ga)$ is called \textbf{taut} if $M$ is irreducible and $R(\ga)$ is incompressible and Thurston norm-minimizing in $[R(\ga)]\in H_2(M,\ga).$
\edefn
\bthm[{\cite[Theorem 1.4]{juhasz2008floer} for $SFH$ and  \cite[Theorem 7.12]{kronheimer2010knots} for $SHI$}]\label{thm_2: SHI detects tautness}
Suppose $(M,\ga)$ is a balanced sutured manifold with $M$ irreducible. Then the followings are equivalent.
\begin{itemize}
    \item $(M,\ga)$ is taut.
    \item $SFH(M,\ga)\neq0$.
    \item $SHI(M,\ga)\neq0$.
\end{itemize}
\ethm
Another important property is about the product manifold.
\bthm[{\cite[Corollary 9.6]{juhasz2008floer} for $SFH$ and \cite[Theorem 7.18]{kronheimer2010knots} for $SHI$, both are based on \cite[Theorem 1.1]{ni2007knot}}]\label{product}
Suppose $(M,\ga)$ is a balanced sutured manifold and a homology product (\textit{c.f.} \cite[Definition 9.1]{juhasz2008floer}). Then the followings are equivalent.
\begin{itemize}
\item $(M,\ga)$ is a product sutured manifold (\textit{c.f.} Definition \ref{defn: balanced diagram}).
\item $SFH(M,\ga)\cong \mathbb{Z}$.
\item $SHI(M,\ga)\cong\mathbb{C}$.
\end{itemize}
\ethm

In Theorem \ref{thm_2: KM's definition of SHI}, only the isomorphism class of $SHI$ is well-defined. Later, Baldwin and Sivek improved the naturality of $SHI$, making it possible to discuss elements in $SHI$. Similar work is done by Juh\'{a}sz, Thurston and Zemke \cite{Juhasz2012} for $SFH$ over $\mathbb{Z}_2$, and Kutluhan, Sivek, and Taubes \cite{Kutluhan2013} for sutured $ECH$.

\bthm[{\cite[Section 9]{baldwin2015naturality}}]\label{thm_2: naturality}
For a balanced sutured manifold $(M,\ga)$ and any two closures $(Y_1,R_1,\omega_1)$ and $(Y_2,R_2,\omega_2)$ of $(M,\ga)$, there is an isomorphism
$$\Phi_{1,2}:I^{\omega_1}(Y_1|R_1)\xra{\cong} I^{\omega_2}(Y_2|R_2),$$
which is well-defined up to multiplication by a unit, \textit{i.e.}, a nonzero complex number. Furthermore, the isomorphism $\Phi$ satisfies the following two conditions.
\begin{enumerate}[(1)]
    \item If $(Y_1,R_1,\omega_1)=(Y_2,R_2,\omega_2)$, then $$\Phi_{1,2}\doteq {\rm id},$$where $\doteq$ means equal up to multiplication by a unit.
    \item If there is a third closure $(Y_3,R_3,\omega_3)$, then we have
$$\Phi_{1,3}\doteq\Phi_{2,3}\circ\Phi_{1,2}:I^{\omega_1}(Y_1|R_1)\ra I^{\omega_3}(Y_3|R_3).$$
\end{enumerate}
\ethm

From Theorem \ref{thm_2: naturality}, for a balanced sutured manifold $(M,\ga)$, Baldwin and Sivek \cite[Section 9]{baldwin2015naturality} constructed a projective transitive system based on the vector spaces $I^{\omega}(Y|R)$ coming from different closures of $(M,\ga)$ and the canonical maps $\Phi$ between them. This projective transitive system is denoted by $$\shi(M,\ga).$$We can regard it as a complex vector space well-defined up to multiplication by a unit. From now on, we will write $\shi(M,\ga)$ for the sutured instanton Floer homology of $(M,\ga)$.

\bdefn[{\cite[Section 7.6]{kronheimer2010knots}}]\label{defn_2: framed instanton Floer homology}
Suppose $Y$ is a closed 3-manifold and $K\subset Y$ is a knot. Let $(Y(1),\delta)$ and $(Y(K),\ga)$ be balanced sutured manifolds defined as in Remark \ref{rem: SFH}. The \textbf{framed instanton Floer homology} of $Y$ is defined by
$$I^{\sharp}(Y) \colonequals I^{S^1}(Y\sharp (S^1\times T^2)|\{1\}\times T^2).$$
It is isomorphic to $\shi(Y(1),\delta)$ (\textit{c.f.} \cite[Section 7.4]{kronheimer2010knots}), so we do not distinguish them.
The \textbf{instanton knot homology} of $(Y,K)$ is defined by
\[KHI(Y,K) \colonequals\shi(Y(K),\ga).\]
\edefn
\brem
In \cite{baldwin2015naturality}, in order to make the definition of $KHI$ independent of different choices of knot complements and the position of the meridional suture, Baldwin and Sivek also added a basepoint to the data. However, since in this paper we only care about the dimension of $KHI$, we overlook this ambiguity and omit the basepoint from our notation. Also, the definition of $\shi(Y(1),\delta)$ depends on a choice of basepoint. For the same reason, we omit the basepoint.
\erem
The surgery exact triangle in Theorem \ref{thm_2: Floer's exact triangle} can easily be generalized for $KHI$. Suppose $M$ is a compact 3-manifold with torus boundary and $K\subset M$ is a knot. Let $\ga_1,\ga_2,\ga_3$ be three simple closed curves on $\partial M$ with
$$\ga_1\cdot \ga_2=\ga_2\cdot \ga_3=\ga_3\cdot\ga_1=-1.$$
For $i\in\{1,2,3\}$, let $Y_i$ be a closed 3-manifold obtained by Dehn filling along $\ga_i$ and let $K_i$ be the knot induced by $K$:
$$(Y_i,K_i)=(M,K)\mathop{\cup}_{\ga_i=\{1\}\times\partial D^2} S^1\times D^2.$$
\bthm\label{thm: Scaduto's exact triangle}
There is an exact triangle
\begin{equation}\label{eq: Scaduto's triangle}
    \xymatrix@R=6ex{
    KHI(Y_1,K_1)\ar[rr]^{f_1}&&KHI(Y_2,K_2)\ar[dl]^{f_2}\\
    &KHI(Y_3,K_3)\ar[ul]^{f_3}&
    }
\end{equation}
Furthermore, all maps in the exact triangle (\ref{eq: Scaduto's triangle}) are induced by cobordism maps between corresponding closures of balanced sutured manifolds induced by $(Y_i,K_i)$.
\ethm
Suppose $(M,\ga)$ is a balanced sutured manifold and $S\subset M$ is a properly embedded surface. We state results about the decomposition of $\shi(M,\ga)$ associated to $S$.
\bdefn[{\cite[Definition 2.25]{li2019decomposition}}]\label{defn_2: admissible surfaces}
Suppose $(M,\ga)$ is a balanced sutured manifold and $S\subset (M,\ga)$ is a properly embedded surface in $M$. The surface $S$ is called an \textbf{admissible surface} if the followings hold.
\begin{enumerate}[(1)]
    \item Every boundary component of $S$ intersects $\ga$ transversely and nontrivially.
    \item We require that $\frac{1}{2}|S\cap \ga|-\chi(S)$ is an even integer.
\end{enumerate}
\edefn
For an admissible surface $S\subset(M,\ga)$, there is a well-defined $\mathbb{Z}$ grading on $\shi(M,\ga)$.

\bthm[\cite{li2019direct}]\label{thm_2: grading on SHI}
Suppose $(M,\ga)$ is a balanced sutured manifold and $S\subset (M,\ga)$ is an admissible surface with $n=\frac{1}{2}|S\cap \ga|$. Then there exists a closure $(Y,R,\omega)$ of $(M,\ga)$ so that $S$ extends to a closed surface $\bar{S}\subset Y$ with $\chi(\bar{S})=\chi(S)-n.$ Let $SHI(M,\ga,S,i)$ denote the $(2i)$-generalized eigenspace of $\mu(\bar{S})$ acting on $SHI(M,\ga)=I^{\omega}(Y|R)$. Then $SHI(M,\ga,S,i)$ is preserved by the canonical maps in Theorem \ref{thm_2: naturality}. Thus, the vector space $$\shi(M,\ga,S,i)$$ is well-defined up to multiplication by a unit. Furthermore, the followings hold.
\begin{enumerate}[(1)]
    \item If $|i|>\frac{1}{2}(n-\chi(S))$, then $\shi(M,\ga,S,i)=0.$
    \item If there is a sutured manifold decomposition $(M,\ga)\stackrel{S}{\leadsto}(M',\ga')$ (\textit{c.f.} \cite[Section 3]{gabai1983foliations} and \cite[Definition 2.7]{juhasz2008floer}), then we have
$$\shi(M,\ga,S,\frac{1}{2}(n-\chi(S)))\cong \shi(M',\ga').$$
    \item For any $i\in\intg$, we have
$$\shi(M,\ga,S,i)=\shi(M,\ga,-S,-i).$$
\end{enumerate}
\ethm
\brem
In \cite{li2019direct}, the grading was only constructed for an admissible surface with a connected boundary. When generalizing it to admissible surfaces with multiple boundary components, more choices arise in the construction of the grading. This new ambiguity was reduced to a combinatorial problem as discussed in \cite[Section 3.3]{li2019direct} and was then resolved in \cite{kavi2019cutting}.
\erem

\brem
Term (1) of Theorem \ref{thm_2: grading on SHI} comes from the adjunction inequality of instanton Floer homology (\textit{c.f.} \cite[Proposition 7.5]{kronheimer2010knots}). Term (2) of Theorem \ref{thm_2: grading on SHI} is a restatement of \cite[Proposition 7.11]{kronheimer2010knots}.
\erem
Suppose $(M,\ga)$ is a balanced sutured manifold, and $S\subset M$ is a properly embedded surface. If $S$ is not admissible, then we isotop $S$ to make it admissible.

\bdefn\label{defn_2: stabilization of surfaces}
Suppose $(M,\ga)$ is a balanced sutured manifold, and $S$ is a properly embedded surface. A \textbf{stabilization} of $S$ is a surface $S^\p$ obtained from $S$ by isotopy in the following sense. This isotopy creates a new pair of intersection points:
$$\partial S'\cap\ga=(\partial{S}\cap\ga)\cup \{p_+,p_-\}.$$
We require that there are arcs $\al\subset \partial{S'}$ and $\be\subset \ga$, oriented in the same way as $\partial{S'}$ and $\ga$, respectively, and the followings hold.
\begin{enumerate}[(1)]
    \item $\partial{\al}=\partial{\be}=\{p_+,p_-\}$.
    \item $\al$ and $\be$ cobound a disk $D$ with ${\rm int}(D)\cap (\ga\cup \partial{S}')=\emptyset$.
\end{enumerate}
The stabilization is called \textbf{negative} if $\partial{D}$ is the union of $\al$ and $\be$ as an oriented curve. It is called \textbf{positive} if $\partial{D}=(-\al)\cup\be$. See Figure \ref{fig: pm_stabilization_of_surfaces}. We denote by $S^{\pm k}$ the surface obtained from $S$ by performing $k$ positive or negative stabilizations, repsectively.
\begin{figure}[ht]
\centering
\begin{overpic}[width=2.5in]{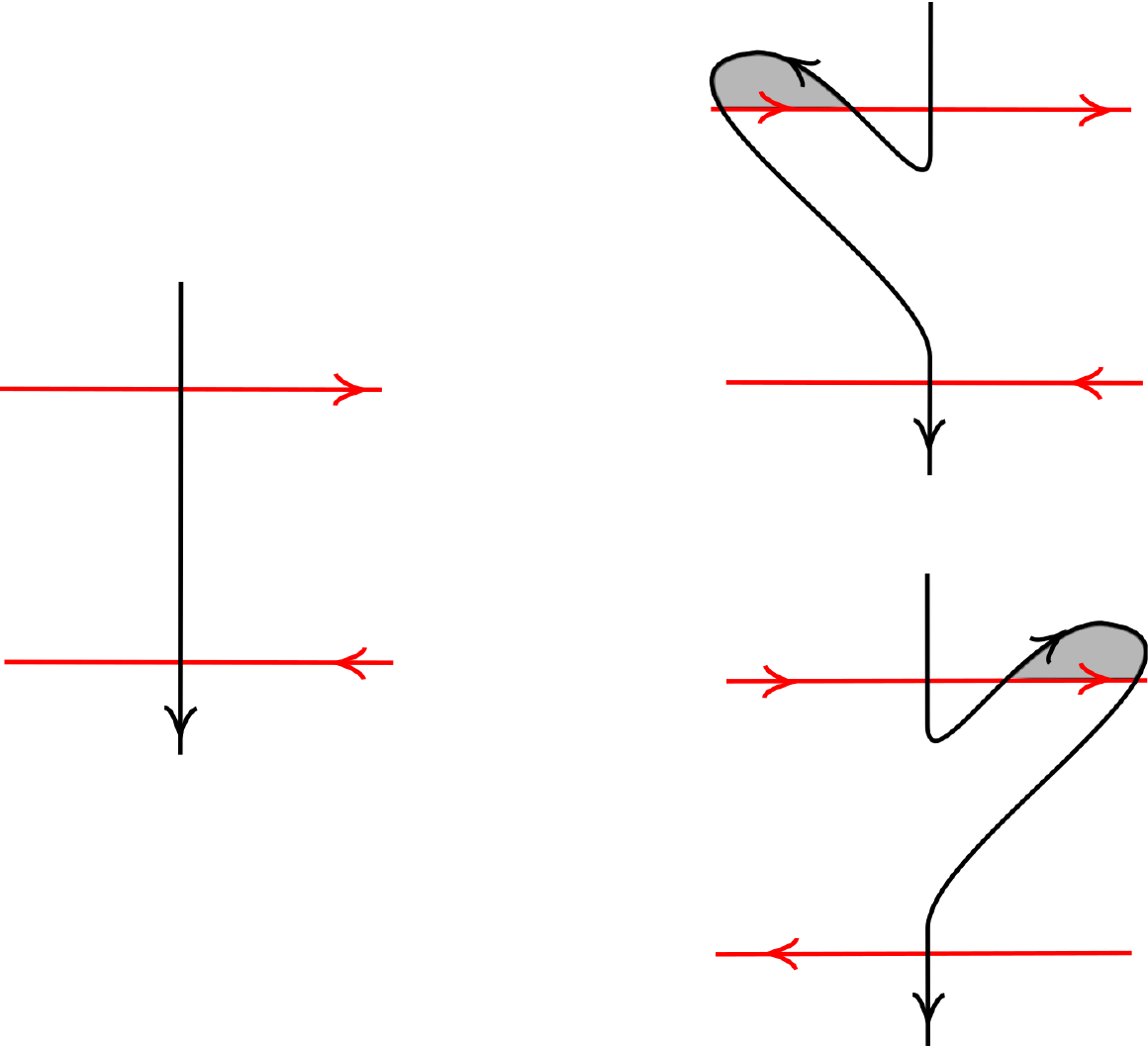}
    \put(13,20){$\partial{S}$}
    \put(-4,57){$\ga$}
    \put(-4,33){$\ga$}
    \put(32,45){\vector(2,1){20}}
    \put(32,45){\vector(2,-1){20}}
    \put(42,21){positive}
    \put(42,66){negative}
    \put(46,83){$D$}
    \put(51,84){\line(1,0){15}}
    \put(63,89){$\al$}
    \put(69,77.5){$\be$}
    \put(93,43){$D$}
    \put(94.5,42){\line(0,-1){7}}
    \put(97,38){$\al$}
    \put(89,28){$\be$}
\end{overpic}
\vspace{0.05in}
\caption{The positive and negative stabilizations of $S$.}\label{fig: pm_stabilization_of_surfaces}
\end{figure}
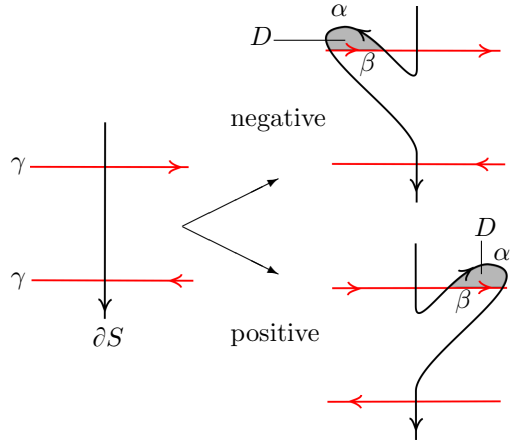
\edefn
The following lemma is straightforward.
\blem\label{lem_2: stabilization and decomposition}
Suppose $(M,\ga)$ is a balanced sutured manifold, and $S$ is a properly embedded surface. Suppose $S^+$ and $S^-$ are obtained from $S$ by performing a positive and a negative stabilization, respectively. Then we have the following.
\begin{enumerate}[(1)]
    \item If we decompose $(M,\ga)$ along $S$ or $S^+$ (\textit{c.f.} \cite[Section 3]{gabai1983foliations} and \cite[Definition 2.7]{juhasz2008floer}), then the resulting two balanced sutured manifolds are diffeomorphic.
    \item If we decompose $(M,\ga)$ along $S^-$, then the resulting balanced sutured manifold $(M',\ga')$ is not taut, as $R_{\pm}(\ga')$ both become compressible.
\end{enumerate}
\elem

\brem\label{rem_2: pm switches according to orientations of the suture}
The definition of stabilizations of a surface depends on the orientations of the suture and the surface. If we reverse the orientation of the suture or the surface, then positive and negative stabilizations switch between each other.
\erem

The following theorem relates the gradings associated to different stabilizations of the same surface.

\bthm[{\cite[Proposition 4.3]{li2019direct}} and {\cite[Proposition 4.17]{wang2020cosmetic}}]\label{thm_2: grading shifting property}
Suppose $(M,\ga)$ is a balanced sutured manifold and $S$ is a properly embedded surface in $M$ that intersects $\ga$ transversely. Suppose all the stabilizations mentioned below are performed on a distinguished boundary component of $S$. Then, for any $p,k,l\in \intg$ such that the stabilized surfaces $S^{p}$ and $S^{p+2k}$ are both admissible, we have
$$\shi(M,\ga,S^{p},l)=\shi(M,\ga,S^{p+2k},l+k).$$
Note that $S^p$ is a stabilization of $S$ as introduced in Definition \ref{defn_2: stabilization of surfaces}, and, in particular, $S^0=S$.
\ethm

\brem
The original form of Theorem \ref{thm_2: grading shifting property} in \cite{li2019direct} was stated for a Seifert surface in the case of a knot complement. However, it is straightforward to generalize the proof to the case of a general admissible surface in a general balanced sutured manifold, given the condition that the decompositions along $S$ and $-S$ are both taut. This extra condition on taut decompositions was then dropped due to the work in \cite{wang2020cosmetic}.
\erem

\subsection{Bypass attachments}\label{subsec: general bypass}
\quad

In this subsection, we review bypass maps for sutured instanton homology.
\bdefn[{\cite[Section 3.4]{honda2000classification}}]\label{defn_2: bypass arc}
Suppose $(M,\ga)$ is a balanced sutured manifold. An arc $\al\subset \partial{M}$ is called a \textbf{bypass arc} if the arc intersects the suture $\ga$ transversely at three points, including two endpoints.

For a bypass arc $\al$, let $P_0$, $P_1$, and $P_2$ be its three intersection points with $\ga$, ordered by any orientation of $\al$. For $i=0,1,2$, let $\ga_{i}$ be the component of $\ga$ containing $P_i$. If $\ga_{0}=\ga_{1}\neq \ga_{2}$ or $\ga_{1}=\ga_{2}\neq \ga_{0}$, then $\al$ is called a \textbf{wave bypass}. If $\ga_{0}=\ga_{2}\neq \ga_{1}$, then $\al$ is called an \textbf{anti-wave bypass}.
\edefn
\brem
The names of wave and anti-wave follow from \cite[Section 7]{Greene2016}, where waves and anti-waves are arcs whose endpoints are on the same curve. For an anti-wave bypass $\al$, after removing the component of $\ga$ that only contains one intersection point, the arc $\al$ becomes a wave or an anti-wave. See Definition \ref{defn: antiwave}.
\erem
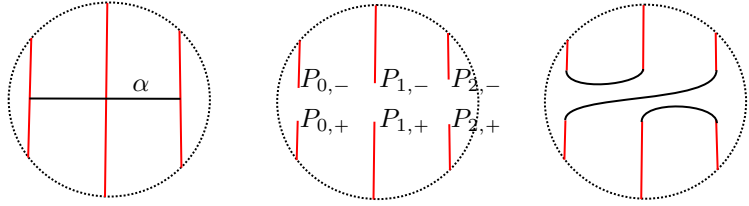
\begin{figure}[ht]
\centering
\begin{overpic}[width=0.7\textwidth]{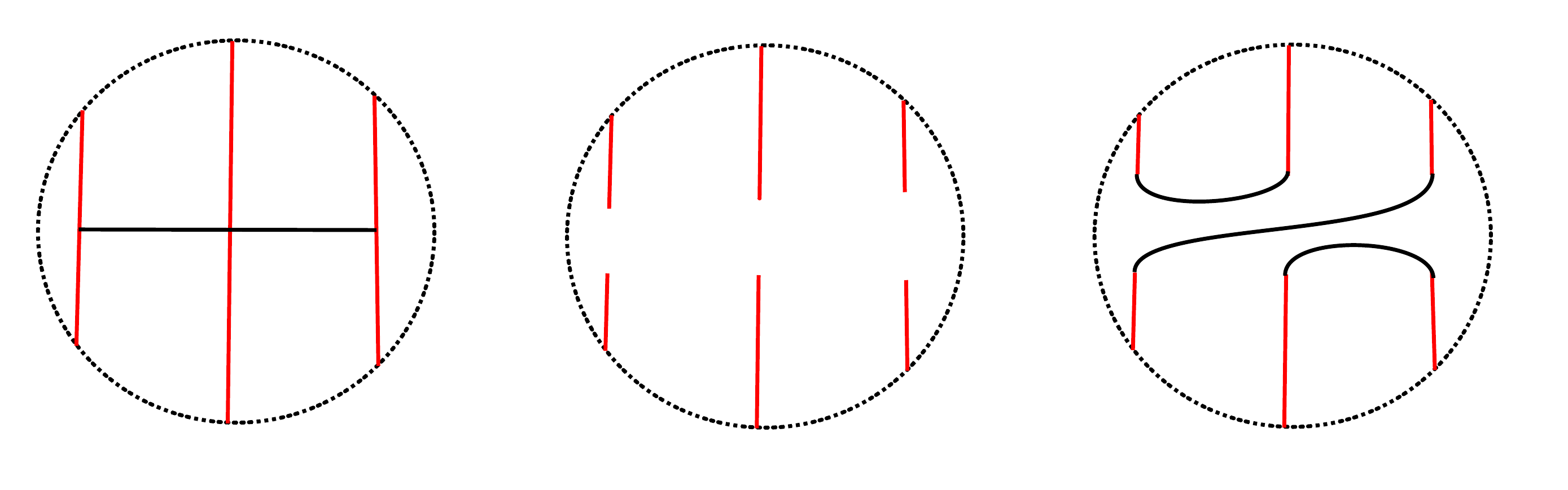}
    \put(18,17.5){$\alpha$}
    \put(39,18){$P_{0,-}$}
    \put(39,13){$P_{0,+}$}
    \put(49,18){$P_{1,-}$}
    \put(49,13){$P_{1,+}$}
    \put(58,18){$P_{2,-}$}
    \put(58,13){$P_{2,+}$}
    %\put(73,28){$\be$}%
\end{overpic}
\caption{The bypass arc and the bypass attachment.}\label{fig: the bypass arc}
\end{figure}

Given a bypass arc $\al$ on a balanced sutured manifold $(M,\ga_1)$, we can change the suture locally as follows. Let $D\subset \partial M$ be a neighborhood of the arc $\al\subset \partial{M}$, which is a disk intersecting $\ga$ in three arcs. There are six endpoints after removing $ \ga\cap D$ from $\ga$, labelled as follows. Suppose $P_{i,-}$ and $P_{i,+}$ are two endpoints corresponding to $P_i$, where the sign is chosen so that the oriented arc $\al$, together with the arc-component of $\ga\cap D$ from $P_{i,-}$ to $P_{i,+}$, gives an oriented framing of $\partial M$. Then we connect these six endpoints by the `left-handed-principle': in $D$, a new suture $\ga_2$ is obtained by connecting $P_{0,-}$ to $P_{1,-}$, connecting $P_{2,-}$ to $P_{0,+}$, and connecting $P_{2,+}$ to $P_{1,+}$. See Figure \ref{fig: the bypass arc} for an example of a bypass arc and the corresponding new suture.
\begin{prop}[{\cite[Section 2.3]{honda2002gluing}}]\label{prop: honda}Suppose $(M,\ga)$ is a balanced sutured manifold and $\al$ is a bypass arc. Suppose $\ga_2$ is the new suture described as above. Then $(M,\ga_2)$ is still a balanced sutured manifold.

If $\al$ is a wave bypass, the suture $\ga_2$ is obtained from $\ga_1$ via a `mystery move' (\textit{c.f.} \cite[Figure 8]{honda2002gluing}). If $\al$ is an anti-wave bypass, the suture $\ga_2$ is obtained from $\ga_1$ via a positive Dehn twist on $\partial M$. In both cases, the numbers of components of $\ga_1$ and $\ga_2$ are the same.
\end{prop}
\bdefn[]\label{defn_2: bypass attachment}
The change from $(M,\ga_1)$ to $(M,\ga_2)$ is called a \textbf{bypass attachment} along $\al$.
\edefn
\brem
The definition of a bypass attachment is due to \cite[Section 3.4]{honda2000classification}. Originally, a bypass attachment is a thickened half-disk attached to a contact 3-manifold $M$ along an arc $\al\subset \partial M$, which carries some particular contact structure. The dividing set on the boundary $\partial M$ can be thought of as equivalent to the suture. After the half-disk-attachment, the dividing set, or the suture, is changed in the way described in Definition \ref{defn_2: bypass attachment}. For our purpose, we do not require the balanced suture manifold $(M,\ga_1)$ to carry a contact structure, while we can still perform an abstract bypass attachment by modifying the suture in a neighborhood of $\al$.
\erem
A bypass attachment induces a map
$$\psi_1:\shi(-M,-\ga_1)\ra\shi(-M,-\ga_2).$$
This map can be explained in the following two ways.

{\bf Contact handle decomposition}. By Ozbagci \cite[Section 3]{ozbagci2011contact}, the half-disk-attachment can be decomposed into two contact handle attachments. First, one can attach a contact 1-handle along two endpoints of the bypass arc $\al$. Then one can attach a contact 2-handle along a circle that is the union of $\al$ and an arc on the attached contact 1-handle. Topologically, the 1-handle and the 2-handle form a canceling pair, so the diffeomorphism type of the 3-manifold does not change. However, the contact structure is changed, and the suture $\ga_1$ is replaced by $\ga_2$. In \cite[Section 5]{baldwin2016instanton}, Baldwin and Sivek constructed contact handle attaching maps for $\shi$. Following Ozbagci's idea, they defined the map
$$\psi_1:\shi(-M,-\ga_1)\ra\shi(-M,-\ga_2)$$
to be the composition of contact handle attaching maps corresponding to the contact 1-handle and the 2-handle attaching.

{\bf Contact gluing maps}. The half-disk attachment can be reinterpreted as follows. We start with $(M,\ga_1)$. Then pick $[1,2]\times\partial M$ to be a collar of the boundary, carrying a particular contact structure $\xi$ specified by the bypass attachment, so that the boundary $\{1,2\}\times \partial M$ is convex and the dividing set is $(-\ga_1)\sqcup\ga_2$, with $\ga_i\subset\{i\}\times\partial M$ for $i\in\{1,2\}$. Then we glue $[1,2]\times \partial M$ to $M$ by the identification $\{1\}\times \partial M=\partial M$. The new 3-manifold is diffeomorphic to $M$, while the suture $\ga_1$ is replaced by $\ga_2$. In \cite{honda2008contact}, Honda, Kazez, and Mati\'c defined a gluing map for $SFH$
$$\Phi_{\xi}:SFH(-M,-\ga_1)\ra SFH(-M,-\ga_2),$$
which was then re-visited by Juh\'asz and Zemke \cite{juhasz1803contact}. Later, the first author \cite{li2018gluing} defined a similar gluing map for $\shi$, and we can define
$$\psi_1=\Phi_{\xi}:\shi(-M,-\ga_1)\ra\shi(-M,-\ga_2).$$

These two points of view are equivalent due to \cite[Section 4]{li2018gluing}. We have some useful corollaries.

\blem\label{lem_2: disjoint bypass commutes}
Suppose $(M,\ga)$ is a balanced sutured manifold and $\al,\be\subset\partial M$ are two bypass arcs with $\al\cap\be=\emptyset$. Let $\psi_{\al}$ and $\psi_{\be}$ be the bypass maps associated to $\al$ and $\be$, respectively. Let $(M,\ga')$ be the resulting balanced sutured manifold after bypass attachments along both $\al$ and $\be$. Then we have
$$\psi_{\al}\circ\psi_{\be}=\psi_{\be}\circ\psi_{\al}:\shi(-M,-\ga)\ra\shi(-M,-\ga').$$
\elem
\bpf
Consider bypasses as the compositions of contact handle attachments. Since $\al\cap\be=\emptyset$, the contact handles associated to $\al$ and $\be$ are attached to disjoint regions on $\partial M$. Then it follows immediately that the corresponding contact handle attaching maps commute with each other.
\epf

\blem\label{lem_2: isotopic bypasses}
Suppose $(M,\ga)$ is a balanced sutured manifold and $\al_0,\al_1\subset\partial M$ are two bypass arcs. Suppose further that these two arcs are isotopic as bypass arcs, \textit{i.e.}, there is a smooth family $\al_t$ of bypass arcs for $t\in[0,1]$. Then $\al_1$ and $\al_2$ lead to isotopic balanced sutured manifold $(M,\ga')$, and the bypass maps $\psi_{\al_1}$ and $\psi_{\al_2}$ are the same:
$$\psi_{\al_1}=\psi_{\al_2}:\shi(-M,-\ga)\ra\shi(-M,-\ga').$$
\elem

\bpf
It follows from the contact handle decomposition interpretation of the bypass attachments.
\epf

\brem
On the level of contact geometry, Honda has already proved Lemma \ref{lem_2: disjoint bypass commutes} and Lemma \ref{lem_2: isotopic bypasses} in \cite{honda2000classification}. Thus, these two lemmas can also be proved by combining Honda's results with the functoriality of gluing maps $\Phi_{\xi}$ in \cite{li2018gluing}.
\erem

An important properties of bypass maps is the bypass exact triangle. Suppose $(M,\ga_1)$ is a balanced sutured manifold, and $\al$ is a bypass arc. Suppose $D$ is a neighborhood of $\al\subset \partial M$ and $(M,\ga_2)$ is obatined by the bypass attachment along $\al$. As shown in Figure \ref{fig: the bypass triangle}, after attaching a bypass along $\al$, there is an obvious bypass arc $\be\subset D$. When we do the bypass attachment along $\be$, we obtain the third balanced sutured manifold $(M,\ga_3)$. It still carries an obvious bypass arc $\theta$. When we further do the bypass attachment along $\theta$, we obtain $(M,\ga_1)$ again. Let $\psi_1$, $\psi_2$, and $\psi_3$ be the bypass maps associated to $\al$, $\be$, and $\theta$, respectively. We have the following theorem.

\begin{figure}[ht]
\centering
\begin{overpic}[width=0.5\textwidth]{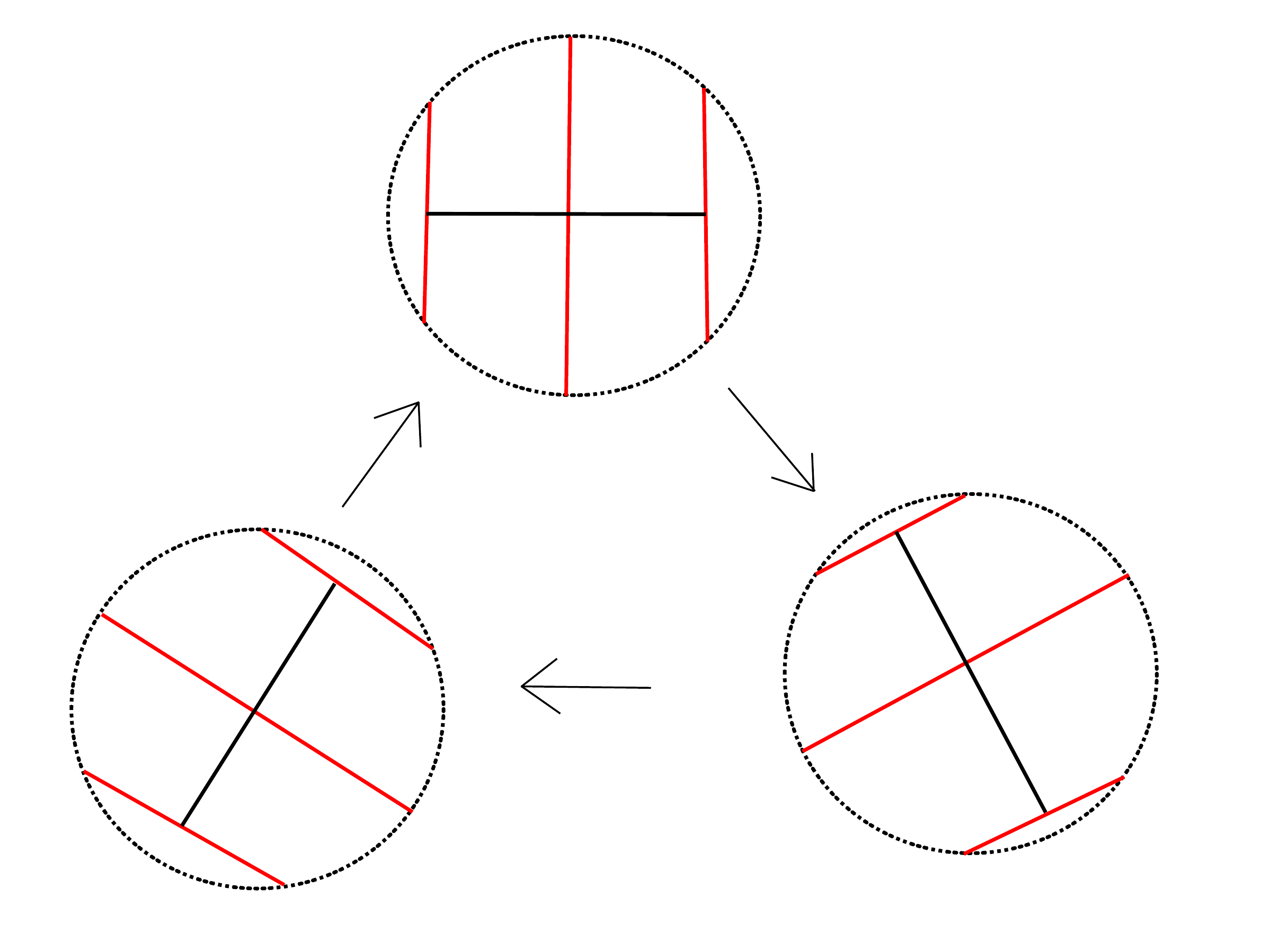}
    \put(18,22){$\theta$}
    \put(38,58){$\al$}
    \put(73,28){$\be$}
\end{overpic}
\vspace{-0.05in}
\caption{The bypass triangle.}\label{fig: the bypass triangle}
\end{figure}

\bthm[{\cite[Theorem 1.20]{baldwin2018khovanov}}]\label{thm_2: bypass exact triangle on general sutured manifold}
There exists an exact triangle
\begin{equation*}
\xymatrix@R=6ex{
\shi(-M,-\ga_1)\ar[rr]^{\psi_1}&&\shi(-M,-\ga_2)\ar[dl]^{\psi_2}\\
&\shi(-M,-\ga_3)\ar[lu]^{\psi_3}&
}
\end{equation*}
\ethm

\section{Instanton Floer homology and Heegaard diagrams}
\label{sec: Instanton theory and Heegaard diagrams}
\subsection{Balanced sutured manifolds with tangles}\label{subsec: sutured manifold with tangles}
\bdefn[{\cite[Definition 1.1]{xie2019tangle}}]\label{defn: tangle}
Suppose $(M,\ga)$ is a balanced sutured manifold. A \textbf{tangle} $T\subset (M,\ga)$ is a properly embedded 1-submanifold such that $T\cap A(\ga)=\emptyset.$ A tangle $T$ is called \textbf{balanced} if $$|T\cap R_+(\ga)|=|T\cap R_-(\ga)|.$$ A component $a$ of $T$ is called \textbf{vertical} if $a$ is an arc from $R_+(\ga)$ to $R_-(\ga)$. A tangle $T $ is called \textbf{vertical} if every component of $T$ is vertical. Note that vertical tangles are balanced.

Suppose $T\subset (M,\ga)$ is a vertical tangle, we construct a new balanced sutured manifold $(M_T,\ga_T)$, where $M_T=M\backslash N(T)$ and $\ga_T$ is the union of $\ga$ and one meridian for each component of $T$.
\edefn

\bthm[{\cite{xie2019tangle}}]
Suppose $(M,\ga)$ is a balanced sutured manifold and suppose $T\subset (M,\ga)$ is a balanced tangle. Then there is a finite-dimensional complex vector space $SHI(M,\ga,T)$, whose isomorphism class is a topological invariant of the triple $(M,\ga,T)$.
\ethm
We have the following theorem.

\bthm[{\cite[Lemma 7.10]{xie2019tangle}}]\label{thm_2: isomorphism between SHI with and without tangles}
For a vertical tangle $T\subset (M,\ga)$, there is an isomorphism
$$SHI(M,\ga,T)\cong SHI(M_T,\ga_T).$$
\ethm

Then we introduce Heegaard diagrams of closed 3-manifolds and knots.
\bdefn
A \textbf{(genus $g$) diagram} is a triple $(\Sigma,\al,\be)$, where
\begin{enumerate}[(1)]
    \item $\Sigma$ is a closed surface of genus $g$;
    \item $\al=\{\al_1,\dots,\al_m\}$ and $\be=\{\be_1,\dots,\be_n\}$ are two sets of pair-wise disjoint simple closed curves on $\Sigma$. We do not distinguish the set and the union of curves.
\end{enumerate}
Let $N_0$ be the manifold obtained from $\Sigma\times [-1,1]$ by attaching 3–dimensional 2–handles along $\al_i \times \{-1\}$ and $\be_j\times \{1\}$ for each integer $i\in[1,m]$ and each integer $j\in[1,n]$. Let $N$ be the manifold obtained from $N_0$ by capping off spherical boundaries. A diagram $(\Sigma,\al,\be)$ is called \textbf{compatible} with a 3-manifold $M$ if $M\cong N$. In such case, we also write $M$ is compatible with $(\Sigma,\al,\be)$, or $(\Sigma,\al,\be)$ is a diagram of $M$.
\edefn
\bdefn\label{defn_2: heegaard diagram}
A \textbf{(genus $g$) Heegaard diagram} is a (genus $g$) diagram $(\Sigma,\al,\be)$ satisfying the following conditions.
\begin{enumerate}[(1)]
    \item $|\al|=|\be|=g$, \textit{i.e.}, there are $g$ curves in either tuple.
    \item $\Sigma\backslash \al$ and $\Sigma\backslash\be$ are connected.
\end{enumerate}
Given a Heegaard diagram $(\Sigma,\al,\be)$, the manifolds compatible with $(\Sigma,\al,\emptyset)$ and $(\Sigma,\emptyset,\be)$ are called the \textbf{$\al$-handlebody} and the \textbf{$\be$-handlebody}, respectively.
\edefn

\bdefn
A \textbf{(genus $g$) doubly-pointed Heegaard diagram} $(\Sigma,\al,\be,z,w)$ is a (genus $g$) Heegaard diagram with two points $z$ and $w$ in $\Sigma\backslash \al\cup\be$. Let $a\subset\Sigma\backslash\al$ and $b\subset\Sigma \backslash\be$ be two arcs connecting $z$ to $w$. Suppose $a^\p$ and $b^\p$ are obtained from $a$ and $b$ by pushing them into $\al$-handlebody and $\be$-handlebody, respectively. A doubly-pointed Heegaard diagram $(\Sigma,\al,\be,z,w)$ is called \textbf{compatible} with a knot $K$ in a closed 3-manifold $Y$ if $(\Sigma,\al,\be)$ is compatible with $Y$ and the union $a^\p\cup b^\p$ is isotopic to $K$.
\edefn
\bdefn
Suppose $(\Sigma,\al,\be)$ is a Heegaard diagram of a closed 3-manifold $Y$. A knot $K\subset Y$ is called the \textbf{core knot} of $\be_i$ for some $\be_i\subset \be$ if it is constructed as follows. Let $M$ be the manifold compatible with the diagram $(\Sigma,\al,\be\backslash \be_i)$. It has a torus boundary and $\be_i$ induces a simple closed curve $\be_i^\p$ on $\partial M$. Dehn filling $M$ along $\be_i^\p\subset \partial M$ gives $Y$. Let $K$ be the image of $S^1\times {\rm 0}\subset S^1\times D^2$ under the filling map, where $S^1\times D^2$ is the filling solid torus.
\edefn
The following is a basic fact in 3-dimensional topology.

\bprop[{\cite[Section 2.2]{ozsvath2004holomorphicknot}}]
For any closed 3-manifold $Y$ and any knot $K\subset Y$, there is a doubly-pointed Heegaard diagram compatible with $(Y,K)$.
\eprop
In the rest of this subsection, we provide the construction of the balanced sutured handlebody $(H,\ga)$ used in Theorem \ref{thm_1: from heegaard diagram to SHI}.
\bcons\label{cons: doubly-pointed}
Suppose $Y$ is a closed 3-manifold and $K\subset Y$ is a knot. Suppose $(\Sigma,\al,\be,z,w)$ is a genus $(g-1)$ doubly-pointed Heegaard diagram compatible with $(Y,K)$. Consider the manifold $M$ obtained from $\Sigma\times [-1,1]$ by attaching a 3-dimensional 1-handle along $\{z,w\}\times \{1\}$. Let $\Sigma^\p$ be the component of $\partial M$ with genus $g$. Let $\al_{g}\subset \Sigma^\p$ be the curve obtained by running from $z$ to $w$ and then back over the 1-handle. Let $\be_{g}\subset \Sigma^\p$ be a small circle around $z$. Set $$\al^\p=\al\times\{1\}\cup \{\al_{g}\}~{\rm and}~ \be^\p=\be\times\{1\}\cup \{\be_{g}\}.$$Then $(\Sigma^\p,\al^\p,\be^\p)$ is a genus $g$ Heegaard diagram compatible with $Y$. Since $\be_{g}$ is a meridian of $K$, the knot $K$ is the core knot of $\be_{g}$.
\econs

\bcons\label{cons: heegaard diagram gives rise to tangles}
Suppose $Y$ is a closed 3-manifold and $(\Sigma^\p,\al^\p,\be^\p=\{\be_1,\dots,\be_g\})$ is a genus $g$ Heegaard diagram compatible with $Y$. Let $Y(1)$ be obtained from $Y$ by removing a 3-ball. The manifold $Y(1)$ can be obtained from the $\al^\p$-handlebody by attaching 3–dimensional 2-handles along $\be_i$ for each integer $i\in[1,g]$. Note that a 3-dimensional 2-handle can be thought of as $[-1,1]\times D^2$ attached along $[-1,1]\times\partial D^2$. Let $\theta_i=[-1,1]\times \{0\}$ be the co-core of the 2-handle attached along $\be_i$. We have a properly embedded tangle in $Y(1)$: $$T=\theta_1\cup\dots\cup\theta_g.$$

Pick a simple closed curve $\delta\subset\partial{Y(1)}$ such that for any $i$, two endpoints of $\theta_i$ lie on two different sides of $\delta$. From the construction, the manifold $Y(1)_{T}=Y(1)\backslash N(T)$ is the $\al^\p$-handlebody and the suture $\delta_T$ consists of all $\be_i$ curves and a curve $\be_{g+1}$ induced by $\delta$, \textit{i.e.}$$\delta_T=\be_1\cup\cdots\cup\be_g\cup \be_{g+1}.$$Hence $R_{+}(\delta_T)$ and $R_-(\delta_T)$ can be obtained from $\Sigma\backslash \be$ by cutting along $\be_g$, which are both spheres with $(g+1)$ punctures.

\econs

\bcons\label{cons: meridian sutures}
Suppose $Y$ is a closed 3-manifold and $K\subset Y$ is a knot. Suppose $(\Sigma,\al,\be=\{\be_1,\dots,\be_{g-1}\},z,w)$ is a genus $(g-1)$ doubly-pointed Heegaard diagram of $(Y,K)$. We apply Construction \ref{cons: doubly-pointed} to obtain a genus $g$ Heegaard diagram $(\Sigma^\p,\al^\p,\be^\p=\{\be_1,\dots,\be_{g}\})$ of $Y$, and then apply Construction \ref{cons: heegaard diagram gives rise to tangles} to obtain a balanced sutured handlebody $$(H,\ga)=(Y(1)_T,\delta_T=\be_1\cup\cdots\cup\be_{g+1}).$$

Note that the diagram $(\Sigma^\p,\al^\p,\be)$ is compatible with the knot complement $Y(K)$. Suppose $\be_g^\pp$ and $\be_{g+1}^\pp$ are curves on $\partial Y(K)$ induced by $\be_g$ and $\be_{g+1}$, respectively. Since $\be_g^\pp\cap \be_{g+1}^\pp=\emptyset$ and $\partial Y(K)\cong T^2$, the curve $\be_{g+1}^\pp$ is parallel to $\be_{g}^\pp$. Since $\be_g^\pp$ is a meridian of $K$ and $(Y(K),\be_g^\pp\cup \be_{g+1}^\pp)$ is a balanced sutured manifold, the curve $\be_{g+1}^\pp$ must be another meridian of $K$ with the orientation opposite to that of $\be_{g}^\pp$.
\econs
We provide an explicit construction of the curve $\be_{g+1}\subset \partial H$ in Construction \ref{cons: meridian sutures}.
\bcons\label{cons: heegaard diagram gives rise to sutured handlebody}
Suppose $(\Sigma^\p,\al^\p,\be^\p=\{\be_1,\dots,\be_g\})$ is a genus $g$ Heegaard diagram compatible with a closed 3-manifold $Y$. Let $H$ be the $\al^\p$-handleboby. For any integer $i\in[1,g]$, let $\be_i$ be oriented arbitrarily and let $\be_i^\p\subset \partial H$ be the curve obtained by pushing off $\be_i$ to the right with respect to the orientation. Suppose $\be_i^\p$ is oriented reversely. Let $\be_{g+1}$ be the curve obtained from $\be_i^\p$ by band sums with respect to orientations so that $\be_{g+1}$ is disjoint from $\be_1,\dots,\be_{g}$. Set
$$\ga=\be_1\cup\dots\cup\be_{g+1}.$$
It is straightforward to check that $(H,\ga)$ is the one obtained in Construction \ref{cons: meridian sutures}.
\econs
We can also obtain the original 3-manifold $Y$ from the sutured handlebody $(H,\ga)$ as follows.

\bcons\label{cons: sutured handlebody gives rise to heegaard diagrams}
Suppose $H$ is a handlebody, and $\ga$ is a suture on $\partial H$ such that $R_{+}(\ga)$ and $R_-(\ga)$ are both spheres with $(g+1)$ punctures. Let $\Sigma=\partial H$. Suppose $\Sigma$ has genus $g$. Let $\al_1$,\dots,$\al_g$ be boundaries of $g$ compressing disks $D_1,\dots, D_g$ so that $H\backslash (D_1\cup\cdots\cup D_g)$ is a 3-ball. Since $R_{+}(\ga)$ and $R_-(\ga)$ are both spheres with $(g+1)$ punctures, the suture $\ga$ has $(g+1)$ components. We can take arbitrary $g$ of them to form $\be$. Then $(\Sigma,\al,\be)$ is a Heegaard diagram. Let $Y$ be a closed 3-manifold compatible with $(\Sigma,\al,\be)$. Since different choices of such $g$ curves from $\ga$ are related to each other by a finite sequence of handle slides, the manifold $Y$ is well-defined up to diffeomorphism.

Let $\delta$ be the remaining component of $\ga$ and let $T$ be the union of co-cones of $\be_i$ curves as in Construction \ref{cons: heegaard diagram gives rise to tangles}. It is straightforward to check that $(Y(1)_T,\delta_T)=(H,\ga)$.
\econs

\subsection{A dimension inequality for tangles}\label{subsec: dimension inequality for tangles}
\quad

In this subsection, we prove a generalization of Proposition \ref{prop: tangle inequality}.
\bprop\label{prop_3: dimension inequality for tangles}
Suppose $(M,\ga)$ is a balanced sutured manifold and $T$ is a vertical tangle. Suppose $a\subset T$ is a component of $T$ so that
$[a]=0\in H_1(M,\partial M;\mathbb{Q}).$
Let $T'=T\backslash a$. Then we have
$${\rm dim}_{\mathbb{C}}SHI(-M,-\ga,T')\leq {\rm dim}_{\mathbb{C}}SHI(-M,-\ga,T).$$
\eprop

\bpf
We prove this proposition in followings steps. By Theorem \ref{thm_2: isomorphism between SHI with and without tangles}, it suffices to prove $$\dim_{\mathbb{C}}\shi(-M_{T},-\ga_{T})\le\dim_{\mathbb{C}}\shi(-M_{T^\p},-\ga_{T^\p}),$$where $(M_T,\ga_T)$ and $(M_{T^\p},\ga_{T^\p})$ are constructed as in Definition \ref{defn: tangle}.

{\bf Step 1}. We construct an auxiliary manifold $M_{T_0}$ with a family of sutures $\Ga_{n}$ for $n\in \mathbb{N}\cup \{-,+\}$ on $\partial M_{T_0}$. 

Since $[a]=0\in H_1(M,\partial M;\mathbb{Q}),$ there exist $q,k\in\intg$ and arcs $b_1,\dots,b_k\subset \partial M$ such that there exists a surface $S$ in $M$ with $\partial S$ consisting of $b_1$,\dots, $b_k$ and $q$ copies of $a$. Suppose components of $T$ are $a_1,a_2,\dots,a_m$, with $a_1=a$.

As in Definition \ref{defn: tangle}, we form a new balanced sutured manifold $(M_T,\ga_T)$ as follows. Since $T$ has $m$ components, $\partial N(T)$ intersects each of $R_+(\ga)$ and $R_-(\ga)$ in $m$ disks and intersects ${\rm int}(M)$ in $m$ cylinders. Let $C_1,\dots,C_m$ be the cylinders corresponding to $a_1,\dots,a_m$, respectively. Let $M_T=M\backslash{\rm int}(N(T))$. For any integer $i\in [1,n]$, let $\ga_i\subset C_i$ be a simple closed curve representing the generator of $H_1(C_i)$. Let $a_i$ be oriented from $R_+(\ga)$ to $R_-(\ga)$. Then $\ga_i$ has an induced orientation from $a_i$. Let
$$\ga_T=\ga\cup\ga_1\cup\dots\cup \ga_m.$$

The surface $S$ is modified into a properly embedded surface $S_T$ in $M_T$ as follows. First, for the part of $\partial S$ consists of $q$ copies of $a$, we isotop them to be on $C_1$. Then $b_j$ for any integer $j\in[1,k]$ can be viewed as an arc on $\partial M_T\backslash C_1$. We can isotop $S$ to make it intersect $a_i$ transversely for any integer $i\in[2,m]$. Let $S_T$ be obtained from $S$ by removing disks in $N(T)$. Hence $S_T\cap C_1$ consists of $q$ arcs, each intersecting $\ga_1$ transversely at one point. For any integer $i\in[2,m]$, the intersection $S_T\cap C_i$ is a (possibly empty) collection of circles that are parallel to $\ga_i$.

Note that $b_1$ is disjoint from all $\ga_i$, but intersects the original suture $\ga$. Since the arc $b_1\subset \partial S_T\subset \partial M_T$ has one endpoint in $R_+(\ga_T)$ and the other in $R_-(\ga_T)$, the intersection number of $\ga$ and $b_1$ must be odd. Let $b_1$ be oriented from $R_+(\ga_T)$ to $R_-(\ga_T)$ and let the intersection points between $b_1$ and $\ga$ be $p_1,\dots,p_l$ with $l$ odd, ordered by the orientation of $b_1$. Let $b_1^\p$ be a perturbation of $b_1$ such that $b_1^\p$ and $b_1$ meet at endpoints. Suppose the intersection points between $b_1^\p$ and $\ga$ are $q_1,\dots,q_l$ so that $q_i$ is near $p_i$ for integer $i\in[1,l]$.

If $l=1$, then $(M_T,\ga_T)$ is enough for the proof. If $l>1$, we have to perform the following modification on $(M_T,\ga_T)$. We attach a contact 1-handle to $(M,\ga)$ along $q_1$ and $q_{l-1}$ in the sense of Baldwin and Sivek \cite[Section 3.2]{baldwin2016instanton}, or equivalently, attach a product 1-handle in the sense of Kronheimer and Mrowka \cite[Proof of Proposition 6.9]{kronheimer2010knots}. They both proved that the balanced sutured manifolds before and after attaching such a 1-handle have exactly the same closure. Thus, after attaching the 1-handle, the sutured instanton Floer homology does not change. We still use $(M,\ga)$ and $(M_T,\ga_T)$ to denote sutured manifolds after attaching the 1-handle. Now we can choose an arc $\zeta$ satisfying the following conditions.

\benu
\item Endpoints of $\zeta$ are contained in $\partial C_1$.

\item The arc $\zeta$ intersects $\ga$ transversely at one point.

\item The arc $\zeta$ is disjoint from $S_T$.
\eenu
The arc $\zeta$ can be obtained by first going along $b_1'$ until reaching $q_1$, then going along the 1-handle, and going back to $b_1'$ at the point $q_{l-1}$ and then keeping going along $b_1'$. Finally, we slightly perturb this arc to make it disjoint from $S_T$. See the middle subfigure of Figure \ref{fig: adding_one_handle}.

\begin{figure}[ht]
\centering
\begin{overpic}[width=4in]{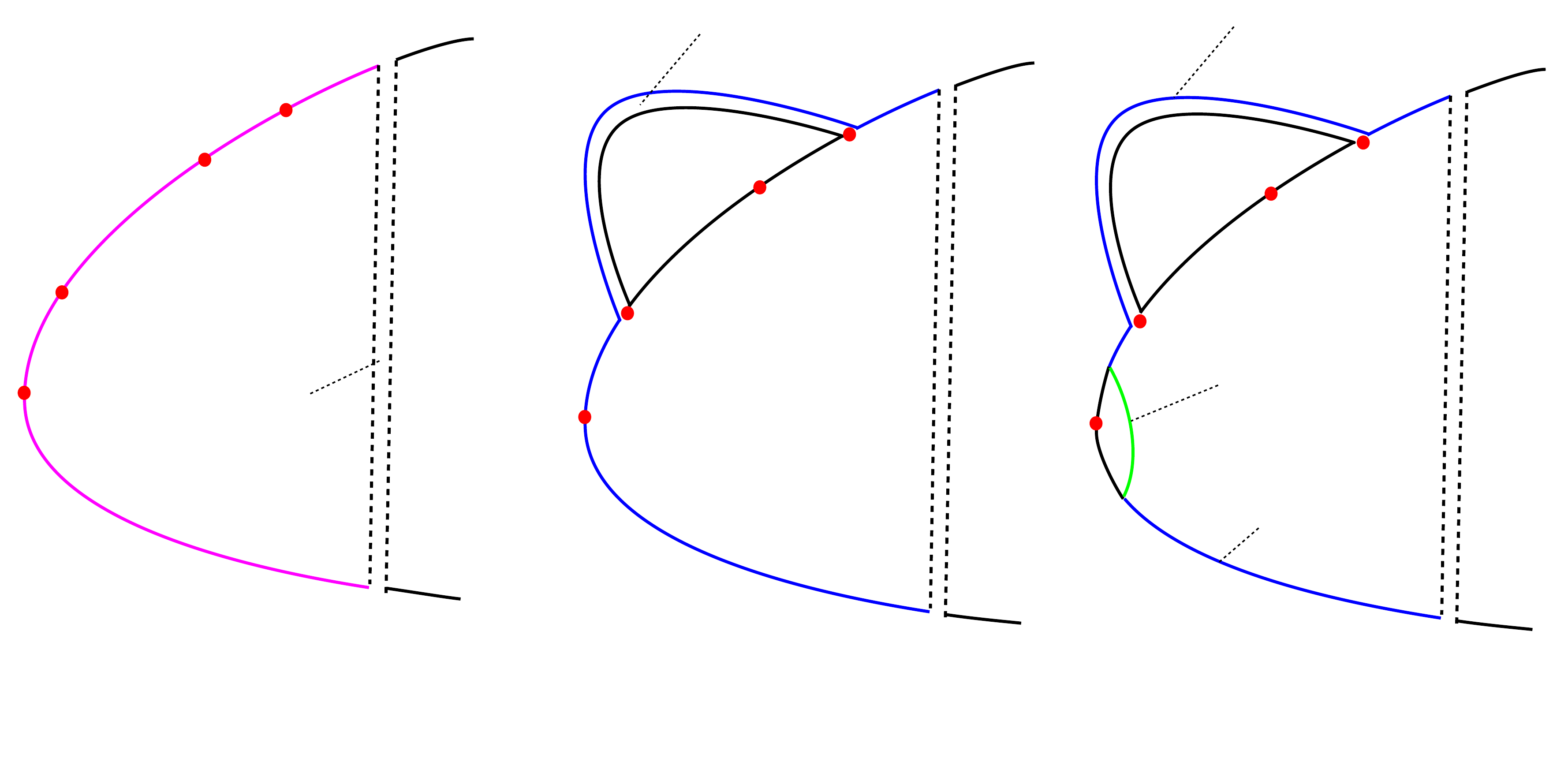}
    \put(11,23){$N(a_1)$}
    \put(3,15){$b_1$}
    \put(36,15){$b_1'$}
    \put(-2.5,24){$p_l$}
    \put(-4,31){$p_{l-1}$}
    \put(8,40){$p_{2}$}
    \put(15,44){$p_{1}$}
    \put(33,23){$q_l$}
    \put(42,30){$q_{l-1}$}
    \put(48,36){$q_{2}$}
    \put(53,39){$q_{1}$}
    \put(45,47){1-handle}
    \put(79,47.5){$\zeta_+$}
    \put(81,17){$\zeta_{-}$}
     \put(78,25){$a_0$}
    %\put(63,15){$a_{+,2}^1$}
    %\put(90,10){$\lambda$}
    %\put(96,4){$-\mu$}
\end{overpic}
\vspace{-0.5in}
\caption{Left, the arc $b_1$; middle, the 1-handle on $\beta_1'$ and the arc $\zeta$; right, the arcs $\zeta_+$, $a_0$, and $\zeta_-$.}\label{fig: adding_one_handle}
\end{figure}
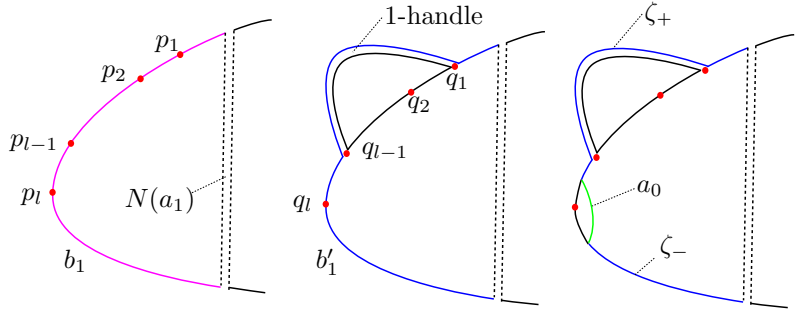

Let $a_0$ be the arc obtained by pushing a neighborhood of $q_l$ in $\zeta$ into the interior of $M_T$. Suppose the endpoints of $a_0$ are still in $\zeta$ and $a_0$ is disjoint from $S_T$. The arc $a_0$ is a vertical tangle in $M_T$ and hence also a vertical tangle in the original manifold $M$. Let $T_0=T\cup a_0$, and $T_0^\p=T^\p\cup a_0$. Let $(M_{T_0},\ga_{T_0})$ be obtained similarly as $(M_{T},\ga_T)$. Since $M_{T_0}=M_T\backslash N(a_0)$, the cylinder $C_0$ and the suture $\ga_0$ are defined similarly as $C_i$ and $\ga_i$ for any integer $i\in [1,m]$. A sketch is shown in the left-subfigure of Figure \ref{fig: the new Ga_0}. Let $\zeta_{\pm}$ be two parts of $\zeta$ contained in $R_{\pm}(\ga_{T_0})$, respectively.

\begin{figure}[ht]
\centering
\begin{overpic}[width=0.7\textwidth]{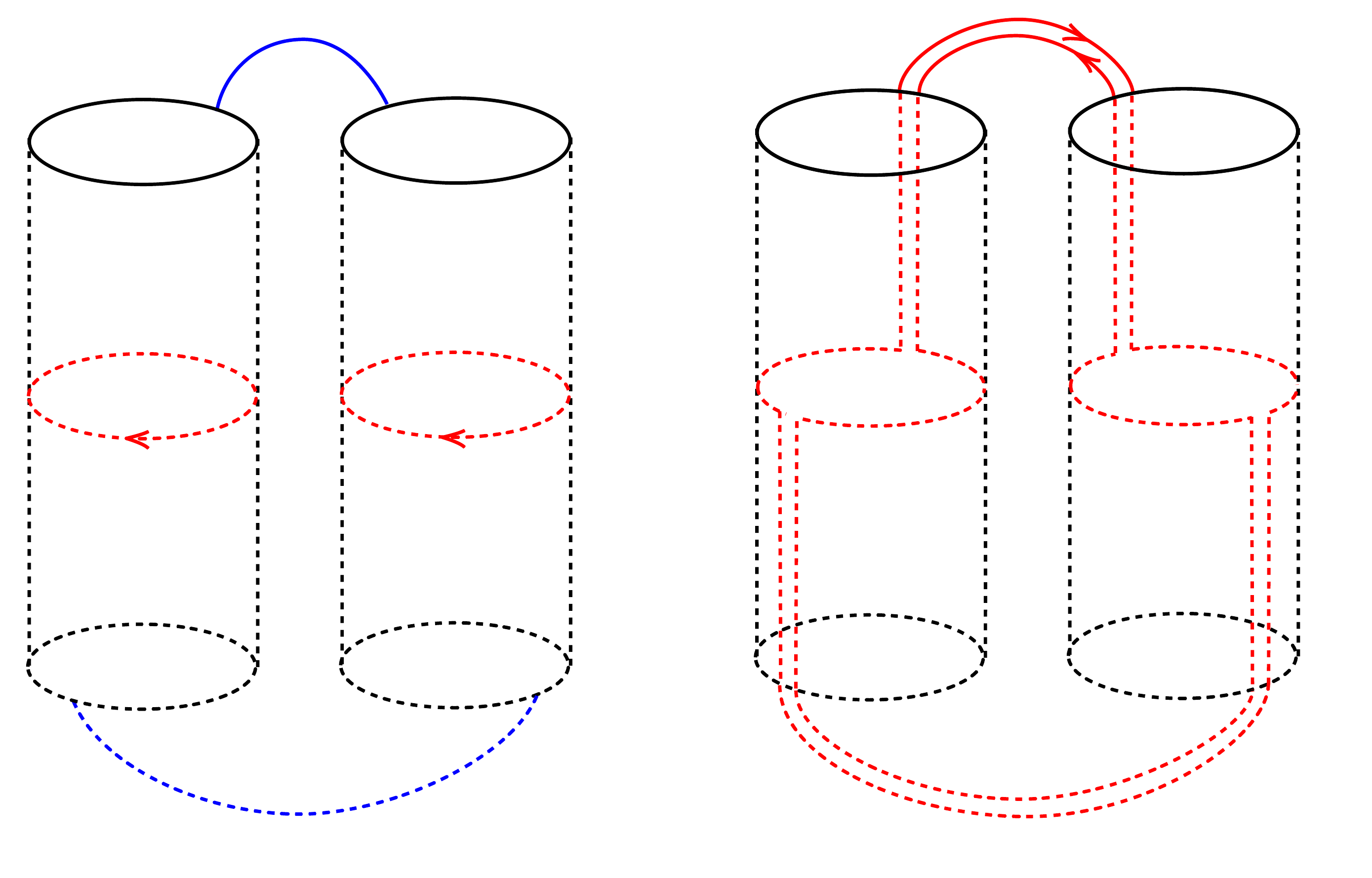}
    \put(7,23){$N(a_0)$}
    \put(30,23){$N(a_1)$}

    \put(43,15){$R_{-}(\ga_{T_0})$}
    \put(43,55){$R_{+}(\ga_{T_0})$}

    \put(8,40){$\ga_0$}
     \put(31.5,40){$\ga_1$}

    \put(20,62.5){$\zeta_+$}
    \put(20,7){$\zeta_-$}

    \put(73,63.5){$\Ga_0$}
\end{overpic}
\vspace{-0.3in}
\caption{Left, a sketch of the balanced sutured manifold $(M_{T_0},\ga_{T_0})$ with the arcs $\zeta_+$ and $\zeta_-$; right, the suture $\Ga_0$.}\label{fig: the new Ga_0}
\end{figure}
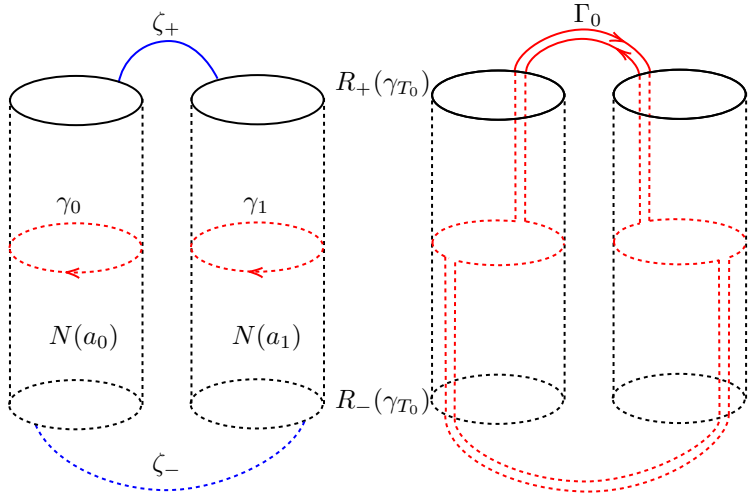

Now, we pick a family of sutures $\Ga_n$ on $M_{T_0}$ such that $\Ga_n$ is the suture obtained from $\ga_{T_0}$ by replacing $\ga_0$ and $\ga_1$ by other two curves. We describe the two new curves as follows. For $\Ga_0$, the new curves are depicted in the right subfigure of Figure \ref{fig: the new Ga_0}. Note that the part of the new curves in $R_{\pm}(\ga_{T_0})$ consists of two parallel copies of $\zeta_{\pm}$, respectively. The suture $\Ga_n$ is obtained from $\Ga_0$ by Dehn twists along $-\ga_1$ for $n$ times, as shown in the left subfigure of Figure \ref{fig: the new Ga_n}.

\begin{figure}[ht]
\centering
\begin{overpic}[width=0.7\textwidth]{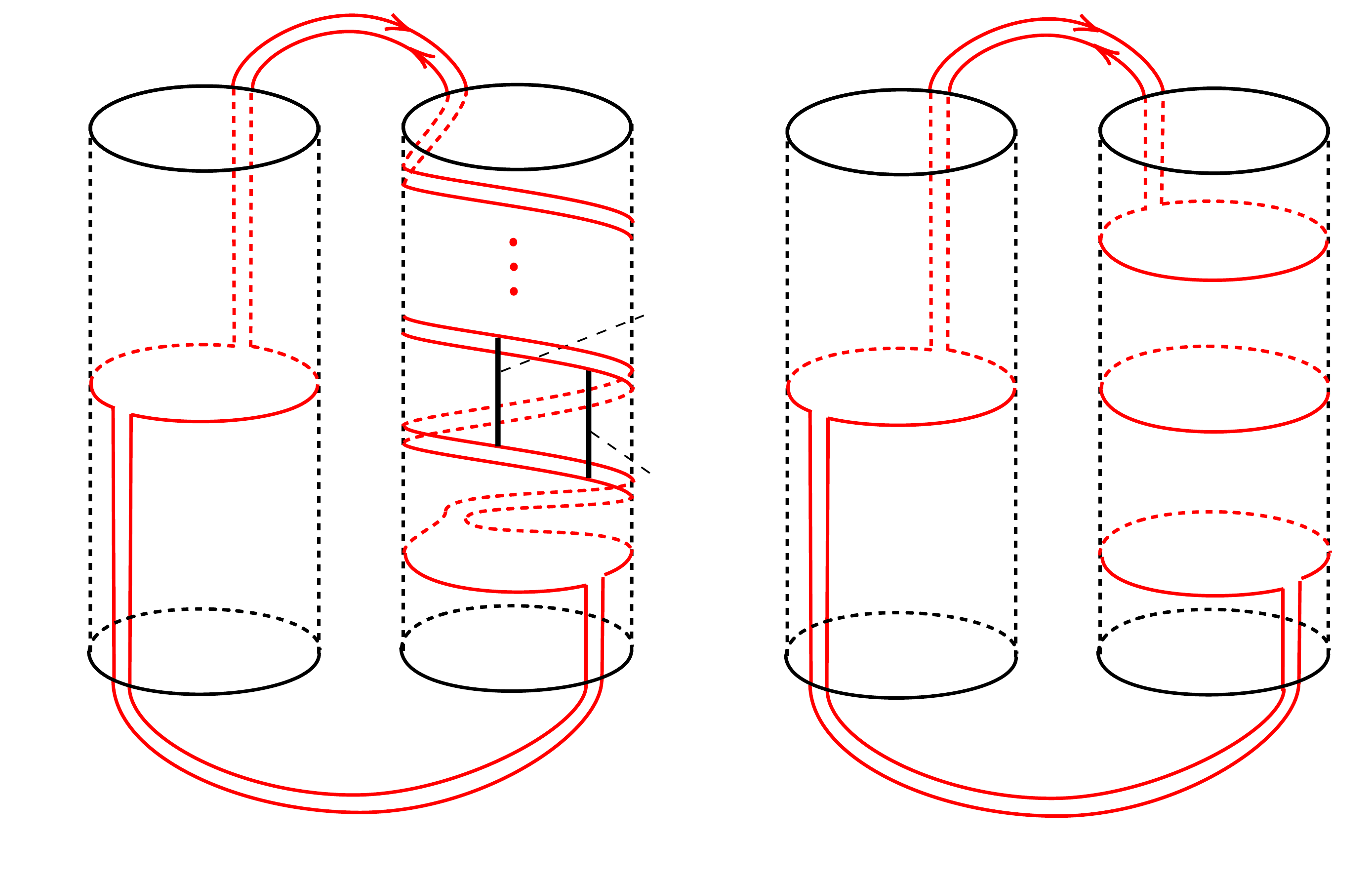}
    \put(48,28){$\eta_+$}
    \put(48,41){$\eta_-$}

    \put(47,15){$R_{-}(\ga_{T_0})$}
    \put(47,55){$R_{+}(\ga_{T_0})$}

    \put(23,64){$\Ga_{n}$}
    \put(75,64){$\Ga_+$}
\end{overpic}
\vspace{-0.3in}
\caption{Left, the suture $\Ga_n$, and the bypass arcs $\eta_+$ and $\eta_-$; right, the suture $\Ga_+$.}\label{fig: the new Ga_n}
\end{figure}

There are two obvious bypass arcs in the left-subfigure of Figure \ref{fig: the new Ga_n}, denoted by $\eta_+$ and $\eta_-$, respectively. By Theorem \ref{thm_2: bypass exact triangle on general sutured manifold}, these two bypass arcs induce two bypass exact triangles:
\begin{equation}\label{eq_3: bypass triangle, +}
    \xymatrix@R=6ex{
    \shi(-M_{T_0},-\Ga_{n-1})\ar[rr]^{\psi_{+,n}^{n-1}}&&\shi(-M_{T_0},-\Ga_{n})\ar[dl]^{\psi_{+,+}^{n}}\\
    &\shi(-M_{T_0},-\Ga_{+})\ar[ul]^{\psi_{+,n-1}^{+}}&
    }
\end{equation}
and
\begin{equation}\label{eq_3: bypass triangle, -}
    \xymatrix@R=6ex{
    \shi(-M_{T_0},-\Ga_{n-1})\ar[rr]^{\psi_{-,n}^{n-1}}&&\shi(-M_{T_0},-\Ga_{n})\ar[dl]^{\psi_{-,-}^{n}}\\
    &\shi(-M_{T_0},-\Ga_{-})\ar[ul]^{\psi_{-,n-1}^{-}}&
    }
\end{equation}
respectively, where $\Ga_{-}$ is the same as the original suture $\ga_{T_0}$, and $\Ga_+$ is the suture as depicted in the right subfigure of Figure \ref{fig: the new Ga_n}.

Note that the bypasses are attached to $\eta_+$ and $\eta_-$ from the exterior of the 3-manifold $M_{T_0}$, though the point of view in Figure \ref{fig: the new Ga_n} is from the interior of the manifold. Hence readers have to take extra care when performing these bypass attachments.

{\bf Step 2}. We use bypass maps and bypass triangles to derive a dimension inequality about $M_{T_0}$.

Recall we have a properly embedded surface $S_T\subset M_T$. Since $a_0\cap S_T=\emptyset$, we can regard $S_T$ as a properly embedded surface in $M_{T_0}$. Let $S_T$ be oriented so that the orientation of $\partial S_T$ coincides with that of $a$ ($=a_1$). Note that $\zeta_+$ and $\zeta_-$ are both disjoint from $S_T$. We can perform stabilizations on $S_T$ so that the followings hold.
\begin{enumerate}[(1)]
\item $S_T$ is admissible with respect to the suture $\Ga_-$ ($=\ga_{T_0}$).
\item For any $n\in \mathbb{N}\cup\{-,+\}$, $S_T$ has minimal possible number of intersection points with the part of the suture $\Ga_n\backslash(\ga\cup\ga_2\cup\dots\cup\ga_m).$
\item $S_T$ is disjoint from $\zeta_+\cup\zeta_-$.
\end{enumerate}

The surface after stabilizations is still denoted by $S_T$. After obtaining the surface $S_T$ satisfying the above three conditions, we further perform stabilizations as follows. For any $n\in\mathbb{N}\cup\{-,+\}$, if $S_T$ is already admissible with respect to $\Ga_n$, then let $S_n$ be the surface $S_T$ without any further change. If $S_T$ is not admissible with respect to $\Ga_n$, then we perform a negative stabilization on $S_T$ within $C_1$ to make it admissible, and write $S_n$ for the resulting surface. Equivalently, define a map
$$\tau:\mathbb{N}\cup\{-,+\}\mapsto \{0,-1\}$$
\begin{equation*}
\tau(n)=\left\{
\begin{array}{cl}
	0& {\rm If~}|S_T\cap\Ga_n\cap C_1|~{\rm is~odd},\\
	-1& {\rm If~}|S_T\cap\Ga_n\cap C_1|~{\rm is~even}.
\end{array}
\right.
\end{equation*}
Then we take $S_n=S_T^{\tau(n)}$, where the superscript follows Definition \ref{defn_2: stabilization of surfaces}. Note that all the stabilizations are with respect to $\Ga_n$ rather than $-\Ga_n$.

By Theorem \ref{thm_2: grading on SHI}, for any $n\in \mathbb{N}\cup\{-,+\}$, there is a closure $(Y_n,R_n,\omega_n)$ for $(M_{T_0},\Ga_n)$ so that $S_n$ extends to a closed surface $\bar{S}_n$. Let
\begin{equation*}
i^n_{max}=-\frac{1}{2}\chi(\bar{S}_n), ~{\rm and}~i^n_{min}= \frac{1}{2}\chi(\bar{S}_n)-\tau(n).
\end{equation*}

\begin{lem}\label{lem_3: top and bottom nontrivial gradings}
	If $i>i^{n}_{max}$ or $i<i^n_{min}$, then
	$\shi(-M_{T_0},-\Ga_n,S_n,i)=0.$
\end{lem}
\bpf
If $\tau(n)=0$, then the lemma follows directly from term (1) of Theorem \ref{thm_2: grading on SHI}. If $\tau(n)=-1$, then term (1) of Theorem \ref{thm_2: grading on SHI} only implies that for $i<i_{min}^n-1$,
$$\shi(-M_{T_0},-\Ga_n,S_n,i)=0.$$
For the remaining case where $i=i_{min}^n-1$, from term (3) of Theorem \ref{thm_2: grading on SHI}, we have
$$\shi(-M_{T_0},-\Ga_n,S_n,i)=\shi(-M_{T_0},-\Ga_n,-S_n,-i)=\shi(-M_{T_0},-\Ga_n,-S_n,-\frac{\chi(\bar{S}_n)}{2}).$$
From the construction of $S_n$, we know that $-S_n$ is obtained from $S_T$ by a negative stabilization with respect to the suture $-\Ga_n$. Hence we can apply Lemma \ref{lem_2: stabilization and decomposition} and term (2) of Theorem \ref{thm_2: grading on SHI} to obtain the vanishing result.
\epf

\brem
\textit{A priori}, we do not know if $\shi$ is non-vanishing at the gradings $i_{max}^n$ and $i^{n}_{min}$, though this does not make any difference in the proof of Proposition \ref{prop_3: dimension inequality for tangles}.
\erem

Next, we will derive a graded version of bypass exact triangles (\ref{eq_3: bypass triangle, +}) and (\ref{eq_3: bypass triangle, -}). To do so, we will discuss more about the surface $S_n$. Since $\partial S_T$ contains $q$ copies of $a_1$, for any $n\in\mathbb{N}$, we have
$$|S_T\cap \Ga_+\cap C_1|=q,~|S_T\cap \Ga_-\cap C_1|=3q,~{\rm and}~|S_T\cap \Ga_n\cap C_1|=(2n+1)q.$$
From Theorem \ref{thm_2: grading on SHI}, we know that for any $n\in\mathbb{N}$,
\begin{equation}\label{eq_3: indices, 1}
	\chi(\bar{S}_-)=\chi(\bar{S}_+)-q+\tau(-)~{\rm and}~\chi(\bar{S}_n)=\chi(\bar{S}_+)-nq+\tau(n).
\end{equation}
From Lemma \ref{lem_3: top and bottom nontrivial gradings}, we have
\begin{equation}\label{eq_3: indices, 3}
	\lim_{n\ra+\infty}i^n_{max}=+\infty~{\rm and}~\lim_{n\ra+\infty}i^n_{min}=-\infty.
\end{equation}
To present the grading shifting property better, we make the following definition.
\bdefn\label{defn_3: shifting the grading}
Suppose $(M,\ga)$ is a balanced sutured manifold and $S$ is an admissible surface in $(M,\ga)$. For any $i,j\in\intg$, define
$$\shi(M,\ga,S,i)[j]=\shi(M,\ga,S,i-j).$$
\edefn
\blem\label{lem_3: graded bypass triangle}
For any $n\in\mathbb{N}$, we have two exact triangles
\begin{equation*}\label{eq_6: graded bypass, +}
\xymatrix{
\shi(-M_{T_0},-{\Ga}_{n},S_n)[{i}_{min}^{n+1}-{i}_{min}^{n}]\ar[rr]^{\quad\quad\psi^{n}_{+,n+1}}&&\shi(-M_{T_0},-{\Ga}_{n+1},S_{n+1})\ar[dll]^{\psi^{n+1}_{+,+}}\\
\shi(-M_{T_0},-{\Ga}_{+},S_{+})[{i}_{max}^{n+1}-{i}_{max}^{+}]\ar[u]^{\psi^{+}_{+,n}}&&
}
\end{equation*}
and
\begin{equation*}\label{eq_3: graded bypass, -}
\xymatrix{
\shi(-M_{T_0},-{\Ga}_{n},S_{n})[{i}_{max}^{n+1}-{i}_{max}^n]\ar[rr]^{\quad\quad\psi^{n}_{-,n+1}}&&\shi(-M_{T_0},-{\Ga}_{n+1},S_{n+1})\ar[dll]^{\psi^{n+1}_{-,-}}\\
\shi(-M_{T_0},-{\Ga}_{-},S_{-})[{i}_{min}^{n+1}-{i}_{min}^{-}]\ar[u]^{\psi^{-}_{-,n}}&&
}.
\end{equation*}
Furthermore, all maps in the above two exact triangles are grading preserving.
\elem

\begin{proof}
We only prove the grading shifting behavior of the map $\psi_{+,+}^{n+1}$ in the triangle (\ref{eq_3: bypass triangle, +}), and the proof for any other map is similar. For simplicity, we also assume that $\tau(n)=\tau(+)=0$ (Note that $\tau(+)=0$ by definition), and other cases are similar. From Subsection \ref{subsec: general bypass}, we know bypass maps are constructed via contact handle maps and ultimately via cobordisms maps associated to Dehn surgeries (\textit{c.f.} \cite[Section 3]{baldwin2016instanton}). It is obvious that the bypass arc is disjoint from $\partial S_{n+1}$. By construction of the grading in Theorem \ref{thm_2: grading on SHI}, this implies that the following map is grading preserving:
$$\psi_{+,+}^{n+1}:\shi(-M_{T_0},-\Ga_{n+1},S_{n+1})\ra\shi(-M_{T_0},-\Ga_{+},S_{n+1}).$$
From Figure \ref{fig: the new Ga_n with bypass attached}, it is straightforward to check that $S_{n+1}$ is obtained from $S_+$ by
$$\frac{1}{2}(|S_{n+1}\cap \Ga_+\cap C_1|-|S_+\cap \Ga_+\cap C_1|)=2(i^{n+1}_{max}-i^{+}_{max})$$
negative stabilizations, with respect to the suture $\Ga_+$. Hence they become positive stabilizations with respect to $-\Ga_+$. Then the grading shift follows from Theorem \ref{thm_2: grading shifting property}.
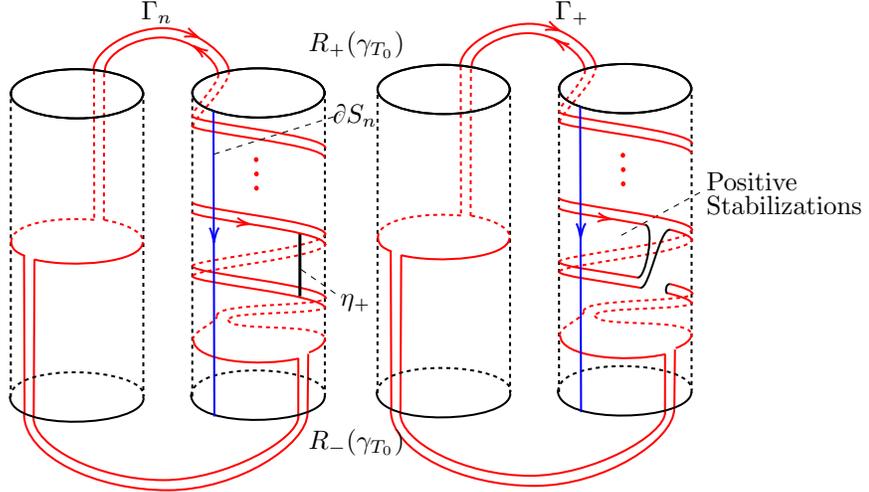
\begin{figure}[ht]
\centering
\begin{overpic}[width=0.7\textwidth]{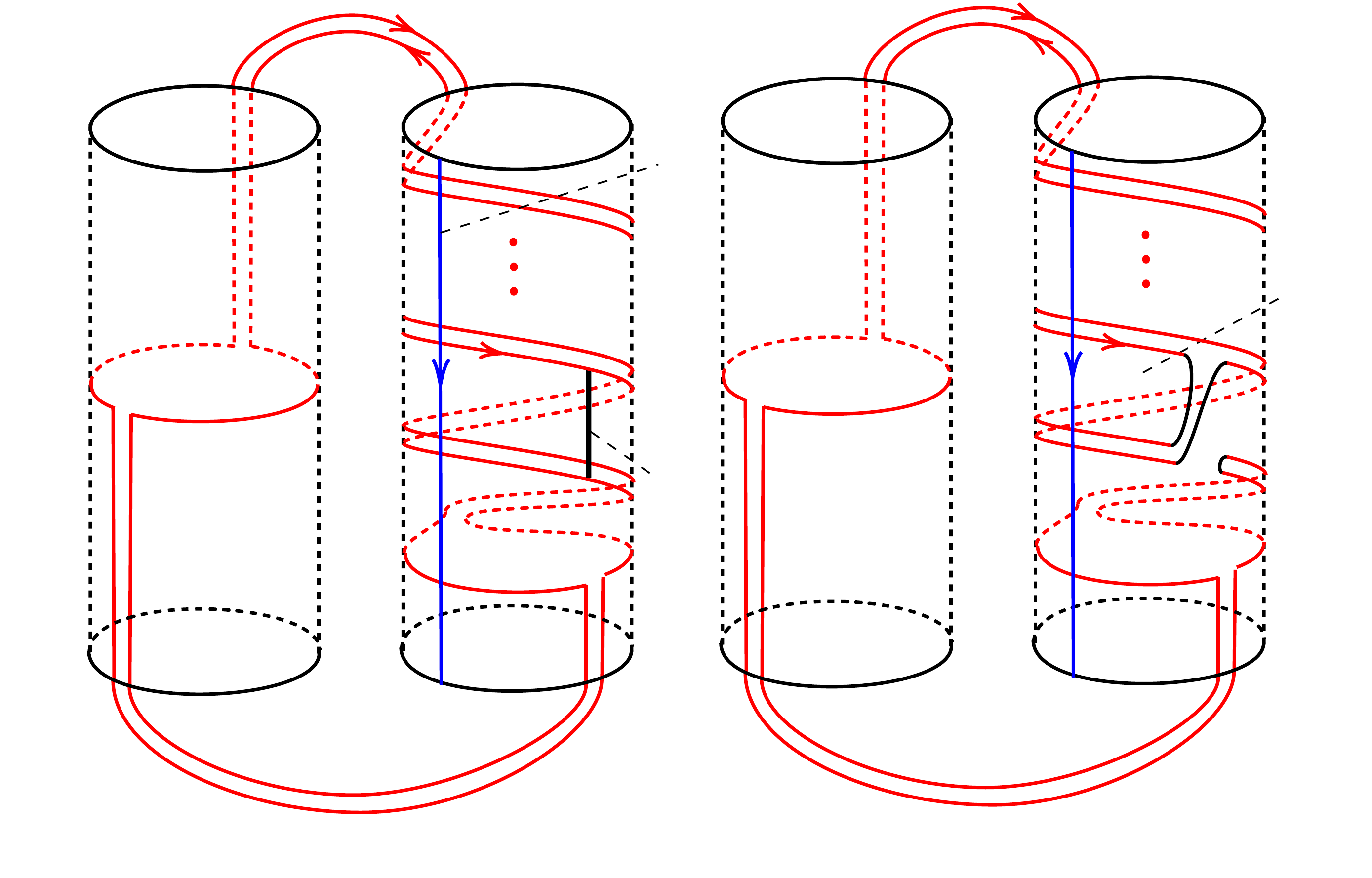}
    \put(48,28){$\eta_+$}
    \put(47,51){$\partial S_n$}
    \put(94,42.5){Positive}
    \put(94,39.5){Stabilizations}
    \put(44,10){$R_{-}(\ga_{T_0})$}
    \put(44,60){$R_{+}(\ga_{T_0})$}

    \put(23,64){$\Ga_{n}$}
    \put(75,64){$\Ga_+$}
\end{overpic}
\vspace{-0.3in}
\caption{Left, the suture $\Ga_n$, the surface $S_n$, and the bypass $\eta_+$; right, the suture $\Ga_+$ after the bypass attachment along $\eta_+$.}\label{fig: the new Ga_n with bypass attached}
\end{figure}
\end{proof}

\brem\label{block}
The statement of Lemma \ref{lem_3: graded bypass triangle} can be illustrated in Figure \ref{fig: illustration_of_lemma}, where sutured instanton homologies are denoted by the related sutures and maps are denoted by horizontal arrows. The heights of the blocks depend on $i_{max}-i_{min}$. This illustration is also useful for statements in Section \ref{sec: Instanton theory and knots}.
\erem
\begin{figure}[ht]
\centering
\includegraphics[width=0.7\textwidth]{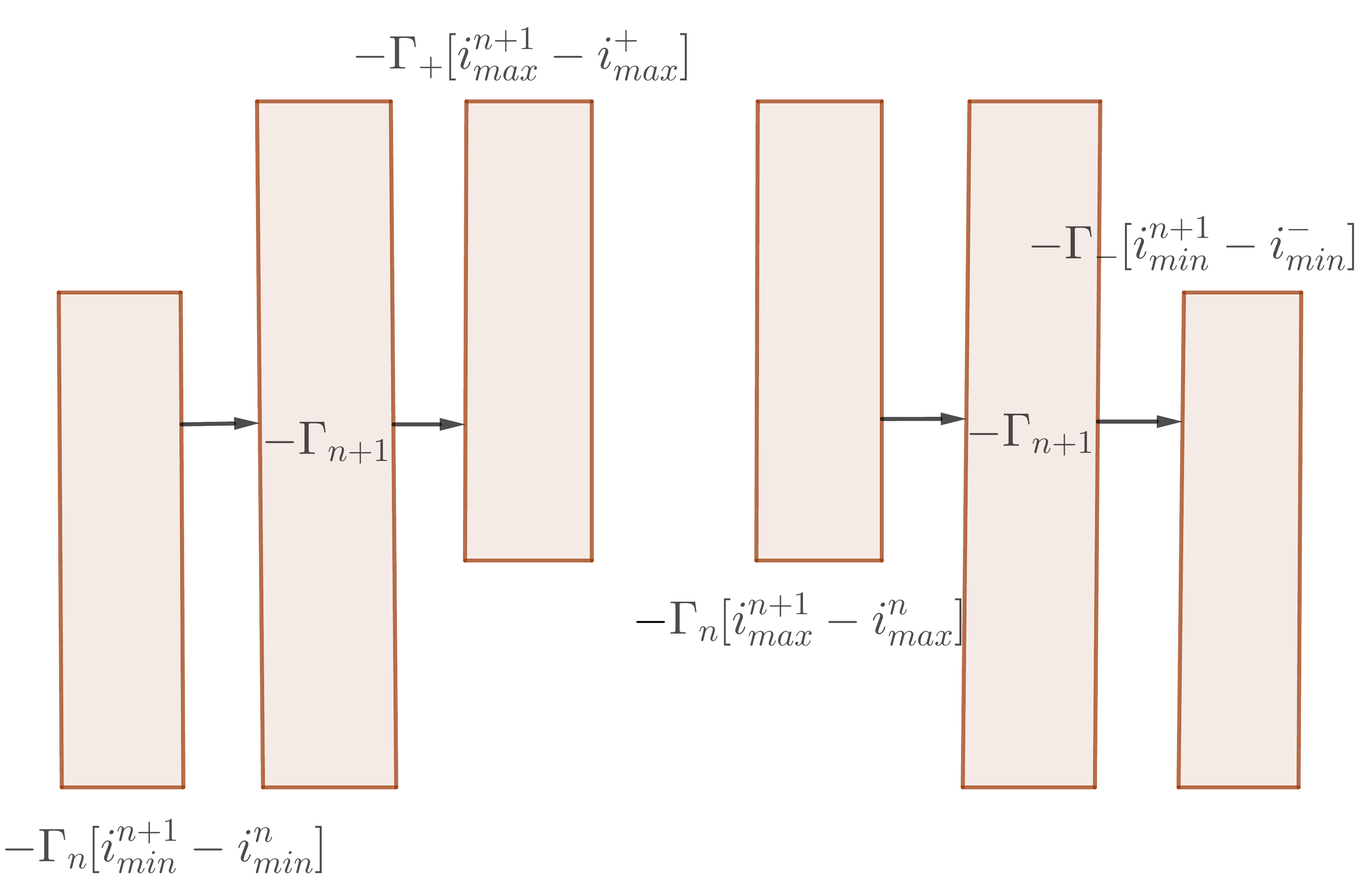}
\caption{Illustration of Lemma \ref{lem_3: graded bypass triangle}.}
\label{fig: illustration_of_lemma}
\end{figure}
Equipped with Lemma \ref{lem_3: graded bypass triangle}, we are able to prove the following lemma. For any $i\in\intg,n\in\mathbb{N}$, let $\psi_{\pm,n+1}^{n,i}$ be the restriction of $\psi_{\pm,n+1}^{n}$ on the $i$-th grading associated to $S_n$.

\begin{lem}\label{lem_3: iso at particular gradings}
	The map
$$\psi_{+,n+1}^{n,j}:\shi(-M_{T_0},-\Ga_n,S_n,j)\ra\shi(-M_{T_0},-\Ga_{n+1},S_{n+1},j-(i^{n}_{min}-i^{n+1}_{min}))$$
is an isomorphism if
$$j<i^{n+1}_{max}+(i^{n}_{min}-i^{n+1}_{min})-(i^{+}_{max}-i^{+}_{min}).$$
Similarly, the map
$$\psi_{-,n+1}^{n,j}:\shi(-M_{T_0},-\Ga_n,S_n,j)\ra\shi(-M_{T_0},-\Ga_{n+1},S_{n+1},j+(i^{n+1}_{max}-i^{n}_{max}))$$
is an isomorphism if
$$j>i^{n+1}_{min}-(i^{n+1}_{max}-i^{n}_{max})+(i^{-}_{max}-i^{-}_{min}).$$
\end{lem}
\bpf
We only prove the first statement. The proof of the second argument is similar. Suppose
$$i=j-(i^{n+1}_{min}-i^{n+1}_{min}).$$
Then we know that there is a map
$$\psi_{+,+}^{n+1,i}:\shi(-M_{T_0},-\Ga_{n+1},S_{n+1},i)\ra \shi(-M_{T_0},-\Ga_{+},S_{+},i-i^{n+1}_{max}+i^{+}_{max}).$$
By assumption, we have
\beq
i-i^{n+1}_{max}+i^{+}_{max}&=j-(i^{n}_{min}-i^{n+1}_{min})-i^{n+1}_{max}+i^{+}_{max}\\
&<i^{n+1}_{max}+(i^{n}_{min}-i^{n+1}_{min})-(i^{+}_{max}-i^{+}_{min})-(i^{n}_{min}-i^{n+1}_{min})-i^{n+1}_{max}+i^{+}_{max}\\
&=i^{+}_{min}.
\eeq
Hence it follows from Lemma \ref{lem_3: top and bottom nontrivial gradings} that $\psi_{+,+}^{n+1,i}=0$. By Lemma \ref{lem_3: graded bypass triangle}, the map $\psi_{+,n+1}^{n,j}$ is surjective. The proof of injectivity is similar.
\epf

\blem\label{lem_3: surgery triangle}
For any $n\in\mathbb{N}$, there is an exact triangle
\begin{equation}\label{eq_6: surgery exact triangle}
\xymatrix@R=6ex{
\shi(-M_{T_0},-{\Ga}_{n})\ar[rr]&&\shi(-M_{T_0},-{\Ga}_{n+1})\ar[dl]^{F_{n+1}}\\
&\shi(-M_{T_0'},-\ga_{T_0'})\ar[ul]^{G_n}&
}
\end{equation}

Furthermore, we have two commutative diagrams related to $\psi^{n}_{+,n+1}$ and $\psi^{n}_{-,n+1}$, respectively
\begin{equation*}%\label{eq_6: commutative diagram, G, +}
\xymatrix@R=6ex{
\shi(-M_{T_0},-{\Ga}_{n})\ar[rr]^{\psi_{\pm,n+1}^n}&&\shi(-M_{T_0},-{\Ga}_{n+1})\\
&\shi(-M_{T_0'},-\ga_{T_0'})\ar[ul]^{G_n}\ar[ur]_{G_{n+1}}&
}
\end{equation*}
\elem
\begin{figure}[ht]
\centering
\begin{overpic}[width=0.7\textwidth]{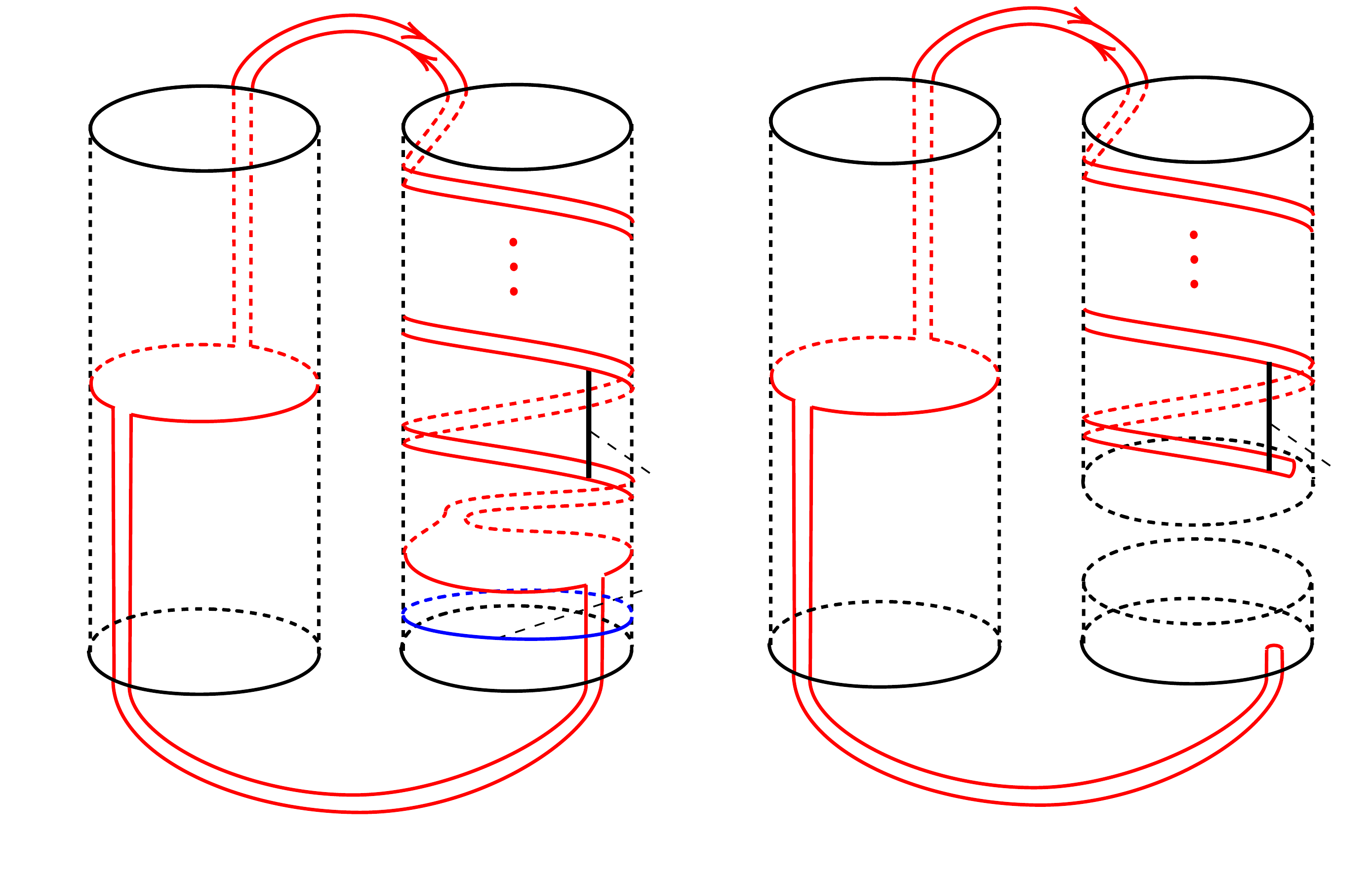}
    \put(48,28){$\eta_+$}
    \put(97,28){$\eta_+'$}
    \put(48,21){$\ga_1$}
    \put(40,60){$M_{T_0}$}
    \put(57,60){$M_{T^{\p}_0}$}

    \put(23,64){$\Ga_{n}$}
    \put(75,65){$\ga_{T^{\p}_0}$}
\end{overpic}
\vspace{-0.3in}
\caption{Left, the suture $\Ga_n$, the meridian $\ga_1$, and the bypass $\eta_+$; right, the balanced sutured manifold $(M_{T^{\p}_0},\ga_{T^{\p}_0})$ after attaching a contact 2-handle along $\ga_1$.}\label{fig: the new Ga_n with 2-handle attached}
\end{figure}
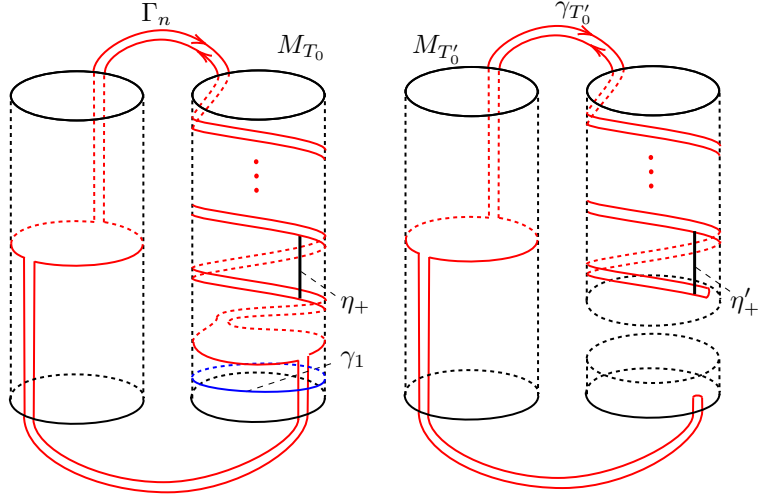
\bpf
Let $\ga_1^\p$ be the curve obtained by pushing $\ga_1$ into the interior of $M_{T_0}$, with the framing from $\partial M_{T_0}$. The $(0,1,\infty)$-surgery triangle associated to $\ga_1^\p$ is the following.

\begin{equation*}
\xymatrix@R=6ex{
((-M_{T_0})_{1},-\Ga_{n+1})\ar[rr]&&((-M_{T_0})_{\infty},-\Ga_{n+1})\ar[dl]\\
&((-M_{T_0})_{0},-\Ga_{n+1})\ar[ul]&
}
\end{equation*}
 Since $\ga_1^\p$ is in the interior of $M_{T_0}$, the surgeries do not influence the procedure of constructing closures of balanced sutured manifolds. Hence from Theorem \ref{thm_2: Floer's exact triangle} we have an exact triangle
\begin{equation*}
\xymatrix@R=6ex{
\shi((-M_{T_0})_{1},-\Ga_{n+1})\ar[rr]&&\shi((-M_{T_0})_{\infty},-\Ga_{n+1})\ar[dl]\\
&\shi((-M_{T_0})_{0},-\Ga_{n+1})\ar[ul]&
}
\end{equation*}
The $\infty$-surgery does not change anything, so
$$((-M_{T_0})_{\infty},-\Ga_{n+1})\cong (-M_{T_0},-\Gamma_{n+1}).$$
The 1-surgery is equivalent to a Dehn twist along $\ga_1^\p$. It does not change the underlying 3-manifold, while the suture $\Ga_{n+1}$ is replaced by $\Ga_n$:
$$((-M_{T_0})_{1},-\Ga_{n+1})\cong(-M_{T_0},-\Gamma_{n}).$$
Finally, for the $0$-surgery, from \cite[Section 3.3]{baldwin2016instanton}, we know that on the level of closures, performing a $0$-surgery is equivalent to attaching a contact 2-handle along $\ga_1\subset\partial M_{T_0}$. Attaching such a contact 2-handle changes $(M_{T_0},\Ga_{n+1})$ to $(M_{T_0'},\ga_{T_0'})$. Hence we obtain the desired exact triangle.

To prove two commutative diagrams, first note that the curve $\ga_{1}^\p$ is disjoint from the bypass arc $\eta_+$. As a result, the related maps commute with each other:
$$\psi_{+,n+1}^n\circ G_n=G_{n+1}\circ \psi_{\eta_+'},$$
where $\eta_+'$ is the bypass arc as shown in the right subfigure of Figure \ref{fig: the new Ga_n with 2-handle attached}. It is straightforward to check that the bypass along $\eta_+'$ is a trivial bypass, and hence from \cite[Section 2.3]{honda2002gluing} it does not change the contact structure. From Subsection \ref{subsec: general bypass}, the bypass maps can be reinterpreted as contact gluing maps, and by the functoriality of instanton contact gluing maps in \cite{li2018gluing}, we know that the bypass map $\psi_{\eta_+}$ corresponding to the trivial bypass $\eta_+$ is the identity map. Hence we conclude that
$$\psi_{+,n+1}^n\circ G_n=G_{n+1}\circ \psi_{\eta_+'}=G_{n+1}\circ {\rm id}=G_{n+1}.$$
The other commutative diagram involving $\psi_{-,n+1}^n$ can be proved similarly.
\epf
\begin{lem}\label{lem_3: G_n is zero}
	For a large enough integer $n$, the map $G_n$ in Lemma \ref{lem_3: surgery triangle} is zero.
\end{lem}
\bpf
We assume the lemma does not hold and derive a constradiction. For any $n$, there exists $x\in \shi(-M_{T_0'},-\ga_{T_0'})$ such that
$$y=G_n(x)\neq0\in \shi(-M_{T_0},-\Ga_n).$$
Suppose
$$y=\sum_{j\in \intg}y_j,~{\rm where~}y_j\in\shi(-M_{T_0},-\Ga_n,S_n,j),$$
$$j_{max}=\max_{y_j\neq0}j{\rm~and~}j_{min}=\min_{y_j\neq0}j.$$
By assumption $j_{max}$ and $j_{min}$ both exist and $j_{max}\geq j_{min}$.
Suppose
$$z=G_{n+1}(x)\in \shi(-M_{T_0},-\Ga_{n+1}),$$
and similarly
$$z=\sum_{j\in \intg}z_j,~{\rm where~}z_j\in\shi(-M_{T_0},-\Ga_{n+1},S_{n+1},j).$$

From facts (\ref{eq_3: indices, 1}) and (\ref{eq_3: indices, 3}), we know that for a large enough integer $n$, we have
$$i^{n+1}_{max}+(i^{n}_{min}-i^{n+1}_{min})-(i^{+}_{max}-i^{+}_{min})>i^{n+1}_{min}-(i^{n+1}_{max}-i^{n}_{max})+(i^{-}_{max}-i^{-}_{min}).$$
Hence at least one of the following two statements must be true.
\begin{enumerate}[(1)]
\item $j_{max}>i^{n+1}_{min}-(i^{n+1}_{max}-i^{n}_{max})+(i^{-}_{max}-i^{-}_{min})$
\item	 $j_{min}<i^{n+1}_{max}+(i^{n}_{min}-i^{n+1}_{min})-(i^{+}_{max}-i^{+}_{min})$
\end{enumerate}

We only work with the case where the first statement is true, and the other case is similar. From Lemma \ref{lem_3: surgery triangle}, we have
$$z=\psi_{+,n+1}^n(y)=\psi_{-,n+1}^n(y).$$
Suppose
$$i=j_{max}=j_{max}+(i_{max}^n-i^{n+1}_{max}),$$
and
$$j'=i+(i^{n}_{min}-i^{n+1}_{min}).$$
By Lemma \ref{lem_3: graded bypass triangle}, we have
$$\psi_{+,n+1}^{n,j'}(y_{j'})=z_{i}=\psi_{-,n+1}^{n,j_{max}}(y_{j_{max}}).$$
Since $j'>j_{max}$, we have $z_i=0$. By Lemma \ref{lem_3: iso at particular gradings}, the first statement implies $\psi_{-,n+1}^{n,j_{max}}$ is an isomorphism. Hence $y_{j_{max}}=0$, which contradicts the assumpotion of $j_{max}$.
\epf
Suppose $n$ is large enough. By the exact triangle (\ref{eq_6: surgery exact triangle}), the fact that $G_n$ is zero implies
$$\dim_{\mathbb{C}}\shi(-M_{T_0'},-\ga_{T_0'})=\dim_{\mathbb{C}}\shi(-M_{T_0},-\Ga_{n+1})-\dim_{\mathbb{C}}\shi(-M_{T_0},-\Ga_{n}).$$
From the exact triangle (\ref{eq_3: bypass triangle, -}) and the fact that $\Ga_-=\ga_{T_0}$, we have
$$\dim_{\mathbb{C}}\shi(-M_{T_0},-\ga_{T_0})\le\dim_{\mathbb{C}}\shi(-M_{T_0},-\Ga_{n+1})-\dim_{\mathbb{C}}\shi(-M_{T_0},-\Ga_{n}).$$

{\bf Step 3}. We obtain the desired inequality from the above equality and inequality.

Note that $a_0$ is an arc obtained by pushing a neighborhood of $q_l$ in $\zeta$ into the interior of $M_T$, and $M_{T_0}$ is obtained from $M_T$ by removing $N(a_0)$. There is an embedded disk $D$ in $M_{T_0}$ whose boundary is the union of $a_0$ and the neighborhood of $q_l$ in $\zeta$. Moreover, $\partial D$ intersects $\ga_{T_0}$ at two points, one of which is $q_l$ and the other is in $\ga_0$. By \cite[Proof of Proposition 6.9]{kronheimer2010knots}, decomposing the sutured manifold $(M_{T_0},\ga_{T_0})$ along $D$ does not change the isomorphism class of the sutured instanton Floer homology ($D$ is a product disk in the sense of \cite[Definition 2.8]{juhasz2008floer}). It is straightforward to check the sutured manifold after the sutured manifold decomposition is exactly $(M_{T},\ga_{T})$. Then we have
$$\dim_{\mathbb{C}}\shi(-M_{T_0},-\ga_{T_0})=\dim_{\mathbb{C}}\shi(-M_{T},-\ga_{T}).$$Similarly, we have
$$\dim_{\mathbb{C}}\shi(-M_{T_0^\p},-\ga_{T_0^\p})=\dim_{\mathbb{C}}\shi(-M_{T^\p},-\ga_{T^\p}).$$
Thus, we conclude$$\dim_{\mathbb{C}}\shi(-M_{T},-\ga_{T})\le\dim_{\mathbb{C}}\shi(-M_{T^\p},-\ga_{T^\p}).$$
\epf
\bpf[Proof of Theorem \ref{thm_1: from heegaard diagram to SHI}]
Suppose $(H,\ga)=(Y(1)_T,\delta_T)$ and $(Y(K),\be_{g}^\pp\cup\be_{g+1}^\pp)$ are obtained from Construction \ref{cons: meridian sutures}. Note that $\be_{g}^\pp\cup\be_{g+1}^\pp$ are parallel copies of the meridian of $K$. Then we have $$KHI(-Y,K)=\shi(-Y(K),-(\be_{g}^\pp\cup\be_{g+1}^\pp))$$ by Definition \ref{defn_2: framed instanton Floer homology}. Since $Y$ is a rational homology sphere, we have $$H_1(Y(1),\partial Y(1);\mathbb{Q})=0.$$In particular, any component of $T$ has trivial rational homology class. Then the theorem follows from Proposition \ref{prop_3: dimension inequality for tangles} and Theorem \ref{thm_2: isomorphism between SHI with and without tangles}.
\epf
\brem
Suppose $(\Sigma,\al,\be)$ is a Heegaard diagram of a rational homology sphere $Y$ and $K$ is the core knot of $\be_i$ for some $\be_i\subset \be$. Suppose $(H,\ga)=(Y(1)_T,\delta_T)$ is obtained from Construction \ref{cons: heegaard diagram gives rise to sutured handlebody}. Then the proof of Theorem \ref{thm_1: from heegaard diagram to SHI} applies without change, and we conclude the same inequality.
\erem
\bpf[Proof of Proposition \ref{prop: direct result}]
Similar to the proof of Theorem \ref{thm_1: from heegaard diagram to SHI}, since the knot $K$ has trivial rational homology class, the corresponding tangle has trivial homology class in $H_1(Y(1),\partial Y(1);\mathbb{Q}).$
\epf

\bcor\label{cor_3: combinatorial bound}
Suppose $(\Sigma,\al,\be)$ is a genus $g$ Heegaard diagram of a rational homology sphere $Y$. Let $(H,\ga)$ be the sutured handlebody obtained by Construction \ref{cons: heegaard diagram gives rise to sutured handlebody}. Suppose
$$n=\frac{1}{2}\sum_{i=1}^g|\al_i\cap\ga|.$$
Then we have
$${\rm dim}_{\mathbb{C}}I^{\sharp}(-Y)\leq 2^{n-g}.$$
\ecor
\bpf
By Theorem \ref{thm_1: from heegaard diagram to SHI}, we know that
$${\rm dim}_{\mathbb{C}}I^{\sharp}(Y)\leq {\rm dim}_{\mathbb{C}}\shi(-H,-\ga).$$
Hence it suffices to prove that
$$\shi(-M,-\ga)\leq 2^{n-g}.$$
First note that if $n<g$, then there exists an integer $i\in[1,g]$ such that $\al_i\cap\ga=\emptyset.$ Hence $(-H,-\ga)$ is not taut. By Theorem \ref{thm_2: SHI detects tautness} $\shi(-M,-\ga)=0$.

When $g=n$, then either there exists an integer $i\in[1,g]$ such that $\al_i\cap\ga=\emptyset,$ and then $\shi(-M,-\ga)=0$ as above, or for each $i\in\{1,\dots,g\}$, $\al_i$ intersects $\ga$ precisely at two points. Thus the disk $D_i\subset H$ bounded by $\al_i$ is a product disk. Let $D=D_1\cup\dots\cup D_g.$ We can perform a sutured manifold decomposition
$(-H,-\ga)\stackrel{D}{\leadsto}(D^3,\delta),$ where $D^3$ is a 3-ball and $\delta$ is a suture on $\partial D^3$. From Theorem \ref{product}, we know that
$${\rm dim}_{\mathbb{C}}\shi(-H,-\ga)={\rm dim}_{\mathbb{C}}\shi(D^3,\delta)=1.$$

We prove other cases by inductiono on $n$. Assume that the statement has been proved for $g\leq n=k-1$. For the case of $n=k$, we proceed as follows. Since $k>g$, there is a curve $\al_i$ satisfying $\al_i\cap \ga\geq 4$. Suppose $\eta\subset\al_i$ is an arc such that $\partial \eta\subset\ga$, and the interior of $\eta$ intersects with $\ga$ transversely once. By Theorem \ref{thm_2: bypass exact triangle on general sutured manifold}, we have a bypass exact triangle from the bypass attachement along $\eta$:
\begin{equation*}
\xymatrix@R=6ex{
&\shi(-H,-\ga)\ar[dr]&\\
\shi(-H,-\ga_2)\ar[ur]&&\shi(-H,-\ga_1)\ar[ll]\\
}
\end{equation*}
It is straightforward to check that
$$\sum_{i=1}^n|\ga_1\cap\al_i|\leq 2k-2{\rm~and~}\sum_{i=1}^n|\ga_2\cap\al_i|\leq 2k-2.$$
By the induction hypothesis and the exactness, we conclude that
$${\rm dim}_{\mathbb{C}}\shi(-H,-\ga)\leq {\rm dim}_{\mathbb{C}}\shi(-H,-\ga_1)+{\rm dim}_{\mathbb{C}}\shi(-H,-\ga_2)\leq 2^{k-g}.$$
Thus, we complete the induction.
\epf

\begin{exmp}\label{exmp_3: heegaard diagram}
Suppose we have a Heegaard diagram $(\Sigma,\al,\be)$ of $Y$ as in Figure \ref{fig: heegaard diagram, 0}. It is straightforward to check that
$$H_1(Y)=\intg_2\oplus\intg_2.$$We can apply Construction \ref{cons: heegaard diagram gives rise to sutured handlebody} to obtain a sutured handlebody $(H,\ga)$. See Figure \ref{fig: heegaard diagram, 1}. By Theorem \ref{thm_1: from heegaard diagram to SHI}, we know that
\begin{equation}\label{eq_3: heegaard diagram, 0}
{\rm dim}_{\mathbb{C}}I^{\sharp}(-Y)\leq {\rm dim}_{\mathbb{C}}\shi(-H,-\ga).
\end{equation}

It remains to bound $\dim_\mathbb{C}\shi(-H,-\ga)$. If we apply Corollary \ref{cor_3: combinatorial bound} directly, we obtain
$${\rm dim}_{\mathbb{C}}\shi(-H,-\ga)\leq 64.$$
However, we can improve this bound by examining the bypass exact triangles more carefully in the following steps.

\begin{figure}[htbp]
\centering
\begin{minipage}[t]{0.48\textwidth}
\centering
\begin{overpic}[width=7cm]{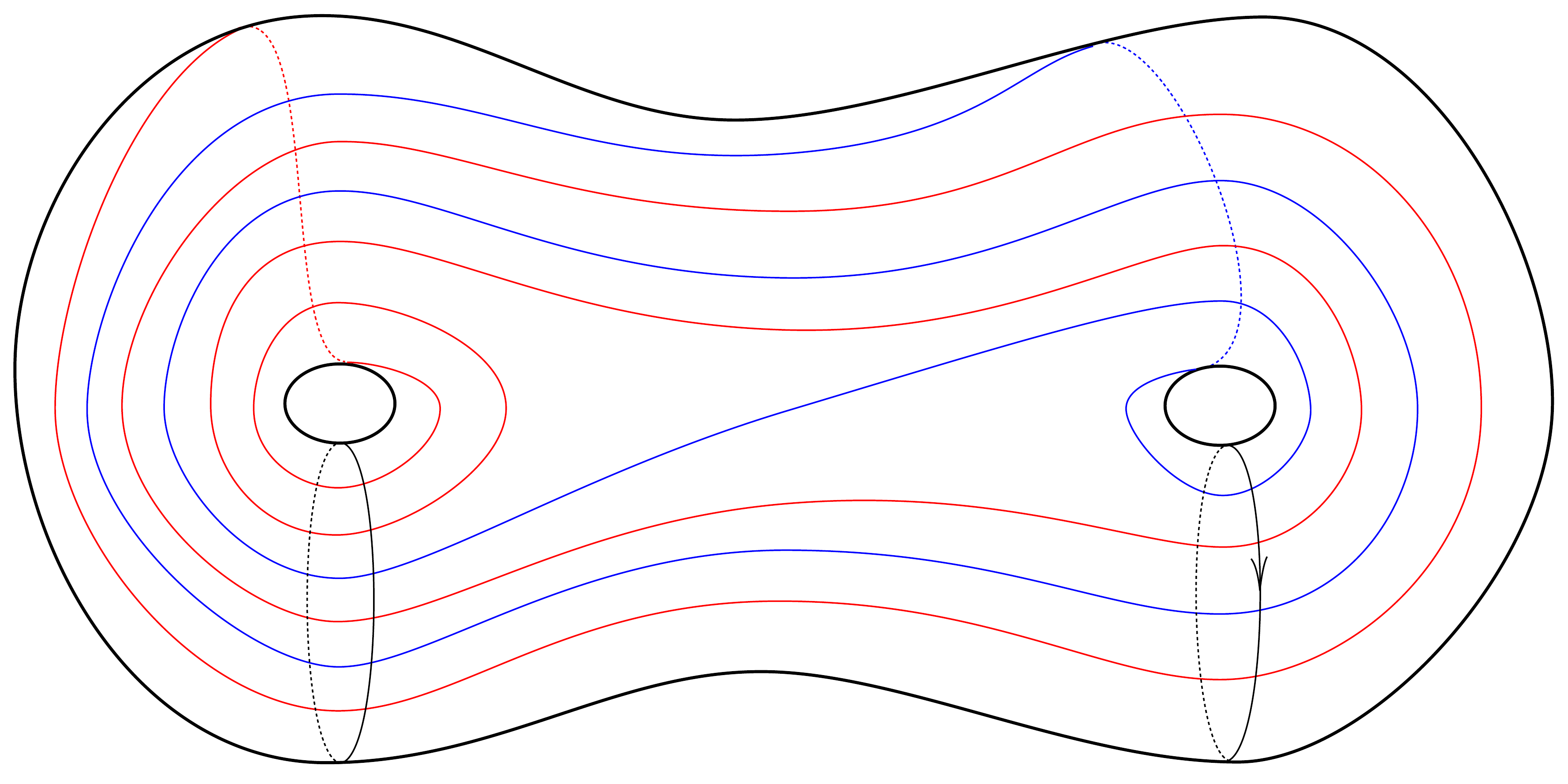}
\put(21,-1){$\al_1$}
\put(78,-1){$\al_2$}
\put(37,27){$\be_1$}
\put(54,22){$\be_2$}
\end{overpic}
\vspace{-0.0in}
\caption{A Heegaard diagram of $Y$.\label{fig: heegaard diagram, 0}}
\end{minipage}
\begin{minipage}[t]{0.48\textwidth}
\centering
\begin{overpic}[width=7cm]{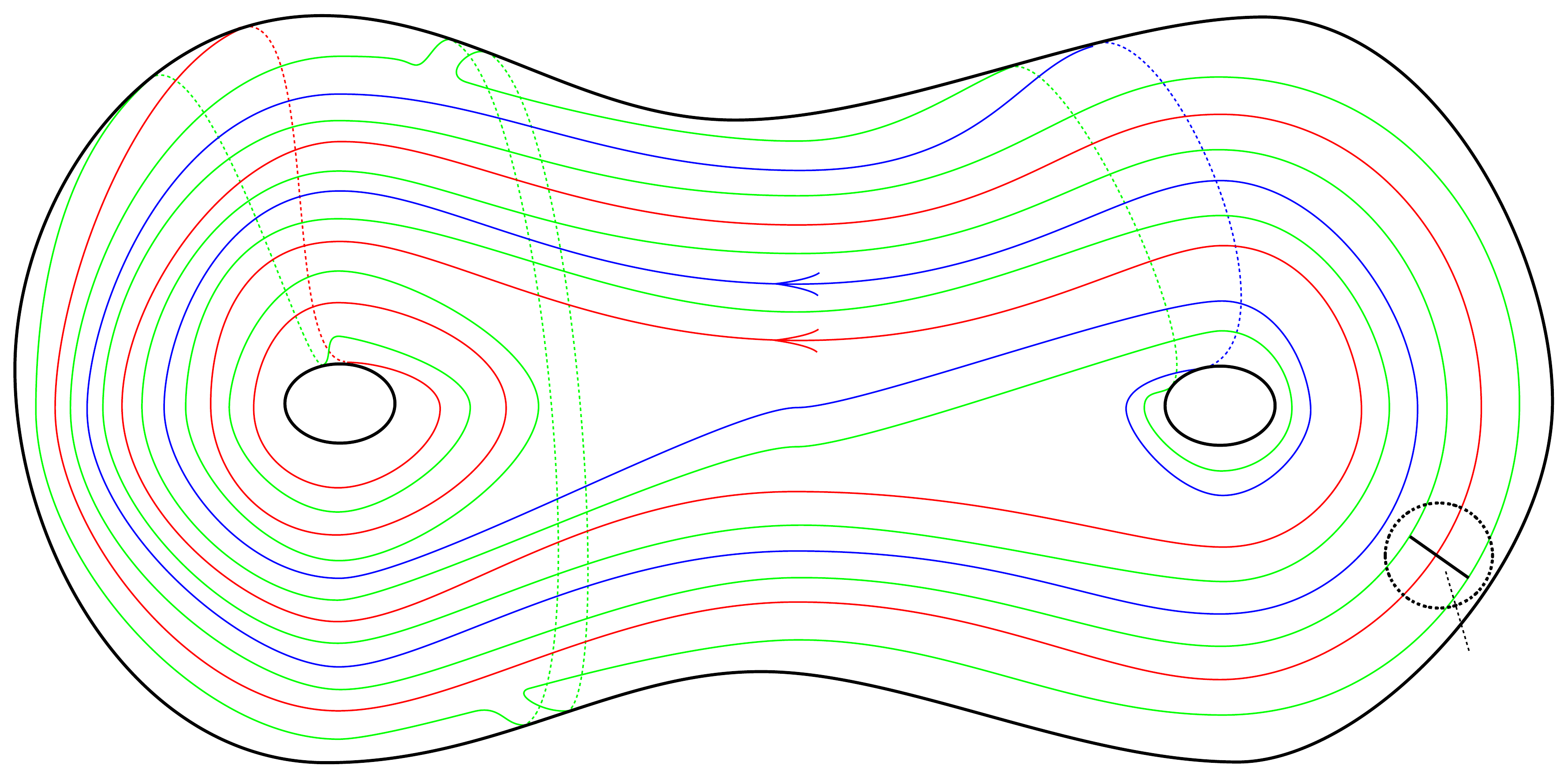}
\put(93,7){$\eta_1$}
\end{overpic}
\vspace{-0.0in}
\caption{The diagram of $(H,\ga)$.\label{fig: heegaard diagram, 1}}
\end{minipage}
\end{figure}
{\bf Step 1}. We can do a bypass attachement along the arc $\eta_1$ as shown in Figure \ref{fig: heegaard diagram, 1}. By Theorem \ref{thm_2: bypass exact triangle on general sutured manifold}, there exists an exact triangle
\begin{equation}\label{eq_3: heegaard diagram, 1}
\xymatrix@R=6ex{
&\shi(-H,-\ga)\ar[dr]&\\
\shi(-H,-\ga_2)\ar[ur]&&\shi(-H,-\ga_1)\ar[ll]\\
}
\end{equation}
The sutures $\ga_1$ and $\ga_2$ are depicted in Figure \ref{fig: heegaard diagram, 2} and Figure \ref{fig: heegaard diagram, 3}, respectively.

\begin{figure}[htbp]
\centering
\begin{minipage}[t]{0.48\textwidth}
\centering
\begin{overpic}[width=7cm]{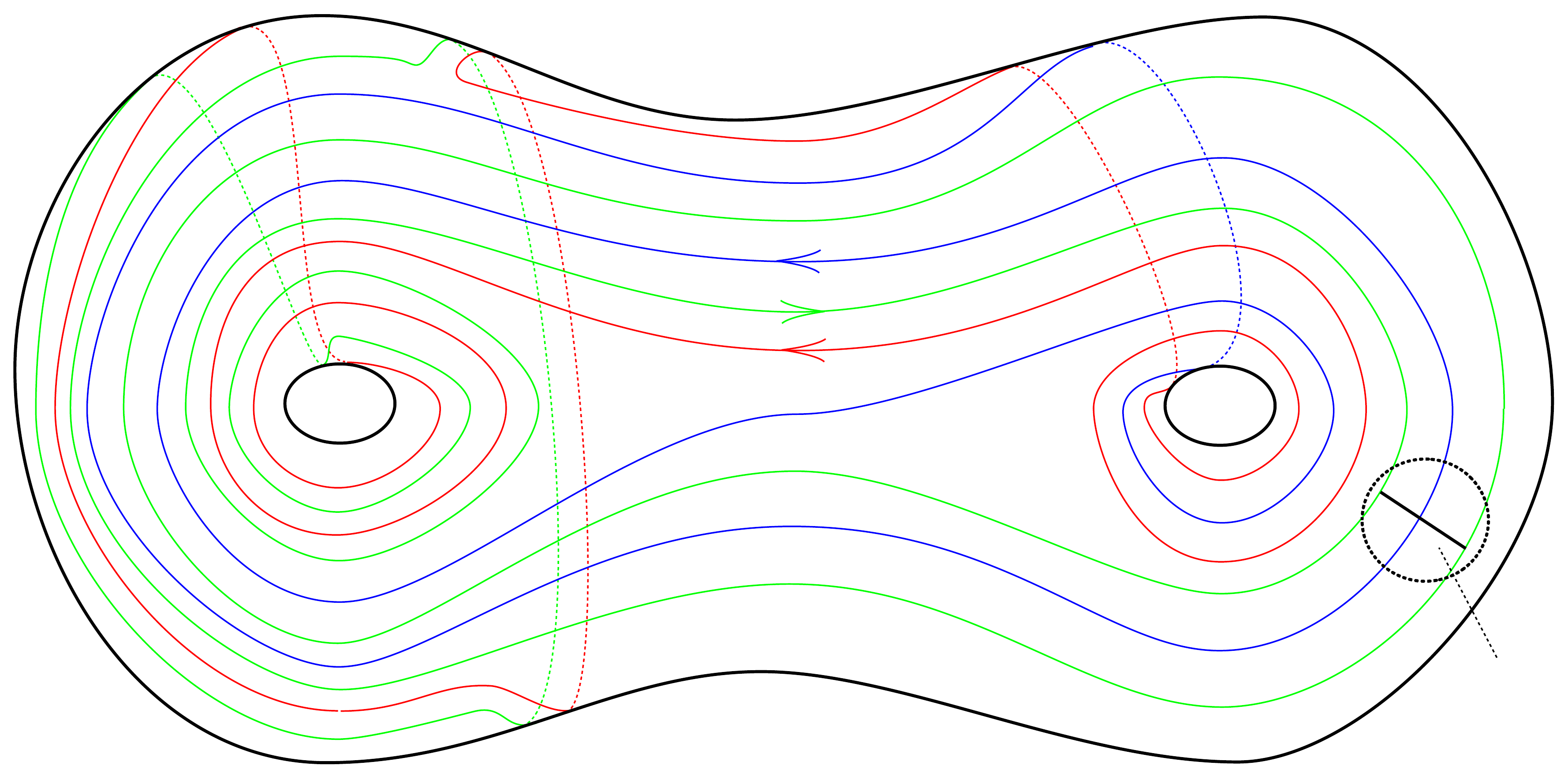}
\put(93,7){$\eta_2$}
\end{overpic}
\vspace{-0.0in}
\caption{The diagram of $(H,\ga_1)$.}\label{fig: heegaard diagram, 2}
\end{minipage}
\begin{minipage}[t]{0.48\textwidth}
\centering
\begin{overpic}[width=7cm]{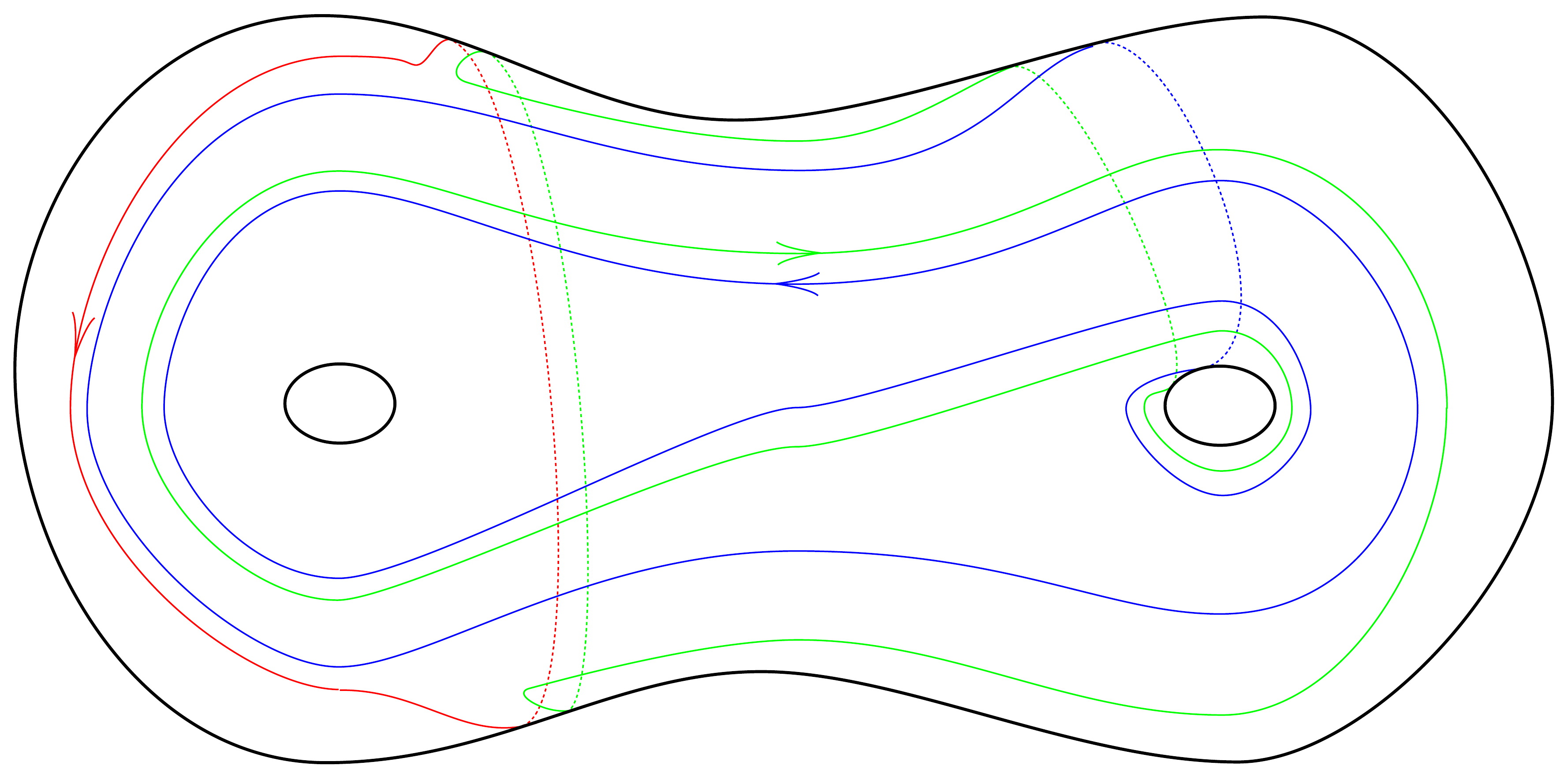}
%\put(93,7){$\eta_1$}
\end{overpic}
\vspace{-0.0in}
\caption{The diagram of $(H,\ga_2)$.}\label{fig: heegaard diagram, 3}
\end{minipage}
\end{figure}
{\bf Step 2}. We compute $\shi(-H,-\ga_2)$. The curve $\al_2$ bounds a disk $D\subset H$. It then induces a grading on $\shi(-H,-\ga_2)$. Note that $D$ is not admissible in the sense of Definition \ref{defn_2: admissible surfaces}, so we perform a negative stabilization on $D$ as in Definition \ref{defn_2: stabilization of surfaces}, and write $D^-$for the resulting disk. By Theorem \ref{thm_2: grading on SHI}, $\shi(-H,-\ga_2,D^-,i)=0$, for $|i|>1$. We can perform a sutured manifold decomposition $$(-H,-\ga_2)\stackrel{D^-}{\leadsto}(V,\ga_2'),$$where $V$ is a solid torus and $\ga_1'$ is depicted as in Figure \ref{fig: heegaard diagram, 4}. From Theorem \ref{thm_2: grading on SHI} and Theorem \ref{thm_2: SHI detects tautness}, we know that
$$\shi(-H,-\ga_2,D^-,1)\cong \shi(V,\ga_2')=0.$$
By Lemma \ref{lem_2: stabilization and decomposition}, Theorem \ref{thm_2: grading on SHI}, and Theorem \ref{thm_2: SHI detects tautness}, we know
\beq
\shi(-H,-\ga_2,D^-,-1)=&\shi(-H,-\ga_2,-D^-,1)\\
=&\shi(-H,-\ga_2,(-D)^+,1)\\
=&0.
\eeq
By Theorem \ref{thm_2: grading on SHI}, Theorem \ref{thm_2: grading shifting property}, and Theorem \ref{thm_2: SHI detects tautness}, we know
\beq
\shi(-H,-\ga_2,D^-,0)=&\shi(-H,-\ga_2,-D^-,0)\\
=&\shi(-H,-\ga_2,(-D)^+,0)\\
=&\shi(-H,-\ga_2,(-D)^-,1)\\
=&\shi(V,\ga_2'').
\eeq
Here $(V,\ga_2'')$ is obtained from $(-H,-\ga_2)$ by decomposing along $(-D)^-$. $V$ is a solid torus and $\ga_2''$ is depicted as in Figure \ref{fig: heegaard diagram, 5}.
\begin{figure}[htbp]
\centering
\begin{minipage}[t]{0.48\textwidth}
\centering
\begin{overpic}[width=7cm]{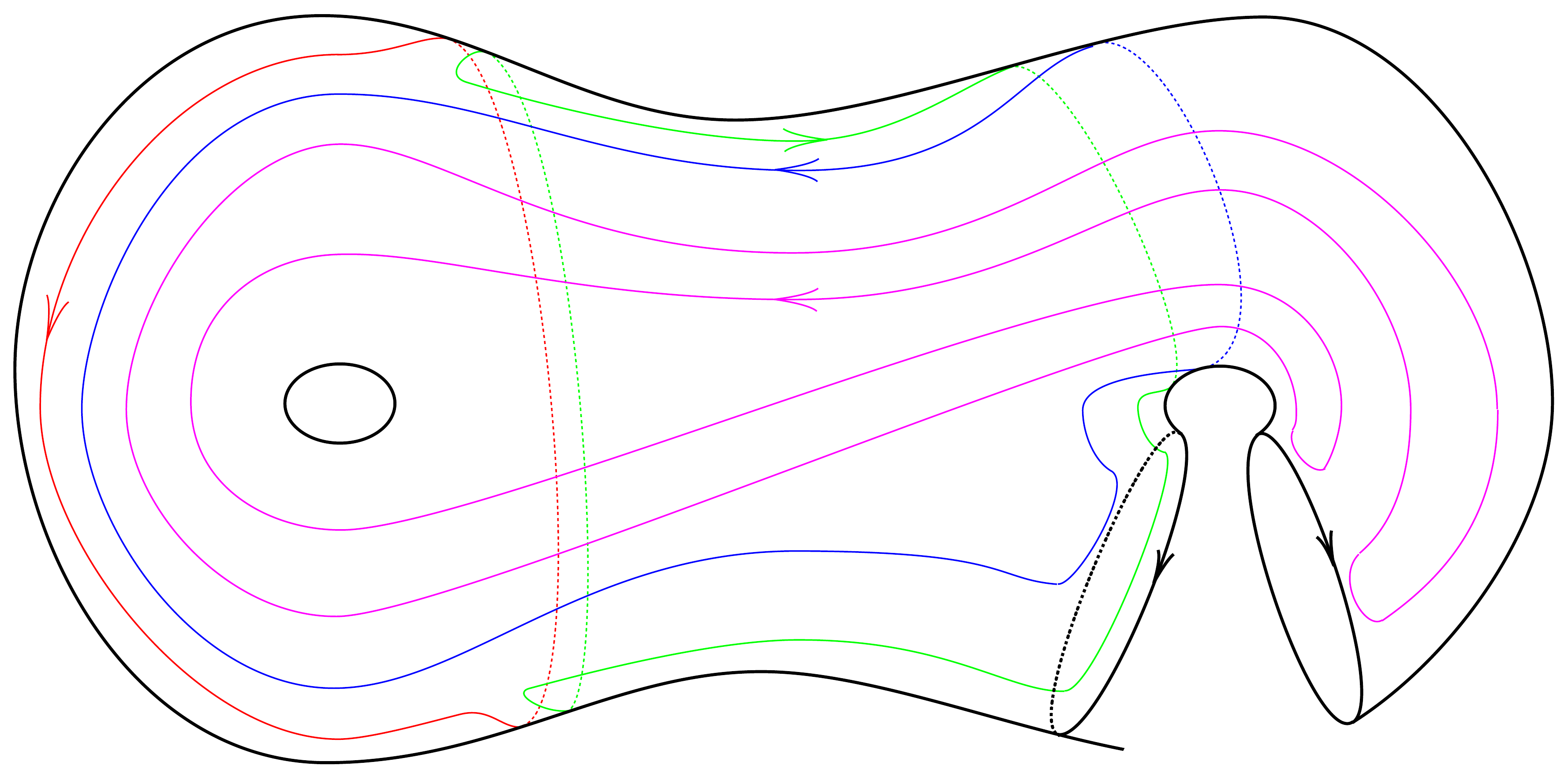}
%\put(93,7){$\eta_1$}
\end{overpic}
\vspace{-0.0in}
\caption{The diagram of $(V,\ga_1')$.}\label{fig: heegaard diagram, 4}
\end{minipage}
\begin{minipage}[t]{0.48\textwidth}
\centering
\begin{overpic}[width=7cm]{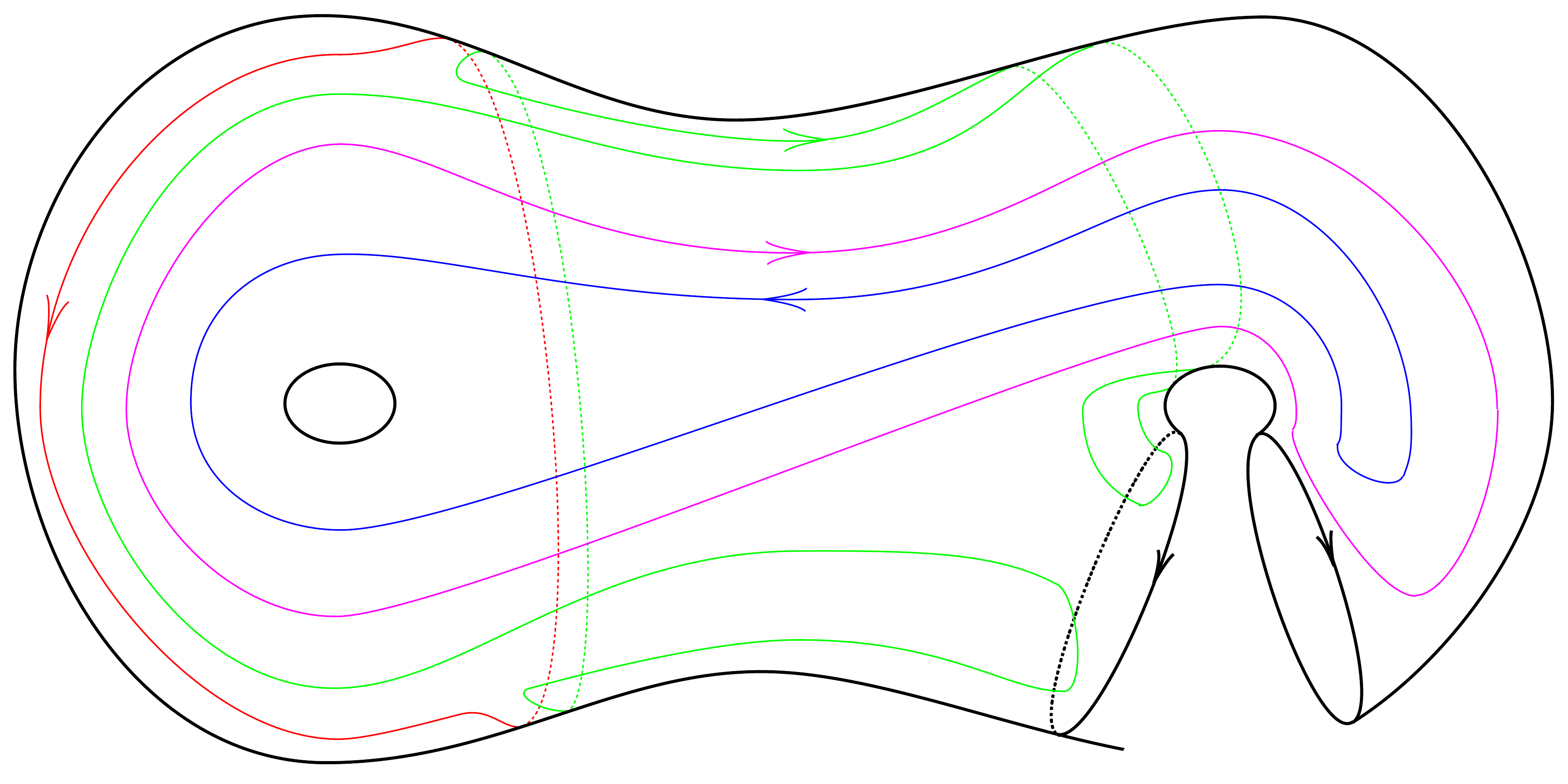}
%\put(93,7){$\eta_1$}
\end{overpic}
\vspace{-0.0in}
\caption{The diagram of $(V,\ga_2'')$.}\label{fig: heegaard diagram, 5}
\end{minipage}
\end{figure}
From \cite[Proposition 1.4]{li2019direct}, we know that
$$\shi(V,\ga_2'')\cong \mathbb{C}^2.$$
Hence we conclude that
\begin{equation}\label{eq_3: heegaard diagram, 2}
    \shi(-H,-\ga_2)\cong\mathbb{C}^2.
\end{equation}

{\bf Step 3}. We can perform a second bypass along the arc $\eta_2$ as shown in Figure \ref{fig: heegaard diagram, 2} on $(H,\ga_1)$ and obtain an exact triangle
\begin{equation}\label{eq_3: heegaard diagram, 3}
\xymatrix@R=6ex{
&\shi(-H,-\ga_1)\ar[dr]&\\
\shi(-H,-\ga_4)\ar[ur]&&\shi(-H,-\ga_3)\ar[ll]^{\psi_{4}^3}\\
}
\end{equation}

It is straightforward to check that both $\ga_3$ and $\ga_4$ intersect the disk $D$ at four points, so we can compute in the same way as we did for $\shi(-H,-\ga_2)$
\begin{equation}\label{eq_3: heegaard diagram, 4}
    \shi(-H,-\ga_3)\cong\mathbb{C}^3,~{\rm and~}\shi(-H,-\ga_4)\cong\mathbb{C}^5.
\end{equation}

From (\ref{eq_3: heegaard diagram, 0}), (\ref{eq_3: heegaard diagram, 1}), (\ref{eq_3: heegaard diagram, 2}), (\ref{eq_3: heegaard diagram, 3}), and (\ref{eq_3: heegaard diagram, 4}), we know that
$${\rm dim}_{\mathbb{C}}I^{\sharp}(-Y)\leq 10.$$
\end{exmp}

\subsection{The instanton knot homology of $(1,1)$-knots}\label{subsec: $(1,1)$-knots}
\quad

In this subsection, we use Theorem \ref{thm_1: from heegaard diagram to SHI} to prove Theorem \ref{thm_1: KHI of 1,1 knots}.
\bdefn
Suppose $p,q\in\mathbb{Z}$ satisfy $p\ge 1,0\le q<p$ and $\operatorname{gcd}(p,q)=1$. Let $\ti{\al}$ and $\ti{\be}$ be two straight lines in $\mathbb{R}^2$ passing the origin with slopes 0 and $p/q$, respectively, and let $r:\mathbb{R}^2\rightarrow T^2$ be the quotient map induced by $(x,y)\to (x+m,y+n)$ for $m,n\in\mathbb{Z}$. Suppose
$\al=r(\ti{\al})~{\rm and}~ \be= r(\ti{\be}).$ Then the manifold compatible with the Heegaard diagram $(T^2,\al,\be)$ is called a \textbf{lens space} and is denoted by $L(p,q)$. Furthermore, the Heegaard diagram $(T^2,\al,\be)$ is called the \textbf{standard diagram} of the lens space. In particular, we regard $S^3$ as a lens space $L(1,0)$.

The lens space is oriented so that the orientation on the $\al$-handlebody is induced from the standard embedding of $S^1\times D^2$ in $\mathbb{R}^3$. With this convention, the lens space $L(p,q)$ comes from the $p/q$-surgery on the unknot in $S^3$.
\edefn
\bdefn
A proper embedded arc $\eta$ in a handlebody $H$ is called a \textbf{trivial arc} if there is an embedded disk $D\subset H$ satisfying $\partial D=\eta \cup (D\cap \partial H)$. The disk $D$ is called the \textbf{cancelling disk} of $\eta$. A knot $K$ in a closed 3-manifold $Y$ admits a \textbf{(1,1)-decomposition} if the followings hold.
\begin{enumerate}[(1)]
    \item $Y$ admits a splitting $Y= H_1 \cup_{T^2}H_2$ so that $H_1\cong H_2\cong S^1\times D^2$.
    \item $K \cap  H_i$ is a properly embedded trivial arc in $H_i$ for $i\in\{1,2\}$.
\end{enumerate}
In this case, $Y$ is either a lens space or $S^1\times S^2$. A knot $K$ admitting a (1,1)-decomposition is called a \textbf{$(1,1)$-knot}.
\edefn
\bprop[{\cite[Section 6.2]{Rasmussen2005} and \cite[Section 2]{Goda2005}}]\label{oneonepara}
For $p,q,r,s\in\mathbb{N}$ satisfying $2q+r\le p$ and $s<p$, a (1,1)-decomposition of a knot determines and is determined by a doubly-pointed diagram. After isotopy, such a diagram becomes $(T^2,\al,\be,z,w)$ in Figure \ref{fig: from pqrs to 11 diagram}, where $p$ is the total number of intersection points, $q$ is the number of strands around either basepoint, $r$ is the number of strands in the middle band, and the $i$-th point on the right-hand side is identified with the $(i + s)$-th point on the left-hand side.
\eprop
\begin{figure}[htbp]
\centering
\includegraphics[width=6cm]{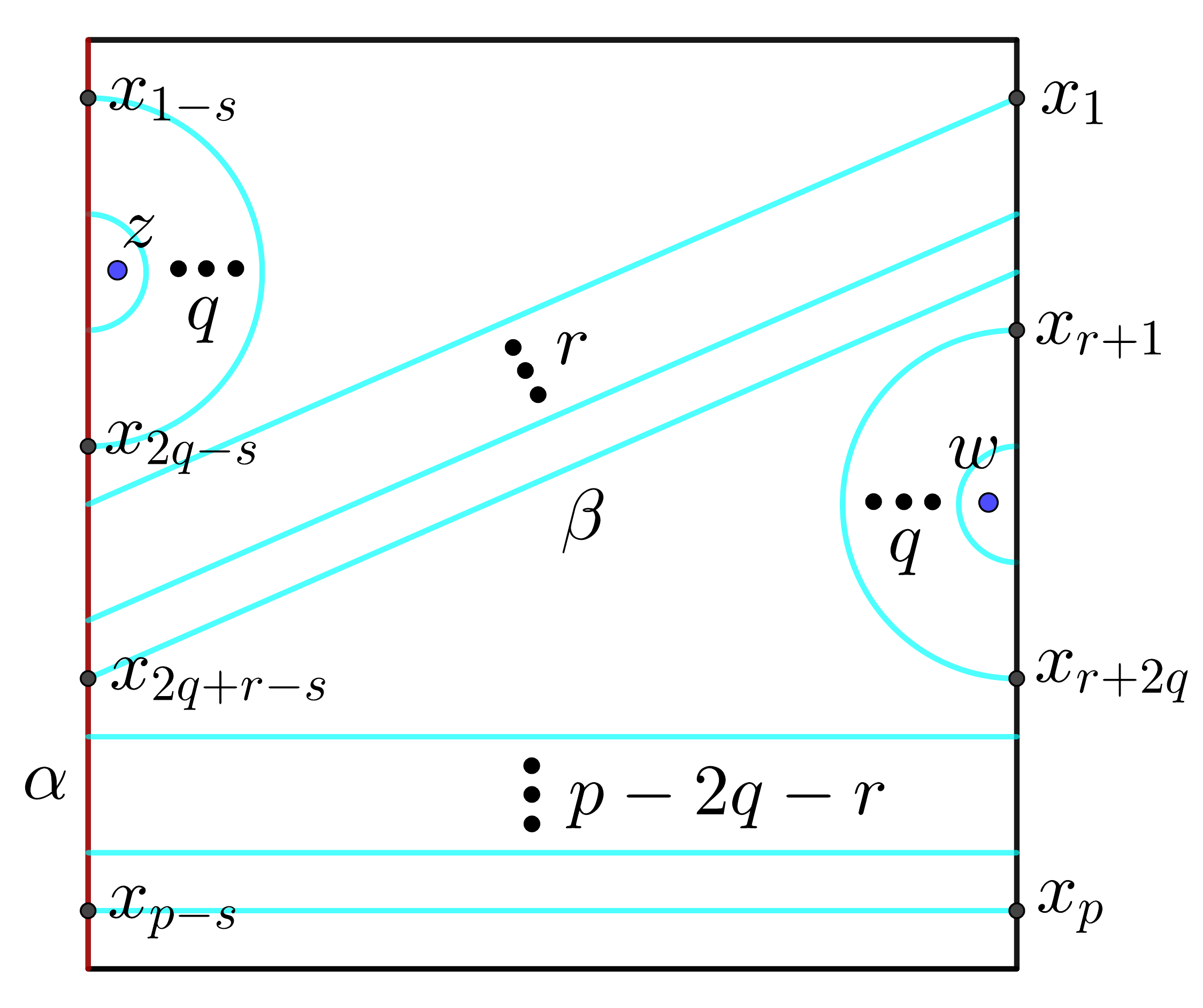}
\caption{(1,1)-diagram.}\label{fig: from pqrs to 11 diagram}
\end{figure}
\bdefn\label{defn: reduced}
A simple closed curve $\be$ on $(T^2,\al,z,w)$ is called {\bf reduced} if the number of intersection points between $\al$ and $\be$ is minimal. The doubly-pointed diagram in Figure \ref{fig: from pqrs to 11 diagram} is called the \textbf{(1,1)-diagram} of type $(p,q,r,s)$, which is denoted by $W(p,q,r,s)$. Strands around basepoints are called \textbf{rainbows} and strands in the bands are called \textbf{stripes}.

If the \textbf{(1,1)-diagram} of $W(p,q,r,s)$ is a Heegaard diagram for some parameters $(p,q,r,s)$, or equivalently, $\be$ has one component and represents a nontrivial homology class in $H_1(T^2)$, then the corresponding knot is also denoted by $W(p,q,r,s)$.

A $(1,1)$-knot whose (1,1)-diagram does not have rainbows is called a \textbf{simple knot} (\textit{c.f.} \cite[Section 2.1]{Rasmussen2007}). For simple knots, let $K(p,q,k)=W(p,0,k,q)$.
\edefn
\begin{prop}\label{mirror}
The mirror knot of a $(1,1)$-knot $W(p,q,r,s)$ is $$W(p,q,p-2q-r,p-s+2q).$$
\end{prop}
\begin{proof}
The Heegaard diagram of the mirror knot of $W(p,q,r,s)$ is obtained by the (1,1)-diagram of $W(p,q,r,s)$ by vertical reflection. We redraw the Heegaard diagram so that the lower band becomes the middle band and the middle band becomes the lower band. This proposition follows from the definition.
\end{proof}
According to \cite[Section 3]{Goda2005} (also \cite[Section 6]{ozsvath2004holomorphicknot}), for the $\widehat{HFK}$ of a $(1,1)$-knot, the generators of the chain complexes are intersection points of $\al$ and $\be$ in the (1,1)-diagram and there is no differential. Thus, the following proposition holds.
\bprop[]\label{prop: rk for HFK}
For a $(1,1)$-knot $K=W(p,q,r,s)$ in $Y$, we have $$\hfk(Y,K)\cong \intg^p.$$
\eprop
We restate Construction \ref{cons: heegaard diagram gives rise to sutured handlebody} more carefully.
\bcons\label{cons: 11 sutured handlebody}
Suppose $(T^2,\al,\be,z,w)$ is the (1,1)-diagram of $W(p,q,r,s)$. We construct a sutured handlebody $(H,\ga)$ as follows, called the {\bf (1,1)-sutured-handlebody} of $W(p,q,r,s)$.

\begin{enumerate}[(1)]
    \item Let $\Sigma$ be the genus-two boundary of the manifold obtained from $[-1,1]\times T^2$ by attaching a 3-dimensional 1-handle along $\{1\}\times\{z,w\}$. For simplicity, when drawing the diagram, the attached 1-handle will still be denoted by two basepoints $z$ and $w$.
    \item Let $\al_1$ and $\be_1$ denote the curves on $\Sigma$ induced from $\al$ and $\be$, respectively. Let $\be$ be oriented so that the innermost rainbow around $z$ is oriented clockwise, which induces an orientation of $\be_1$. If there is no rainbow, let $\be$ be oriented so that each stripe goes from left to right in Figure \ref{fig: from pqrs to 11 diagram}.
    \item Consider the straight arc connecting $z$ to $w$ in Figure \ref{fig: from pqrs to 11 diagram}. It induces a simple closed curve $\al_2$ on $\Sigma$ by going along the 1-handle. Let $\be_{2}$ be the curve on $\Sigma$ induced by a small circle around $z$, oriented counterclockwise.
    \item Let $\ga_1$ and $\ga_2$ be obtained by pushing off $\be_1$ and $\be_2$ to the right with respect to the orientation. Suppose they are oriented reversely with respect to $\be_1$ and $\be_2$, respectively. Let $a_0$ be a straight arc connecting the innermost rainbow of $\be$ around $z$ to the above small circle. It induces an arc connecting $\ga_1$ to $\ga_2$, still denoted by $a_0$. Let $\ga_3$ be obtained by a band sum of $\ga_1$ and $\ga_2$ along $a_1$, with the induced orientation.
    \item Let $H$ be the handlebody compatible with the diagram $(\Sigma, \{\al_1,\al_2\},\emptyset)$ and let $$\ga=\ga_1\cup \ga_2\cup \ga_3.$$Rainbows and stripes are defined similarly for sutures.
\end{enumerate}
\econs

The main goal is to prove the following theorem.
\bthm\label{thm: less1}
Suppose $(H,\ga)$ is the (1,1)-sutured-handlebody of $W(p,q,r,s)$ constructed in Construction \ref{cons: 11 sutured handlebody}. Then we have
$$\dim_{\mathbb{C}}\shi(-H,-\ga)\leq p.$$
\ethm
Before proving this theorem, we first use it to derive Theorem \ref{thm_1: KHI of 1,1 knots}.
\begin{proof}[Proof of Theorem \ref{thm_1: KHI of 1,1 knots}]
Combining Theorem \ref{thm_1: from heegaard diagram to SHI},  Proposition \ref{prop: rk for HFK}, and Theorem \ref{thm: less1}, for a $(1,1)$-knot $K=W(p,q,r,s)$ in a lens space $Y$, we have
$$\dim_{\mathbb{C}}KHI(-Y,K)\leq \dim_{\mathbb{C}}\shi(-H,-\ga)\leq p= {\rm rk}_{\mathbb{Z}}\widehat{HFK}(-Y,K).$$Then the theorem follows from Proposition \ref{mirror}, \textit{i.e.} the mirror knot of a $(1,1)$-knot is still a $(1,1)$-knot with the same intersection number $p$.
\end{proof}

\begin{proof}[Proof of Theorem \ref{thm: less1}]
We prove the theorem by induction on $p$ for any (1,1)-diagram of $W(p,q,r,s)$ where $\be$ has only one component. This includes the case that $\be$ represents a trivial homology class. The induction is based on the bypass exact triangle in Theorem \ref{thm_2: bypass exact triangle on general sutured manifold}. We will show three balanced sutured manifolds in the bypass exact triangle are all (1,1)-sutured handlebodies, where one is the (1,1)-sutured handlebody we want and the other two are (1,1)-sutured handlebodies with smaller number $p$. By straightforward algebra, if the dimension inequality holds for two terms in the bypass exact triangle, then it also holds for the third term.

For the base case, consider $p=1$. The curves $\al_1,\al_2,\be_1,\be_2$ in Construction \ref{cons: 11 sutured handlebody} satisfy$$|\al_1\cap\be_1|=|\al_2\cap \be_2|=1.$$It is straightforward to check $(H,\ga)$ is a product sutured manifold, so is $(-H,-\ga)$. Then Theorem \ref{product} implies $$\dim_\mathbb{C} SHI(-H,-\ga)=1.$$

Now we deal with the case where $p>1$. In Construction \ref{cons: 11 sutured handlebody}, the innermost rainbow around $z$, if exists, is oriented clockwise. Suppose $\delta_1$ is either the innermost rainbow around $z$, or a stripe that is closest to $z$ with $z$ on its right-hand side. Suppose $\delta_2$ is another rainbow or stripe that is closest to $\delta_1$ and is to the left of $\delta_1$. See Figure \ref{fig: local picture} for all possible cases. Compared to Figure \ref{fig: from pqrs to 11 diagram}, we have rotated the square counterclockwise by 90 degrees for the purpose of a better display.

\begin{figure}[ht]
\centering
\includegraphics[width=0.9\textwidth]{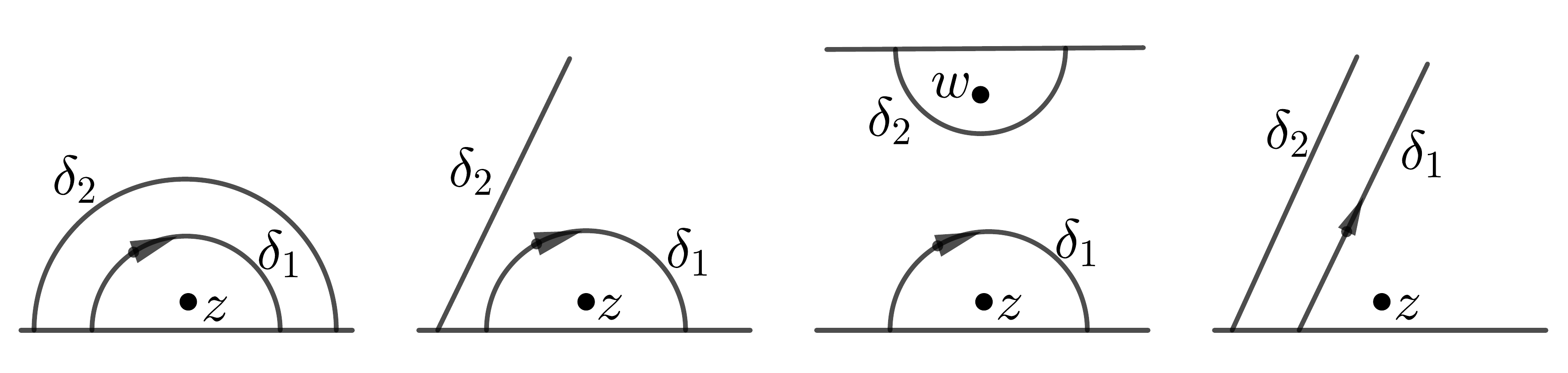}
\caption{Several cases of $\delta_1$ and $\delta_2$.}
\label{fig: local picture}
\end{figure}

We consider two different cases about the orientation of $\delta_2$.

{\bf Case 1}. Suppose $\delta_1$ and $\delta_2$ are oriented parallelly.

We use $W(6,2,1,3)$ shown in Figure \ref{fig: 6213} as an example to carry out the proof, and the general case is similar. In this example, two innermost rainbows around $z$ are oriented parallelly. By construction, the curve $\ga_3$ is parallel (regardless of orientations) to $\ga_1$ outside the neighborhood of the band-sum arc $a_0$. Thus, there exists a unique rainbow of $\ga_3$ between $\delta_1$ and $\delta_2$ around $z$. Let $a_1$ be an anti-wave bypass arc cutting these three rainbows, as shown in Figure \ref{fig: 6213}. Suppose $\ga'$ and $\ga''$ are the other two sutures involved in the bypass triangle associated to $a_2$.

\begin{figure}[ht]
\centering
\includegraphics[width=0.8\textwidth]{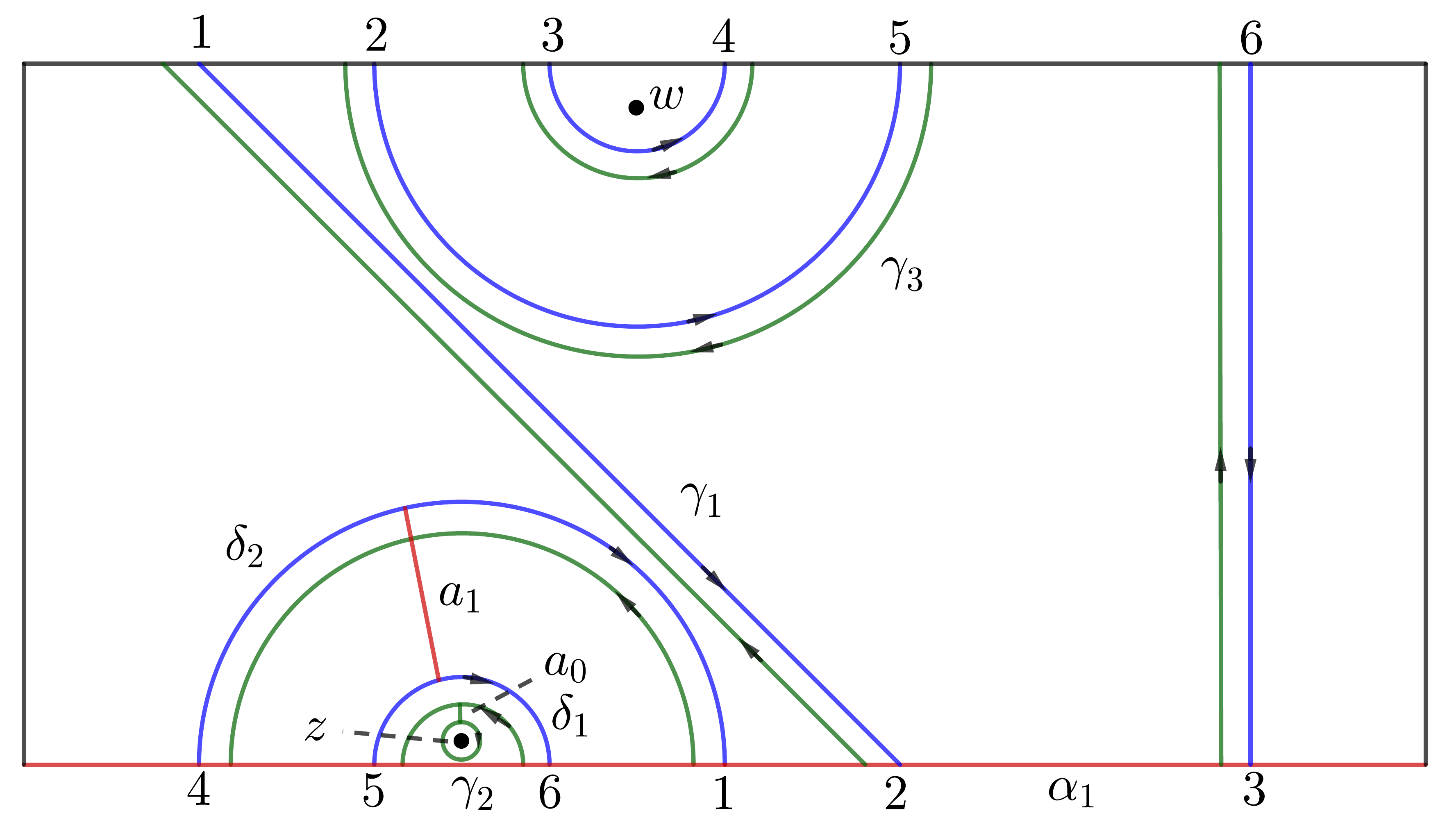}
\caption{The suture related to $W(6,2,1,3)$ and the anti-wave bypass arc.}
\label{fig: 6213}
\end{figure}
\begin{figure}[ht]
\centering
\includegraphics[width=0.8\textwidth]{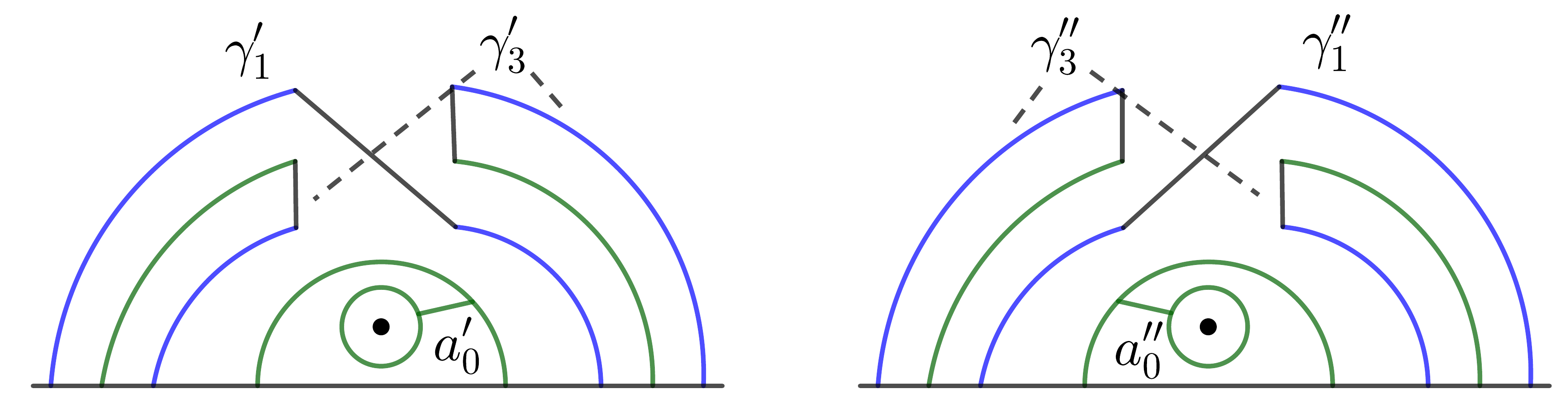}
\caption{Local diagrams after bypass attachments.}
\label{fig: 6213b}
\end{figure}

From Proposition \ref{prop: honda}, we can describe the sutures $\ga^\p$ and $\ga^\pp$ as follows. First, let $\ga_1^1$ and $\ga_1^2$ be two components of $\ga_1\backslash\partial a_1$. Suppose
$$\ga_1^\p=\ga_1^1\cup a_1~{\rm and}~\ga_1^\pp=\ga_1^2\cup a_1$$
as shown in Figure \ref{fig: 6213b}. Second, $\ga^\p$ is obtained from $\ga$ by a Dehn twist along $\ga_1^\pp$, and $\ga^\pp$ is obtained from $\ga$ by a Dehn twist along $\ga_1^\p$.

There is a more direct way to describe $\ga'$ and $\ga''$. First, note that the suture $\ga_2$ is disjoint from both Dehn-twist curves $\ga_1'$ and $\ga_1''$, so $\ga_2$ remains the same in $\ga'$ and $\ga''$. Second, it is straightforward to check the result of $\ga_1$ under the Dehn twist along $\ga_1''$ is $\ga_1'$, and the result of $\ga_1$ under the Dehn twist along $\ga_1'$ is $\ga_1''$. Thus, $\ga_1'$ is a component of $\ga'$ and $\ga_1''$ is a component of $\ga''$.

To figure out the image $\ga_3'$ of $\ga_3$ under the Dehn twist along $\ga_1''$, we first observe that we can isotop the band-sum arc $a_0$ to a new position $a_0'$ such that its endpoints $\partial a_0'$ lie on $\ga_1'\cap\ga_1$ and $\ga_2$, as shown in the left subfigure of Figure \ref{fig: 6213b}. Thus, the facts that $a_0'$ is disjoint from $\ga_1''$ and that $\ga_1'$ is the image of $\ga_1$ under the Dehn twist along $\ga_1''$ imply that performing a Dehn twist along $\ga_1''$ and performing the band sum along $a_0'$ commute with each other. Thus, we conclude that $\ga_3'$ can be obtained from a band sum on $\ga_1'$ and $\ga_2$ along the arc $a_0'$. Similarly we can describe the image $\ga_3''$ of $\ga_3$ under the Dehn twist along $\ga_1'$. Thus, we have described the sutures
$$\ga'=\ga_1'\cup\ga_2\cup\ga_3'~{\rm and}~\ga''=\ga_1'\cup\ga_2\cup\ga_3''$$
explicitly, and it follows that $(H,\ga')$ and $(H,\ga'')$ are both (1,1)-sutured handlebodies. Suppose they are associated to $W(p',q',r',s')$ and $W(p'',q'',r'',s'')$, respectively.

From the above description, both $\ga_1'$ and $\ga_1''$ are reduced. We have \[p'+p''=|\ga_1^\p\cap \al_1|+|\ga_1^\pp\cap \al_1|=|\ga_1\cap \al_1|=p.\]
Thus, the induction applies.

{\bf Case 2}. Suppose $\delta_1$ and $\delta_2$ are oriented oppositely.

An example $W(10,3,1,5)$ is shown in Figure \ref{fig: 10315}. By construction, there is a rainbow of $\ga_3$ to the right of $\delta_2$. Let $a_2$ be a wave bypass arc cutting $\delta_1,\delta_2$, and this rainbow as shown in Figure \ref{fig: 10315}. Suppose $\ga'$ and $\ga''$ are the other two sutures involved in the bypass triangle associated to $a_2$, respectively.
\begin{figure}[ht]
\centering
\includegraphics[width=0.8\textwidth]{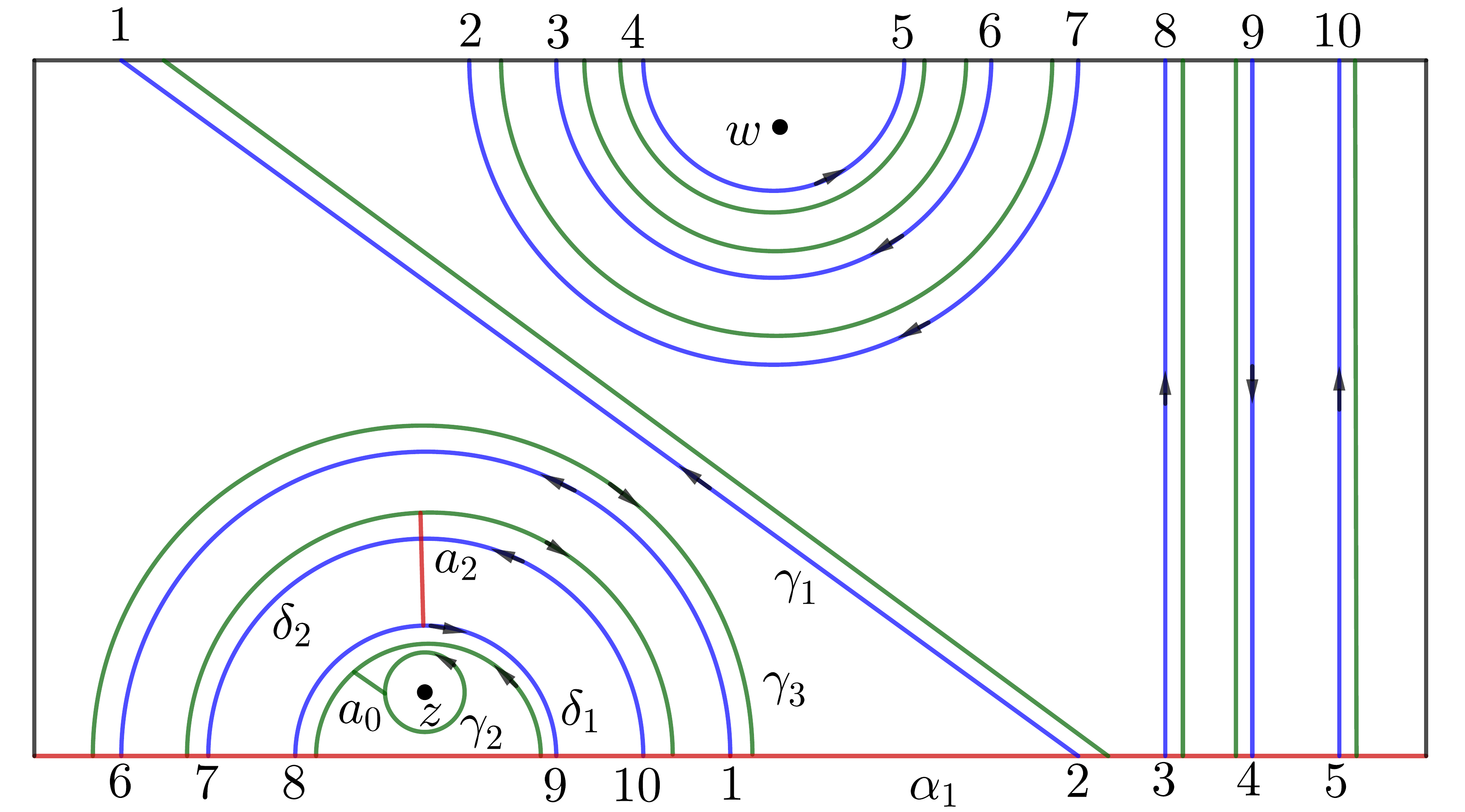}
\caption{The suture related to $W(10,3,1,5)$ and the wave bypass arc.}
\label{fig: 10315}
\end{figure}

\begin{figure}[ht]
\centering
\includegraphics[width=0.8\textwidth]{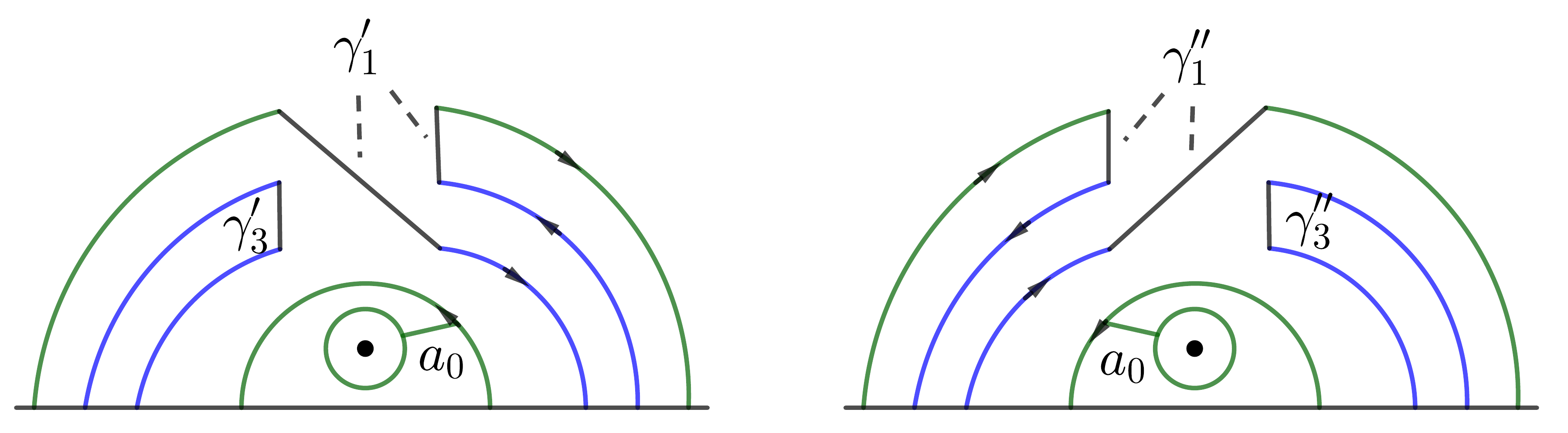}
\caption{Local diagrams after bypass attachments.}
\label{fig: 10315b}
\end{figure}

\begin{figure}[ht]
\centering
\includegraphics[width=0.8\textwidth]{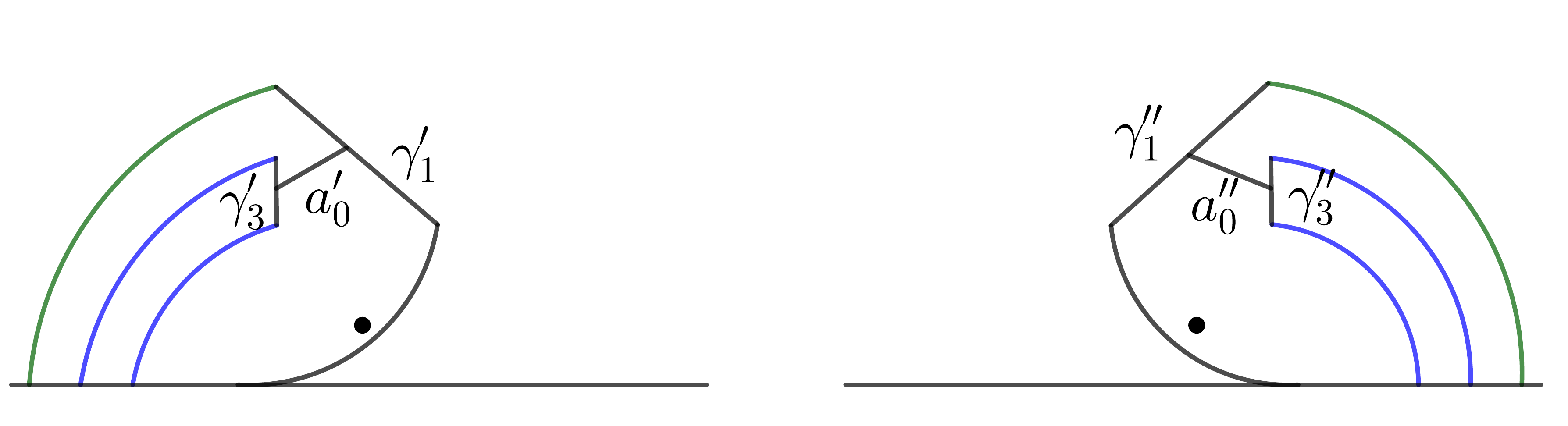}
\caption{Local diagrams after isotopy.}
\label{fig: 10315c}
\end{figure}

To describe the sutures $\ga'$ and $\ga''$ more explicitly, note that the arc $a_2$ cuts $\ga_1$ into two parts $\ga_1^1$ and $\ga_1^2$. Suppose that near $a_2$, $\ga_1^1$ is to the left of $a_2$ and $\ga_1^2$ is to the right of $a_2$. For $\ga'$, it consists of three components:
$$\ga'=\ga_1'\cup\ga_2\cup\ga_3',$$
where $\ga_2$ is as before, $\ga_1'$ is obtained by cutting $\ga_3$ open by $a_2$ and gluing it to $\ga_1^2$, and $\ga_3'$ is obtained by gluing a copy of $a_2$ to $\ga_1^1$. They are depicted as in the left subfigure of Figure \ref{fig: 10315b}. Note that the curve $\ga_1'$ is not reduced. We can isotop the curve along the arc $\ga_1^2$ into a reduced curve. The orientations of curves imply this reduced curve is depicted as in the left subfigure of Figure \ref{fig: 10315c}. Note that $\ga_3'$ is also not reduced. However, from Figure \ref{fig: 10315c} it is straightforward to check that $\ga_3'$ can be thought of as obtained from $\ga_2$ and $\ga_1'$ by a band sum along the arc $a_0'$. Also, it is clear that
$$|\ga_1'\cap \alpha_1|=|\ga_1^1\cap \alpha_1|.$$

Similarly, $\ga''$ consists of three components:
$$\ga''=\ga_1''\cup\ga_2\cup\ga_3'',$$
where $\ga_2$ is as before, $\ga_1''$ is obtained by cutting $\ga_3$ open by $a_2$ and gluing it to $\ga_1^1$, and $\ga_3''$ is obtained by gluing a copy of $a_2$ to $\ga_1^2$. They are depicted as in the right subfigure of Figure \ref{fig: 10315b}. Considering the orientations, we can isotop $\ga_2''$ along $\ga_2^1$ to the position shown in the right subfigure of Figure \ref{fig: 10315c}. Then $\ga_3''$ can be thought of as obtained from $\ga_1$ and $\ga_2''$ by a band sum along $a_0''$. Also,
$$|\ga_1''\cap \alpha_1|=|\ga_1^2\cap \alpha_1|.$$

Hence we conclude that $(H,\ga')$ and $(H,\ga'')$ are both (1,1)-sutured-handlebodies, and
$$|\ga_2'\cap \alpha_1|+|\ga_2''\cap \alpha_1|=|\ga_{2,1}\cap \alpha_1|+|\ga_{2,2}\cap \alpha_1|=|\ga_2\cap \alpha_1|.$$
Thus, the induction applies.

\end{proof}
\subsection{Large surgeries on simple knots}\label{subsec: large surgery}
\quad

In this subsection, we generalize the idea about the anti-wave bypass arc in Case 1 of the proof of Theorem \ref{thm: less1} to prove Theorem \ref{surgery}.

For a simple knot $K=K(p,q,k)\subset Y$ defined as in Definition \ref{defn: reduced}, consider its (1,1)-diagram $(T^2,\al_1,\be_1,z,w)$ and the Heegaard diagram $(\Sigma,\{\al_1,\al_2\},\{\be_1,\be_2\})$ of $Y$ from Construction \ref{cons: doubly-pointed}. Suppose $\al=\{\al_1,\al_2\}$ and $m=\be_2$. Let $l$ be an arc connecting $z$ to $w$ in $T^2-\be$, which induces a curve on $\Sigma$, still denoted by $l$. Then $(\Sigma,\al,\{\be_1\})$ is a diagram of the knot complement $Y(K)$ and $(m,l)$ forms a basis of $H_1(\partial Y(K))$.

Given $p,q\in \mathbb{Z}$ such that $
{\rm gcd}(p,q)=1$, let $\be_2^\p$ be the curve on $\Sigma$ obtained by resolving intersection points of $|q|$ parallel copies of $m$ and $|p|$ parallel copies of $l$. Then $(\Sigma,\al,\{\be_1,\be_2^\p\})$ is a Heegaard diagram of the manifold obtained from the $q/p$-surgery on $K$.

\brem
There are two ways of resolutions, namely the positive resolution and the negative resolution. The choice of the ways depends on orientations of $\Sigma$ and curves and signs of $p$ and $q$. However, the goal in Theorem \ref{surgery} is for any large enough surgery slope without regard for the sign. So the choice here is not important.
\erem

\bdefn[{\cite[Section 7]{Greene2016}}]\label{defn: antiwave}
An arc $a$ in a Heegaard diagram $(\Sigma,\al,\be)$ is called an \textbf{anti-wave} if it satisfies the following conditions.
\begin{enumerate}[(1)]
    \item It is properly embedded in a component $R$ of $\Sigma\backslash(\al\cup\be)$.
    \item Its endpoints lie on the interior of distinct arcs $r_1,r_2$ of $\partial R$, where $r_1$ and $r_2$ are subsets of the same curve $\al_i\subset \al$ or $\be_i\subset \be$ for some $i$.
    \item The local signs of intersection at two endpoints are the same.
\end{enumerate}
\edefn
For the simple knot $K$, there exist anti-waves in each component of $T^2\backslash\al\cup \be$. It is possible to choose $l$ so that $l\cap R=\emptyset$ for some component $R$ of $T^2\backslash(\al\cup\be)$. Let $a$ be an anti-wave in $R$. It induces an anti-wave in $(\Sigma,\al,\{\be_1,\be_2^\p\})$ which is still denoted by $a$.

\blem\label{surgery4}
Given the surgery slope $q/p$, consider the Heegaard diagram $(\Sigma,\al,\{\be_1,\be_2^\p\})$ and the anti-wave $a$ defined as above. Let $\be_1^1$ and $ \be_1^2 $ denote two arcs of $\be_1\backslash \partial a$ and let $\be_{1,i}=\be_1^i\cup a$ for $i\in\{1,2\}$. Suppose $Y_i$ is the manifold compatible with the Heegaard diagram $(\Sigma,\al,\{\be_{1,i},\be_2^\p\})$ and $K_i$ is the core knot of $\be_2^\p$. Suppose $(Y_0,K_0)=(Y,K)$. Then there exists a knot $K^\prime\subset Y$ such that $(Y_i,K_i)$, for $i=0,1,2$, satisfy the exact triangle associated to $K^\prime$ in Theorem \ref{thm: Scaduto's exact triangle}.
\elem
\begin{proof}
Consider neighborhoods $N(\be_1)$ and $N(\be_2^\p)$ of $\be_1$ and $\be_2^\p$ on $\Sigma$, respectively. The manifold
$$\Sigma\backslash {\rm int}(N(\be_1))\cup {\rm int}(N(\be_2^\p))$$
is diffeomorphic to $S^2\backslash\bigcup_{i=1}^4 D_i$, where $D_i$ are pair-wise disjoint disks. Suppose $\partial D_1$ and $\partial D_2$ are images of $\partial N(\be_2')$ under the diffeomorphism and $a$ becomes an arc connecting $\partial D_1$ to $\partial D_2$. There exists a curve $\be_3\subset S^2$ separating $D_1\cup D_2$ and $D_3\cup D_4$. It induces a null-homologous curve on $\Sigma$ which is still denoted by $\be_3$. By construction, $\be_3$ is disjoint from $\be_1,\be_2^\p$ and $\be_{1,i}$ for $i\in\{1,2\}$.

The 3-manifold compatible with the diagram $(\Sigma,\al,\{\be_3\})$ has two toroidal boundary components. It can be regarded as the complement of $K$ and $K^\prime$ in $Y$ for some knot $K^\prime$. Equivalently, $K$ is the core knot of $\be_2$ and $K^\prime$ is the core knot of $\be_1$ and the way how $K$ and $K^\p$ is linked is determined by $\be_3$. After isotopy, $\be_2,\be_{2,1}$ and $\be_{2,2}$ intersect pair-wise at one point. Then Theorem \ref{thm: Scaduto's exact triangle} applies to this case for $K^\p$.
\end{proof}
\blem\label{surgery5}
There exists $N_0>0$ so that for any surgery slope $q/p$ with $|q/p|\ge N_0$, the manifolds $Y_i$ for $i=0,1,2$ in Lemma \ref{surgery4} corresponding to $q/p$ satisfy\[|H_1(Y_0)|=|H_1(Y_1)|+|H_1(Y_2)|.\]
\elem
\begin{proof}
For $i\in\{1,2\}$, it is straightforward to check that $\be_{1,i}$ is a reduced curve without rainbows (\textit{c.f.} Definition \ref{defn: reduced}). Suppose orientations of curves are chosen so that\[|\be_{1,i}\cap \al_1|=\be_{1,i}\cdot \al_1=p_i ~{\rm and}~|\be_{1,i}\cap \al_2|=\be_{1,i}\cdot \al_2=k_i.\]By construction, we have \[p_1+p_2=p_0~{\rm and}~k_1+k_2=k_0.\]It is clear that $K_i$ is the dual knot of some surgery on $K(p_i,q_i,k_i)$ for some $q_i$. Suppose \[l\cdot \al_1=x~{\rm and}~l\cdot \al_2=y.\]Suppose the orientation of $m$ is chosen so that $m\cdot \al_2=1$. Then we have \[\be_2^\p\cdot \al_1=px,{\rm and}~\be_2^\p\cdot \al_2=q+py.\] Thus, for $i\in\{1,2\}$, we have\[|H_1(Y_0)|=|pk_0-(q+py)p_0| ~{\rm and}~|H_1(Y_i)|=|pk_i-(q+py)p_i|.\] One of these orders is the sum of other two. If $|q/p|$ is large enough, the order \[|H_1(Y_0)|=p|k_0-(q/p+y)p_0|\] is greater than $|H_1(Y_i)|$ for $i\in\{1,2\}$. Hence we conclude the lemma.
\end{proof}
\begin{proof}[Proof of Theorem \ref{surgery}]
We prove the theorem by induction on $p_0$ for simple knots $K(p_0,q_0,k_0)$.

For the base case, consider $p_0=1$. The simple knot is the unknot in $S^3$. The dual knot is the core knot of $\be$ for the standard diagram of a lens space, which is also a simple knot (\textit{c.f.} \cite[Section 2.3]{Rasmussen2007}). The theorem follows from Proposition \ref{prop: simple knots}.

When $p_0>1$, consider knots $K_i$ and the related simple knots $K(p_i,q_i,k_i)$ for $i\in\{1,2\}$ in Lemma \ref{surgery4} and Lemma \ref{surgery5}. By induction hypothesis, there exist $N_i$ for $K(p_i,q_i,k_i)$ so that $r$ surgery with $|r|\ge N$ induces an instanton Floer simple knot. Equivalently,\[\dim_\mathbb{C}KHI(Y_1,K_1)=|H_1(Y_1)|~{\rm and}~\dim_\mathbb{C}KHI(Y_2,K_2)=|H_1(Y_2)|.\]

We have to discuss the basis of the homology at first. The basis $(m,l_0)$ of $H_1(\partial Y(K_0))$ is chosen with respect to the anti-wave in the Heegaard diagram corresponding to $K_0$, which induces a basis on $H_1(\partial Y_i(K_i))$. However, in the proofs of Lemma \ref{surgery4} and Lemma \ref{surgery5}, another basis $(m,l_i)$ is chosen with respect to the anti-wave in the Heegaard diagram corresponding to $K_i$. Suppose $l_i=x_im+l$. Then \[qm+pl=(q-px_i)m+pl_i.\]

Suppose $N=\max\{N_1+|x_1|,N_2+|x_2|,N_0\}$. Then Lemma \ref{surgery5} implies
\[|H_1(Y_0)|=|H_1(Y_1)|+|H_1(Y_2)|.\]
Combining Theorem \ref{thm: Scaduto's exact triangle} and
Lemma \ref{surgery4}, we have\[\dim_\mathbb{C}KHI(Y_0,K_0)\le \dim_\mathbb{C}KHI(Y_1,K_1)+\dim_\mathbb{C}KHI(Y_2,K_2).\]Hence we have
\[\dim_\mathbb{C}KHI(Y_0,K_0)\le |H_1(Y_0)|.\]
Combining Theorem \ref{thm_1: from heegaard diagram to SHI} and \cite[Corollary 1.4]{scaduto2015instanton}, we have \[\dim_\mathbb{C}KHI(Y_0,K_0)\ge \dim_\mathbb{C}I^\sharp(Y_0)\ge|H_1(Y_0)|.\]Thus,\[\dim_\mathbb{C}KHI(Y_0,K_0)= |H_1(Y_0)|.\]and the induction applies.
\end{proof}
\brem
For the above proof, the induction on $\dim_\mathbb{C}\shi(-H,-\ga)$ does not work any more because the inequality \[\dim_\mathbb{C} KHI(-Y,K)\le \dim_\mathbb{C}\shi(-H,-\ga)\]might not always be sharp. That is the reason why we switch from bypass exact triangles to surgery exact triangles.
\erem

\section{Instanton Floer homology and the decomposition}
\label{sec: Instanton theory and knots}
\subsection{Basic setups}\label{subsec: basic setup for knots}
\quad

Suppose $Y$ is a closed 3-manifold and $K\subset Y$ is a null-homologous knot. Let $Y(K)$ be the knot complement $Y\backslash {\rm int}( N(K))$. Any Seifert surface $S$ of $K$ gives rise to a framing on $\partial Y(K)$: the longitude $\lambda$ can be picked as $S\cap \partial Y(K)$ with the induced orientation from $S$, and the meridian $\mu$ can be picked as the meridian of the solid torus $N(K)$ with the orientation so that $\mu\cdot \lambda=-1$. The `half lives and half dies' fact for 3-manifolds implies that the following map has a 1-dimensional image:
$$\partial_*: H_2(Y(K),\partial Y(K);\mathbb{Q})\ra H_1(\partial Y(K);\mathbb{Q}).$$
Hence any two Seifert surfaces lead to the same framing on $\partial Y(K)$. We write $g(K)$ for the minimal genus of the Seifert surface of $K$. If a Seifert surface of minimal surface is chosen, we also write it as $g(S)$.

\bdefn\label{defn_4: surgery on pairs}
The framing $(\mu,\lambda)$ defined as above is called the \textbf{canonical framing} of $(Y,K)$. With respect to this canonical framing, let
$$\widehat{Y}_{q/p}=Y(K)\cup_{\phi}S^1\times D^2$$
be the 3-manifold obtained from $Y$ by a $q/p$ surgery along $K$, \textit{i.e.}, $$\phi(\{1\}\times \partial D^2)=q\mu+p\lambda.$$

When the surgery slope is understood, we also write $\widehat{Y}_{q/p}$ simply as $\widehat{Y}$. Let $\widehat{K}$ be the dual knot, \textit{i.e.}, the image of $S^1\times\{0\}\subset S^1\times D^2$ in $\widehat{Y}$ under the gluing map.
\edefn

\begin{conv}
Throughout this section, we will always assume that ${\rm gcd}(p,q)=1$ and $q> 0$ or $(p,q)=(1,0)$ for a Dehn surgery. Especially, the original pair $(Y,K)$ can be thought of as a pair $(\widehat{Y},\widehat{K})$ obtained from $(Y,K)$ by the $1/0$ surgery. Moreover, we will always assume that the knot complement $Y(K)$ is irreducible. This is because if $Y(K)$ is not irreducible, then $Y(K)\cong  Y^\p(K^\p)\sharp Y^\pp$ for some closed 3-manifold $Y^\p,Y^\pp$ and a null-homologous knot $K^\p\subset{Y^\p}$. By the connected sum formula \cite[Section 1.8]{li2018contact}, we have $$\shi(Y(K),\ga)\cong \shi(Y^\p(K^\p),\ga)\otimes I^\sharp(Y^\pp)$$for any suture $\ga$. Hence all results hold after tensoring $I^\sharp(Y^\pp)$.
\end{conv}
% \brem
% The reason to use
% \erem

Next, we describe various families of sutures on the knot complement. Suppose $K\subset Y$ is a null-homologous knot and the pair $(\widehat{Y},\widehat{K})$ is obtained from $(Y,K)$ by a $q/p$ surgery. Note we can identify the complement of $K\subset Y$ with that of $\widehat{K}\subset \widehat{Y}$, \textit{i.e.} $\widehat{Y}(\widehat{K})=Y(K).$

On $\partial{Y(K)}$, there are two framings: One comes from $K$, and we write longitude and meridian as $\lambda$ and $\mu$, respectively. The other comes from $\widehat{K}$. Note only the meridian $\hat{\mu}$ of $\widehat{K}$ is well-defined, and by definition, it is $\hat{\mu}=q\mu+p\lambda.$
% \brem
% The reason using $q$ and $p$ for coefficients of $\mu$ and $\lambda$, respectively, is to be consistent with the definition of `slope' in \cite{honda2000classification}.
% \erem
\bdefn\label{defn_4: Ga_n-hat sutures}
If $p=0$, then $q=1$ and $\hat{\mu}=\mu$. We can take $\hat{\lambda}=\lambda$. If $(q,p)=(0,1)$, then we take $\hat{\lambda}=-\mu$. If $p,q\neq 0$, then we take $\hat{\lambda}=q_0\mu+p_0\lambda$, where $(q_0,p_0)$ is the unique pair of integers so that the following conditions are true.
\benu
\item $0\le |p_0|<|p|$ and $p_0p\le 0$.
\item $0\le |q_0|<|q|$ and $q_0q\le 0$.
\item $p_0q-pq_0=1$.
\eenu
In particular, if $(q,p)=(n,1)$, then $\hat{\lambda}=-\mu$.

For a homology class $x\lambda+y\mu$, let $\ga_{x\lambda+y\mu}$ be the suture consisting of two disjoint simple closed curves representing $\pm(x\lambda+y\mu)$ on $\partial{Y(K)}$. Furthermore, for $n\in\intg$, define
$$\widehat{\Ga}_{n}(q/p)=\ga_{\hat{\lambda}-n\hat{\mu}}=\ga_{(p_0-np)\lambda+(q_0-nq)\mu},~{\rm and}~ \widehat{\Ga}_{\mu}(q/p)=\ga_{\hat{\mu}}=\ga_{p\lambda+q\mu}.$$
Suppose $(q_n,p_n)\in\{\pm(q_0-nq,p_0-np)\}$ such that $q_n\ge 0$. 

When $\hat{\lambda}$ and $\hat{\mu}$ are understood, we omit the slope $q/p$ and simply write $\widehat{\Ga}_n$ and $\widehat{\Ga}_{\mu}$. When $(q,p)=(1,0)$, we write $\Ga_n$ and $\Ga_\mu$ instead.
\edefn

\brem\label{rem: remark on indexing the sutures}
Since the two components of the suture must be given opposite orientations, the notations $\ga_{x\lambda+y\mu}$ and $\ga_{-x\lambda-y\mu}$ represent the same suture on the knot complement $Y(K)$. Our choice makes $q_{n+1}\le q_n$ for $n<-1$ and $q_{n+1}\ge q_n$ for $n\ge 0$.
\erem

Finally, we sketch the proofs of Proposition \ref{prop: I sharp and SHI} and Theorem \ref{thm: torsion spin c decomposition}. The essential arguments are proved in the next two subsections.

\begin{proof}[Proof of Proposition \ref{prop: I sharp and SHI}]
Suppose $\hat{\mu}=q\mu+p\lambda$. The set $\mathcal{G}$ of sutures consists of $-\widehat{\Ga}_n$ for all $n\in\mathbb{N}$ satisfying $q_n>q+2g(K)$, where $g(K)$ is the Seifert genus of $K$. For any $\ga\in\mathcal{G}$. The grading in term (1) is from Theorem \ref{thm_2: grading on SHI}, where the admissible surface $S$ is the Seifert surface of $K$ with minimal genus (up to a stabilization, \textit{c.f.} Definition \ref{defn_4: S tau}).

For $\ga=-\widehat{\Ga}_{n}$ in term (2), the image of $f_{K,\ga}$ is a direct summand of `middle gradings' of $\shi(-Y(K),-\widehat{\Ga}_n)$, which is denoted by $\mathcal{I}_+(-\widehat{Y},\widehat{K})$ in Definition \ref{defn_4: essential component}. The isomorphism $f_\ga$ is defined in Proposition \ref{prop_4: F_n is an isomorphism}. It is the restriction of $F_n$ on the corresponding gradings, where $F_n$ is defined in Lemma \ref{lem_4: surgery triangle for knots} (a special case of Lemma \ref{lem_3: surgery triangle} for knots).

For $\ga_1=-\widehat{\Ga}_{n_1},\ga_2=-\widehat{\Ga}_{n_2}$ in term (3), the isomorphism $g_{K,\ga_1,\ga_2}$ is defined in Lemma \ref{lem_4: psi pm is isomorphism} (see also Remark \ref{rem: grading shift}). It is the restrictions of bypass maps on corresponding gradings.

Term (4) is from commutative diagrams in Lemma \ref{lem_4: surgery triangle for knots}.
\end{proof}
\begin{proof}[Proof of Theorem \ref{thm: torsion spin c decomposition}]
We prove this theorem for $-\widehat{Y}$. Suppose $\mu\subset\partial \widehat{Y}(\widehat{K})$ is a simple closed curve such that $|\mu\cdot \lambda|=1$. Suppose $Y$ is the manifold obtained by Dehn filling along $\mu$ and suppose $K$ is the dual knot in $Y$. By the assumption of the Seifert surface $S$, we know that $K$ is a null-homologous knot in $Y$. Moreover, we know that $(\widehat{Y},\widehat{K})$ is obtained from $Y$ by performing the $q/p$-surgery along $(Y,K)$ with respect to the canonical framing induced by $S$. The choice of $\mu$ is not important since it will only change the integer $p$. Then we can apply the construction in Proposition \ref{prop: I sharp and SHI}. In particular, we can use the term (2) of Proposition \ref{prop: I sharp and SHI} for any $\ga\in\mathcal{G}$ to define the decomposition. Explicitly, we use $\mathcal{I}_+(\widehat{Y},\widehat{K})$ to decompose $I^\sharp(\widehat{Y})$. By term (3) and term (4) of Proposition \ref{prop: I sharp and SHI}, this decomposition is well-defined up to isomorphism.
\end{proof}

\subsection{Bypasses on knot complements}\label{subsec: bypasses on knot complements}
\quad

Suppose $Y$ is a closed 3-manifold and $K\subset Y$ is a null-homologous knot. Let $(\mu,\lambda)$ be the canonical framing on $Y(K)$ in Definition \ref{defn_4: surgery on pairs}. Suppose $y_3/x_3$ is a surgery slope with $y_3\ge 0$. According to Honda \cite[Section 4.3]{honda2000classification}, there are two basic bypasses on the balanced sutured manifold $(Y(K),\ga_{(x_3,y_3)})$, whose arcs are depicted as in Figure \ref{bypass_arc}. The sutures involved in the bypass triangles were described explicitly in Honda \cite[Section 4.4.4]{honda2000classification}.
\begin{figure}[ht]
\centering
\includegraphics[width=0.4\textwidth]{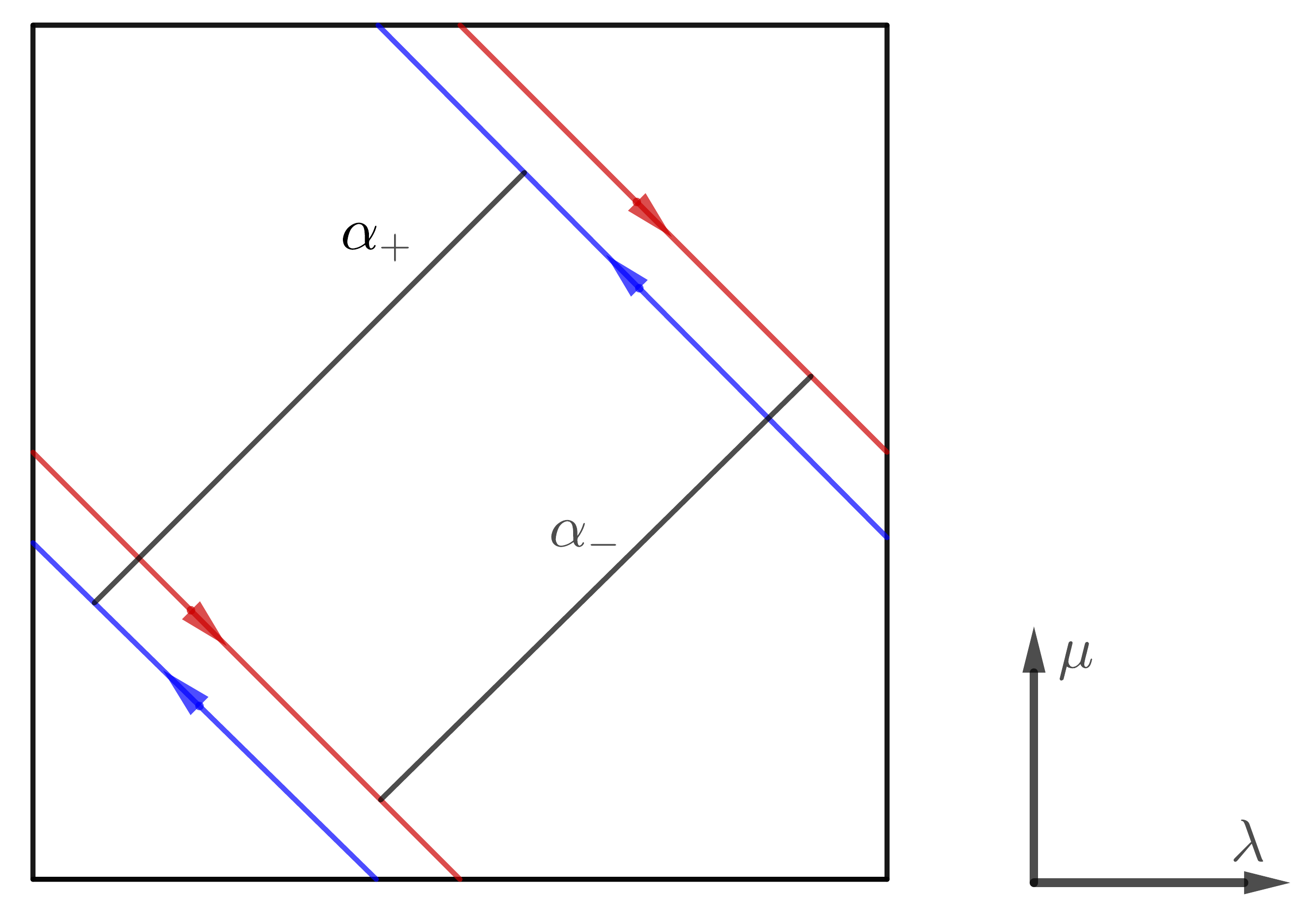}
\caption{Bypass arcs on $\ga_{(1,-1)}$.}
\label{bypass_arc}
\end{figure}
\begin{defn}\label{defn: xi and yi}
For a surgery slope $y_3/x_3$ with $y_3\ge 0$, suppose its continued fraction is
$$
\frac{y_3}{x_3}=[a_0,a_1,\dots,a_n]=a_0-\frac{1}{a_1-\frac{1}{\dots-\frac{1}{a_n}}},$$where integers $a_i< -1$. If $y_3>-x_3> 0$, let
$$\frac{y_1}{x_1}=[a_0,\dots,a_{n-1}]~{\rm and~}\frac{y_2}{x_2}=[a_0,\dots,a_{n}+1].
$$For simplicity, when $a_{i}=-2$ for integer $i\in(k,n]$ and $a_k\neq -2$, we can set $$[a_0,\dots,a_{n}+1]=[a_0,\dots,a_{k}+1].$$If $-x_3>y_3>0$, we do the same thing for $x_3/(-y_3)$. If $y_3>x_3>0$, we do the same thing for $y_3/(-x_3)$. If  $x_3>y_3>0$, we do the same thing for $x_3/(-y_3)$. If $y_3/x_3=1/0$, then set $y_1/x_1=0/1$ and $y_2/x_2=1/(-1)$. If $y_3/x_3=0/1$, then set $y_1/x_1=1/(-1)$ and $y_1/x_1=0/1$.
We always require that $y_1\ge 0$ and $y_2\ge 0$.
\end{defn}
\brem
It is straightforward to use induction to verify that for $y_3>-x_3> 0$,
$$
x_3=x_1+x_2~{\rm and}~y_3=y_1+y_2.$$
\erem
The bypass exact triangle in Theorem \ref{thm_2: bypass exact triangle on general sutured manifold} becomes the following.

\begin{prop}\label{prop_4: bypass exact triangle on knot complement}
Suppose $K\subset Y$ is a null-homologous knot, and suppose the surgery slopes $y_i/x_i$ for $i\in\{1,2,3\}$ are defined as in Definition \ref{defn: xi and yi}. Suppose the indices are considered mod 3. Let $\psi_{+,y_{i+1}/{x_{i+1}}}^{y_i/x_i}$ and $\psi_{-,y_{i+1}/{x_{i+1}}}^{y_i/x_i}$ be bypass maps from two different bypasses, respectively. Then there are two exact triangles related to $\psi_{+,y_{i+1}/{x_{i+1}}}^{y_i/x_i}$ and $\psi_{-,y_{i+1}/{x_{i+1}}}^{y_i/x_i}$, respectively.
\begin{equation*}
\xymatrix@R=6ex{
\shi(-Y(K),-\ga_{(x_2,y_2)})\ar[rr]^{\psi_{\pm,y_3/x_3}^{y_2/x_2}}&&\shi(-Y(K),-\ga_{(x_3,y_3)})\ar[dl]^{\psi_{\pm,y_1/x_1}^{y_3/x_3}}\\
&\shi(-Y(K),-\ga_{(x_1,y_1)})\ar[lu]^{\psi_{\pm,y_2/x_2}^{y_1/x_1}}&
}
\end{equation*}
\end{prop}
\brem
Note that there are two different bypasses, which induce two different exact triangles. However, both of them involve the same set of balanced sutured manifolds.
\erem
\begin{conv}
We will use $\psi^*_{+,*}$ and $\psi^*_{-,*}$ to denote bypass maps with respect to some slopes.
\end{conv}
Next, we describe the bypass exact triangles for $\widehat{\Ga}_n$ and $\widehat{\Ga}_\mu$ in Definition \ref{defn_4: Ga_n-hat sutures}.

\bprop\label{prop_4: bypass exact triangle for Ga_n-hat sutures}
Suppose $K\subset Y$ is a null-homologous knot and suppose the pair $(\widehat{Y},\widehat{K})$ is obtained from $(Y,K)$ by a $q/p$ surgery. Suppose further that the sutures $\widehat{\Ga}_n$ and $\widehat{\Ga}_{\mu}$ are defined as in Definition \ref{defn_4: Ga_n-hat sutures}. Then there are two exact triangles related to $\psi_{+,*}^*$ and $\psi_{-,*}^*$, respectively.
\begin{equation}\label{eq_4: positive bypass}
\xymatrix@R=6ex{
\shi(-Y(K),-\widehat{\Ga}_{n})\ar[rr]^{\psi^{n}_{\pm,n+1}}&&\shi(-Y(K),-\widehat{\Ga}_{n+1})\ar[dl]^{\psi^{n+1}_{\pm,\mu}}\\
&\shi(-Y(K),-\widehat{\Ga}_{\mu})\ar[ul]^{\psi^{\mu}_{\pm,n}}&
}
\end{equation}
\eprop
\begin{proof}
If $\widehat{\Ga}_{n+1}=\ga_{(x_3,y_3)}$ and $y_3>-x_3>0$ and , then it is straightforward to check that
$$\ga_{(x_1,y_1)}=\widehat{\Ga}_{\mu}~{\rm and}~\ga_{(x_2,y_2)}=\widehat{\Ga}_{n},$$where $(x_1,y_1)$ and $(x_2,y_2$ are defined as in Definition \ref{defn: xi and yi}. Then the exact triangles follows from Proposition \ref{prop_4: bypass exact triangle on knot complement}. The similar proof applies to other cases.
\end{proof}
Similar to Lemma \ref{lem_3: surgery triangle}, we have the following proposition.
\blem[{\cite[Section 3]{li2019tau}}]\label{lem_4: surgery triangle for knots}
Suppose $K\subset Y$ is a null-homologous knot and suppose the pair $(\widehat{Y},\widehat{K})$ is obtained from $(Y,K)$ by a $q/p$ surgery. Suppose further that the sutures $\widehat{\Ga}_n$ are defined as in Definition \ref{defn_4: Ga_n-hat sutures}. Then, there is an exact triangle
\begin{equation}\label{eq_2: surgery exact triangle}
\xymatrix@R=6ex{
\shi(-Y(K),-\widehat{\Ga}_{n})\ar[rr]&&\shi(-Y(K),-\widehat{\Ga}_{n+1})\ar[dl]^{F_{n+1}}\\
&\shi(-\widehat{Y}(1),-\delta)\ar[ul]^{G_n}&
}
\end{equation}
where the balanced sutured manifold $(\widehat{Y}(1),\delta)$ is defined as in Remark \ref{rem: SFH}.

Furthermore, we have four commutative diagrams related to $\psi^{n}_{+,n+1}$ and $\psi^{n}_{-,n+1}$, respectively
\begin{equation*}\label{eq_4: commutative diagram, G, +}
\xymatrix@R=6ex{
\shi(-Y(K),-\widehat{\Ga}_{n})\ar[rr]^{\psi_{\pm,n+1}^n}&&\shi(-Y(K),-\widehat{\Ga}_{n+1})\\
&\shi(-\widehat{Y}(1),-\delta)\ar[ul]^{G_n}\ar[ur]_{G_{n+1}}&
}
\end{equation*}
and
\begin{equation*}\label{eq_4: commutative diagram, F, +}
\xymatrix@R=6ex{
\shi(-Y(K),-\widehat{\Ga}_{n})\ar[rd]_{F_n}\ar[rr]^{\psi_{\pm,n+1}^n}&&\shi(-Y(K),-\widehat{\Ga}_{n+1})\ar[ld]^{F_{n+1}}\\
&\shi(-\widehat{Y}(1),-\delta)&
}
\end{equation*}
\elem

The bypass maps in (\ref{eq_4: positive bypass}) behave well under the gradings on $\shi$ associated to the fixed Seifert surface of $K$. To provide more details, let us fix a minimal genus Seifert surface $S$ of $K$.

\begin{conv}
We will always assume by default that the Seifert surface $S$ has minimal possible intersections with any suture $\ga_{(p,q)}$.
\end{conv}

\bdefn\label{defn_4: S tau}
Suppose $K\subset Y$ is a null-homologous knot and $\ga_{(x,y)}$ is a suture on $\partial Y(K)$ with $y\ge 0$. Suppose further that $S$ is a minimal genus Seifert surface of $K$. Let $S^{\tau(y)}$ be a negative stabilization of $S$ if $y$ is even and be the original $S$ if $y$ is odd. When the suture $\ga_{(x,y)}$ is understood, we simply write $S^{\tau}$. More explicitly, we define a map $\tau:\mathbb{N}\mapsto\{0,-1\}$ as
\begin{equation*}
\tau(y)=\left\{
\begin{array}{cl}
    0&y{~\rm is~odd}\\
    -1&y{~\rm is~even}
\end{array}
\right.
\end{equation*}
\edefn
\brem
It is straightforward to check that $S^\tau\subset (M,\ga_{(x,y)})$ is admissible. Note that the negative stabilization is with respect to the suture $\ga_{(x,y)}$ rather than $-\ga_{(x,y)}$. This is important because later we will incorporate this definition of $S^{\tau}$ with the bypass maps, where the orientations of the sutures are reverse (\textit{c.f.} Remark \ref{rem_2: pm switches according to orientations of the suture}).
\erem
\begin{conv}
Note that in Subsection \ref{subsec: dimension inequality for tangles}, we also define another function $\tau$. Since the old definition will no longer be used, and the new tau function serves for the same purpose as the old, we keep using the same notation. From now on, we use the new definition of $\tau$ as in Definition \ref{defn_4: S tau}.
\end{conv}

\blem\label{lem_4: top and bottom non-vanishing gradings}
Suppose $K\subset Y$ is a null-homologous knot and $\ga_{(x,y)}$ is a suture on $\partial Y(K)$ with $y\ge 0$. Suppose further that $S$ is a minimal genus Seifert surface of $K$. Then the maximal and minimal nontrivial gradings of $\shi(-Y(K),-\ga_{(x,y)},S^{\tau})$ are
$$i_{max}=\frac{1}{2}(y-1-\tau(y))+g(S)=\lceil \frac{y-1}{2}\rceil+g(S)$$and$$i_{min}=-\frac{1}{2}(y-1+\tau(y))-g(S)=\lceil -\frac{y-1}{2}\rceil-g(S).$$
\elem

\bpf
The proof is similar to that of Lemma \ref{lem_3: top and bottom nontrivial gradings}, though in the current lemma, we can also identify the top and bottom nontrivial gradings by making use of sutured manifold decompositions. Note that we have assumed that the knot complement $Y(K)$ is irreducible in the convention after Definition \ref{defn_4: surgery on pairs}, and $S$ is a minimal genus Seifert surface of $K$, so the decomposition of $(Y(K),\ga)$ along $S$ and $-S$ are both taut.

When $y$ is odd, we have $S^{\tau}=S$. Then it follows directly from Theorem \ref{thm_2: grading on SHI} that
$$i_{max}=\frac{y-1}{2}+g(S)~{\rm and}~i_{min}=-\frac{y-1}{2}-g(S).$$

When $y$ is even, we have $S^{\tau}=S^-$. Then it follows directly from Theorem \ref{thm_2: grading on SHI} that
$$i_{max}=\frac{y}{2}+g(S).$$
To figure out the grading $i_{min}$, note that
\beq
\shi(-Y(K),-\ga_{(x,y)},S^-,-i_{\max})=&\shi(-Y(K),-\ga_{(x,y)},-(S^-),i_{\max})\\
=&\shi(-Y(K),-\ga_{(x,y)},(-S)^+,i_{\max})\\
=&0\\
\eeq
The last equality follows from Lemma \ref{lem_2: stabilization and decomposition} and the fact that the positive stabilization on $(-S)$ with respect to $\ga_{(x,y)}$ becomes a negative one with respect to $-\ga_{(x,y)}$ (\textit{c.f.} Remark \ref{rem_2: pm switches according to orientations of the suture}).

We also have
\beq
\shi(-Y(K),-\ga_{(x,y)},S^-,1-i_{\max})=&\shi(-Y(K),-\ga_{(x,y)},-(S^-),i_{\max}-1)\\
=&\shi(-Y(K),-\ga_{(x,y)},(-S)^+,i_{\max}-1)\\
=&\shi(-Y(K),-\ga_{(x,y)},(-S)^-,i_{\max})\\
\eeq
The last equality follows from Theorem \ref{thm_2: grading shifting property}. From Lemma \ref{lem_2: stabilization and decomposition} and Theorem \ref{thm_2: grading on SHI}, we have
\beq
\shi(-Y(K),-\ga_{(x,y)},(-S)^-,i_{\max})\cong \shi(-M',-\ga')\neq0,
\eeq
where $(M',\ga')$ is the taut balanced sutured manifold obtained from $(Y(K),\ga_{(x,y)})$ by decomposing along $-S$. Hence we conclude that
$$i_{min}=1-i_{max}=1-\frac{y}{2}-g(S).$$
\epf

\bdefn\label{defn_4: top and bottom non-vanishing gradings}
For any integer $y\in\mathbb{N}$, define
$$i_{max}^y=\lceil \frac{y-1}{2}\rceil+g(S),{~\rm and~}i_{min}^y=\lceil -\frac{y-1}{2}\rceil-g(S).$$
For the suture $\widehat{\Ga}_n=\ga_{(p_n,q_n)}$, define
$$\hat{i}^n_{max}=i_{max}^{q_n}~{\rm and}~\hat{i}^n_{min}=i^{q_n}_{min}.$$
\edefn
\begin{conv}
Note that we use the similar notations as in Subsection \ref{subsec: dimension inequality for tangles}, while from now on, we use the new definitions of $i_{max}^y$ and $i_{min}^y$ as in Definition \ref{defn_4: top and bottom non-vanishing gradings}. We will use $i^*_{max},\hat{i}^*_{max}$ and $i^*_{min},\hat{i}^*_{min}$ to denote the maximal and minimal gradings for the slope specified by $*$.
\end{conv}
In \cite[Section 5]{li2019direct}, a graded version of the bypass exact triangles in Proposition \ref{prop_4: bypass exact triangle on knot complement} is proved, which is similar to Lemma \ref{lem_3: graded bypass triangle}.

% Recall, as in Definition \ref{defn_4: Ga_n-hat sutures}, we have
% $$\widehat{\Ga}_{n}=\ga_{\hat{\lambda}-n\hat{\mu}}=\ga_{(p_n,q_n)},$$
% where $q_n$ is chosen to be always non-negative.

\bprop[{\cite[Proposition 5.5]{li2019direct}}]\label{prop_4: graded bypass for Ga_n-hat sutures}
Suppose $K\subset Y$ is a null-homologous knot and suppose the pair $(\widehat{Y},\widehat{K})$ is obtained from $(Y,K)$ by a $q/p$ surgery. Suppose further that the sutures $\widehat{\Ga}_n$ and $\widehat{\Ga}_{\mu}$ are defined as in Definition \ref{defn_4: Ga_n-hat sutures} and $S$ is a minimal genus Seifert surface of $K$. Then the followings hold. Note that the grading shift notation comes from Definition \ref{defn_3: shifting the grading}.
\benu
\item For $n\in\intg$ so that $q_{n+1}=q_n+q$, \textit{i.e.}, $n\ge 0$, there are two bypass exact triangles:
\begin{equation*}\label{eq_4: positive bypass, graded}
\xymatrix{
\shi(-Y(K),-\widehat{\Ga}_{n},S^{\tau})[\hat{i}_{min}^{n+1}-\hat{i}_{min}^{n}]\ar[rr]^{\quad\quad\psi^{n}_{+,n+1}}&&\shi(-Y(K),-\widehat{\Ga}_{n+1},S^{\tau})\ar[dll]^{\psi^{n+1}_{+,\mu}}\\
\shi(-Y(K),-\widehat{\Ga}_{\mu},S^{\tau})[\hat{i}_{max}^{n+1}-\hat{i}_{max}^{\mu}]\ar[u]^{\psi^{\mu}_{+,n}}&&
}
\end{equation*}
and
\begin{equation*}\label{eq_4: negative bypass, graded}
\xymatrix{
\shi(-Y(K),-\widehat{\Ga}_{n},S^{\tau})[\hat{i}_{max}^{n+1}-\hat{i}_{max}^n]\ar[rr]^{\quad\quad\psi^{n}_{-,n+1}}&&\shi(-Y(K),-\widehat{\Ga}_{n+1},S^{\tau})\ar[dll]^{\psi^{n+1}_{-,\mu}}\\
\shi(-Y(K),-\widehat{\Ga}_{\mu},S^{\tau})[\hat{i}_{min}^{n+1}-\hat{i}_{min}^{\mu}]\ar[u]^{\psi^{\mu}_{-,n}}&&
}
\end{equation*}

\item For $n\in\intg$ so that $q_{n+1}=q_n-q$, \textit{i.e.}, $n< -1$, there are two bypass exact triangles:
\begin{equation*}\label{eq_4: positive bypass, graded, 2}
\xymatrix{
\shi(-Y(K),-\widehat{\Ga}_{n},S^{\tau})\ar[rr]^{\psi^{n}_{+,n+1}\quad\quad}&&\shi(-Y(K),-\widehat{\Ga}_{n+1},S^{\tau})[\hat{i}_{max}^{n}-\hat{i}_{max}^{n+1}]\ar[dll]^{\psi^{n+1}_{+,\mu}}\\
\shi(-Y(K),-\widehat{\Ga}_{\mu},S^{\tau})[\hat{i}_{min}^{n}-\hat{i}_{min}^{\mu}]\ar[u]^{\psi^{\mu}_{+,n}}&&
}
\end{equation*}
and
\begin{equation}\label{eq_4: negative bypass, graded, 2}
\xymatrix{
\shi(-Y(K),-\widehat{\Ga}_{n},S^{\tau})\ar[rr]^{\psi^{n}_{-,n+1}\quad\quad}&&\shi(-Y(K),-\widehat{\Ga}_{n+1},S^{\tau})[\hat{i}_{min}^{n}-\hat{i}_{min}^{n+1}]\ar[dll]^{\psi^{n+1}_{-,\mu}}\\
\shi(-Y(K),-\widehat{\Ga}_{\mu},S^{\tau})[\hat{i}_{max}^{n}-\hat{i}_{max}^{\mu}]\ar[u]^{\psi^{\mu}_{-,n}}&&
}
\end{equation}

\item For $n\in\intg$ so that $q_{n+1}+q_n=q$, \textit{i.e.}, $n=-1$, there are two bypass exact triangles:
\begin{equation}\label{eq_4: positve bypass, graded, 3}
\xymatrix{
\shi(-Y(K),-\widehat{\Ga}_{n},S^{\tau})[\hat{i}_{min}^{\mu}-\hat{i}_{min}^n]\ar[rr]^{\psi^{n}_{+,n+1}}&&\shi(-Y(K),-\widehat{\Ga}_{n+1},S^{\tau})[\hat{i}_{max}^{\mu}-\hat{i}_{max}^{n+1}]\ar[dll]^{\psi^{n+1}_{+,\mu}}\\
\shi(-Y(K),-\widehat{\Ga}_{\mu},S^{\tau})\ar[u]^{\psi^{\mu}_{+,n}}&&
}
\end{equation}
and
\begin{equation}\label{eq_4: negative bypass, graded, 3}
\xymatrix{
\shi(-Y(K),-\widehat{\Ga}_{n},S^{\tau})[\hat{i}_{max}^{\mu}-\hat{i}_{max}^{n}]\ar[rr]^{\psi^{n}_{-,n+1}}&&\shi(-Y(K),-\widehat{\Ga}_{n+1},S^{\tau})[\hat{i}_{min}^{\mu}-\hat{i}_{min}^{n+1}]\ar[dll]^{\psi^{n+1}_{-,\mu}}\\
\shi(-Y(K),-\widehat{\Ga}_{\mu},S^{\tau})\ar[u]^{\psi^{\mu}_{-,n}}&&
}
\end{equation}
\eenu
Furthermore, all maps involved in the above bypass exact triangles are grading preserving.
\eprop
\brem\label{curve invariant}
The above proposition can be understood by Remark \ref{block}. Alternatively, we can understand the above proposition by the following method, which is inspired by the curve invariant introduced by Hanselman, Rasmussen, and Waston \cite{Hanselman2016,Hanselman2018}.
\benu
    \item Consider the lattice $\mathbb{Z}^2\subset \mathbb{R}^2$. A surgery slope $y/x\in \mathbb{Q}\cup\{\infty\}$ corresponds to a straight arc connecting two lattice points in $\mathbb{Z}^2$.
    \item Suppose the sutures $\ga_{(x_1,y_1)}$, $\ga_{(x_2,y_2)}$, and $\ga_{(x_3,y_3)}$ are defined as in Definition \ref{defn: xi and yi}. Then it is easy to see the arcs corresponding to these three sutures bound a triangle containing no lattice point in the interior. There are two different triangles up to translation, which correspond to two different bypass triangles. All bypass maps are clockwise in $\mathbb{R}^2$. Rotation around the origin by $180$ degrees will switch the roles of $\psi_{+,*}^*$ and $\psi_{-,*}^*$.
    \item The height of the middle point of the straight arc indicates the grading before stabilization (so there are gradings of half integers). If the top endpoints of two arcs are the same, the grading shift is about $\hat{i}_{min}^*$. If the bottom endpoints of two arcs are the same, the grading shift is about $\hat{i}_{max}^*$.
\eenu
\erem
\subsection{Decomposing framed instanton Floer homology}\label{subsection: Decomposing framed instanton Floer homology}
\quad

In this subsection, we prove term (2) of Proposition \ref{prop: I sharp and SHI}. Throughout this subsection, let $K\subset Y$ be a null-homologous knot and let the pair $(\widehat{Y},\widehat{K})$ be obtained from $(Y,K)$ by a $q/p$ surgery with $q>0$. Suppose the sutures $\widehat{\Ga}_n=\ga_{(p_n,q_n)}$ and $\widehat{\Ga}_{\mu}$ are defined as in Definition \ref{defn_4: Ga_n-hat sutures} and suppose $S$ is a minimal genus Seifert surface of $K$. The stabilization $S^{\tau}$ of $S$ is chosen according to Definition \ref{defn_4: S tau}. The maximal and minimal gradings of the involved sutured instanton Floer homology are described in Lemma \ref{lem_4: top and bottom non-vanishing gradings}. For any $i\in\intg$, let
$$\psi_{\pm,n+1}^{n,i}=\psi_{\pm,n+1}^{n}|_{\shi(-Y(K),-\widehat{\Ga}_{n},S^{\tau},i)}$$
be the restriction.

\blem\label{lem_4: psi pm is isomorphism}
Suppose $n\in\intg$ so that $q_{n+1}=q_n+q$, \textit{i.e.} $n\ge 0$. Then the map
$$\psi_{+,n+1}^{n,i}: \shi(-Y(K),-\widehat{\Ga}_{n},S^{\tau},i)\ra\shi(-Y(K),-\widehat{\Ga}_{n+1},S^{\tau},i-\hat{i}_{min}^n+\hat{i}_{min}^{n+1})$$
is an isomorphism if $i\leq \hat{i}_{max}^n-2g(S)$. Similarly, the map
$$\psi_{-,n+1}^{n,i}: \shi(-Y(K),-\widehat{\Ga}_{n},S^{\tau},i)\ra\shi(-Y(K),-\widehat{\Ga}_{n+1},S^{\tau},i-\hat{i}_{max}^n+\hat{i}_{max}^{n+1})$$
is an isomorphism if $i\geq \hat{i}_{min}^n+2g(S)$.
\elem

\bpf
The proof the lemma is similar to that of Lemma \ref{lem_3: iso at particular gradings}.
\epf

% \brem\label{rem_4: large enough n}
% Recall in Definition \ref{defn_4: Ga_n-hat sutures}, we require that
% $$\hat{\lambda}-n\hat{\mu}=\pm(p_n\lambda+q_n\mu),$$
% where we always take $q_n>0$. Recall also $\hat{\mu}=q\mu+p\lambda$ and $q>0$, so it follows that $q_{n+1}=q_n+q$ for a large enough integer $n$.
% \erem

\blem\label{lem_4: G_n is zero}
For any $n\in\intg$ so that $q_{n+1}-q=q_n\geq 2g$, the map
$$G_{n}:\shi(-\widehat{Y}(1),-\delta)\ra \shi(-Y(K),-\widehat{\Ga}_n)$$
defined as in (\ref{eq_2: surgery exact triangle}) is the zero map.
\elem
\bpf
The proof of the lemma is similar to that of Lemma \ref{lem_3: G_n is zero}.
\epf

\bcor\label{cor_4: KHI and I sharp}
We have
$$\dim_{\mathbb{C}}\shi(-Y(K),-\widehat{\Ga}_{\mu})\geq\dim_{\mathbb{C}}\shi(-\widehat{Y}(1),-\delta).$$
\ecor
\bpf
Since for a large enough integer $n$, we have $G_n=0$, we know that
$$\dim_{\mathbb{C}}\shi(-\widehat{Y}(1),-\delta)=\dim_{\mathbb{C}}\shi(-Y(K),-\widehat{\Ga}_{n+1})-\dim_{\mathbb{C}}\shi(-Y(K),-\widehat{\Ga}_{n}).$$
Then the corollary follows directly from Proposition \ref{prop_4: bypass exact triangle for Ga_n-hat sutures}.
\epf

\bcor\label{cor_4: KHI for links}
Suppose $L_1$ is a non-empty link in $\widehat{Y}$ that is disjoint from $\widehat{K}$. Let $L_2=L_1\cup \widehat{K}$. Consider the link complements $-\widehat{Y}(L_1)$ and $-\widehat{Y}(L_2)$. For the link complement, let $\widehat{\Ga}_{\mu}$ be the suture consisting of two meridians for each component of the link.
We have
$$\dim_{\mathbb{C}}\shi(-\widehat{Y}(L_2),-\widehat{\Ga}_{\mu})\geq 2\cdot\dim_{\mathbb{C}}\shi(-\widehat{Y}(L_1),-\widehat{\Ga}_{\mu})$$
\ecor

\bpf
The same argument to prove Corollary \ref{cor_4: KHI and I sharp} can be applied verbatim to verify
$$\dim_{\mathbb{C}}\shi(-\widehat{Y}(L_2),-\widehat{\Ga}_{\mu})\geq\shi(-\widehat{Y}(L_1)(1),-\widehat{\Ga}_{\mu}\cup-\delta),$$
where $\widehat{Y}(L_1)(1)$ is obtained from $-\widehat{Y}(L_1)$ by removing a 3-ball disjoint from $L_1$, and $\delta$ is a simple closed curve on the new spherical boundary component. From \cite[Lemma 4.14]{baldwin2016instanton}, we have
$$\shi(-\widehat{Y}(L_1)(1),-\widehat{\Ga}_{\mu}\cup-\delta)=2\cdot\dim_{\mathbb{C}}\shi(-\widehat{Y}(L_1),-\widehat{\Ga}_{\mu}).$$
\epf

\blem\label{lem_4: q-cyclic}
Suppose $n\in\intg$ satisfies $q_{n+1}-q=q_n\geq q+2g(S)$, and suppose $i,j\in \intg$ with
$$\hat{i}_{min}^n+2g(S)\leq i,j\leq \hat{i}_{max}^n-2g(S){\rm~and~}i-j=q.$$
Then we have
$$\shi(-Y(K),-\widehat{\Ga}_n,S^{\tau},i)\cong \shi(-Y(K),-\widehat{\Ga}_n,S^{\tau},j).$$
\elem
\bpf
Since $i\leq \hat{i}_{max}^n-2g(S)$, by Lemma \ref{lem_4: psi pm is isomorphism}, we know
$$\shi(-Y(K),-\widehat{\Ga}_{n},S^{\tau},i)\cong\shi(-Y(K),-\widehat{\Ga}_{n+1},S^{\tau},i-\hat{i}_{min}^n+\hat{i}_{min}^{n+1}).$$
Similarly, since $j\geq \hat{i}_{min}^n+2g(S)$, we know that
$$\shi(-Y(K),-\widehat{\Ga}_{n},S^{\tau},j)\cong\shi(-Y(K),-\widehat{\Ga}_{n+1},S^{\tau},j-\hat{i}_{max}^n+\hat{i}_{max}^{n+1}).$$
Note also that
\beq
i-\hat{i}_{min}^n+\hat{i}_{min}^{n+1}=&j-\hat{i}_{max}^n+\hat{i}_{max}^{n+1}+q+(\hat{i}_{max}^n-\hat{i}_{min}^n)-(\hat{i}_{max}^{n+1}-\hat{i}_{min}^{n+1})\\
=&j-\hat{i}_{max}^n+\hat{i}_{max}^{n+1}+q+(q_n-1+2g(S))-(q_{n+1}-1+2g(S))\\
=&j-\hat{i}_{max}^n+\hat{i}_{max}^{n+1}.
\eeq
Hence we obtain the desired result.
\epf

\bdefn\label{defn_4: essential component}
Suppose $n\in\intg$ satisfies $q_{n+1}-q=q_n\geq q+2g(S)$. Define
$$\mathcal{I}_+(-\widehat{Y},\widehat{K},i)=\shi(-Y(K),-\widehat{\Ga}_n,S^{\tau},\hat{i}_{max}^n-2g(S)-i),$$
and
$$\mathcal{I}_+\wyk=\bigoplus_{i=0}^{q-1}\mathcal{I}_+(-\widehat{Y},\widehat{K},i).$$
\edefn

\brem\label{rem: grading shift}
From Lemma \ref{lem_4: psi pm is isomorphism}, the definition of $\mathcal{I}_+\wyk$ is independent of the choice of the integer $n$ satisfying the required condition. Also, by Lemma \ref{lem_4: q-cyclic}, the definition of $\mathcal{I}_+\wyk$ would be the same (up to a $\mathbb{Z}_q$ grading shift) if we consider arbitrary $q$ consecutive gradings within the range $[\hat{i}_{min}^n+2g(S),\hat{i}_{max}^n-2g(S)].$
\erem

Next, our goal is to show that there is an isomorphism
$$\mathcal{I}_+\wyk\cong\shi(-\widehat{Y}(1),-\delta).$$
To do so, we first introduce some notations for performing computations.

\begin{defn}\label{defn_4: blocks and sizes}
Suppose $n\in\intg$. The direct sum of some consecutive gradings of $$\shi(-Y(K),-\widehat{\Ga}_n,S^{\tau})$$is called a {\bf block}. For a block $A$, the number of gradings involved is called the {\bf size} of $A$.%, denoted by $s(A)$.
\end{defn}

\begin{exmp}\label{exmp_4: blocks}
Suppose $n\in\intg$ satisfies $q_n\geq q+2g(S)$. Let $A,B,C$ and $D$ be the blocks consisting of the top $2g(S)$ gradings, the next $q$ gradings, the next $q_{n}-q-2g(S)$ gradings, and the last $2g(S)$ gradings of $\shi(-Y(K),-\widehat{\Ga}_n,S^{\tau})$, respectively. We write
\begin{equation*}
\shi(-Y(K),-\widehat{\Ga}_n,S^{\tau})=\left(
\begin{array}{c}
    A\\
    B\\
    C\\
    D
\end{array}
\right).
\end{equation*}
From Definition \ref{defn_4: essential component}, we know that $\mathcal{I}_+\wyk$ is itself a block and in fact
$$\mathcal{I}_+\wyk=B.$$
Also, we can write
\begin{equation*}
\shi(-Y(K),-\widehat{\Ga}_n,S^{\tau})=\left(
\begin{array}{c}
    A\\
    E\\
    F\\
    D
\end{array}
\right),
\end{equation*}
where $E$ and $F$ are of size $(q_{n}-q-2g(S))$ and $q$, respectively. By comparing the gradings, we have
\begin{equation*}
\left(
\begin{array}{c}
    B\\
    C
\end{array}
\right)=\left(
\begin{array}{c}
    E\\
    F
\end{array}
\right)
\end{equation*}
\textit{A priori}, we do not have $B=E$ and $C=F$ since they have different sizes. However, when putting together, the total size of $B$ and $C$ equals that of $E$ and $F$.
\end{exmp}

\blem\label{lem_4: the dimension equals}
Let $\mathcal{I}_+\wyk$ be defined as in Definition \ref{defn_4: essential component}. Then we have
$$\dim_{\mathbb{C}}\mathcal{I}_+\wyk=\dim_{\mathbb{C}}\shi(-\widehat{Y}(1),-\delta).$$
\elem
\bpf
Suppose $n\in\intg$ satisfies $q_n\geq q+2g(S)$. We can apply Proposition \ref{prop_4: graded bypass for Ga_n-hat sutures}. Using blocks, we have the following. (There is no enough room for writing down the whole notation for $\shi$, so we use the sutures to denote them.)
\begin{equation*}
\xymatrix@R=0.2ex{
{\rm size}&-\widehat{\Ga}_{\mu}\ar[rr]^{\psi_{+,n}^{\mu}}&&-\widehat{\Ga}_{n}\ar[rr]^{\psi_{+,n+1}^{n}}&&-\widehat{\Ga}_{n+1}\ar[rr]^{\psi_{+,\mu}^{n+1}}&&-\widehat{\Ga}_{\mu}\\
q&G&&&&X_1&&G\\
2g(S)&H&&A&&X_2&&H\\
q_n-q-2g(S)&&&E&&X_3&&\\
q&&&F&&X_4&&\\
2g(S)&&&D&&X_5&&\\
}
\end{equation*}
From the exactness, we know that
$$X_1=G,~X_3=E,~X_4=F,{\rm~and~}X_5=D.$$

There is another bypass exact triangle, and similarly we have
\begin{equation*}
\xymatrix@R=0.2ex{
{\rm size}&-\widehat{\Ga}_{\mu}\ar[rr]^{\psi_{-,n}^{\mu}}&&-\widehat{\Ga}_{n}\ar[rr]^{\psi_{-,n+1}^{n}}&&-\widehat{\Ga}_{n+1}\ar[rr]^{\psi_{-,\mu}^{n+1}}&&-\widehat{\Ga}_{\mu}\\
2g(S)&&&A&&A&&\\
q&&&B&&B&&\\
q_n-q-2g(S)&&&C&&C&&\\
2g(S)&I&&D&&X_6&&I\\
q&J&&&&J&&J\\
}
\end{equation*}
Comparing the two expressions of $\shi(-Y(K),-\widehat{\Ga}_{n+1},S^{\tau})$, we have
\begin{equation*}
\left(
\begin{array}{c}
    G\\
    X_2\\
    E\\
    F\\
    D
\end{array}
\right)=\shi(-Y(K),-\widehat{\Ga}_{n+1},S^{\tau})=\left(
\begin{array}{c}
    A\\
    B\\
    C\\
    X_6\\
    J
\end{array}
\right).
\end{equation*}
Taking sizes into consideration, we know that
\begin{equation*}
\left(
\begin{array}{c}
    G\\
    X_2
\end{array}
\right)=\left(
\begin{array}{c}
    A\\
    B
\end{array}
\right),~E=C,{\rm~and~}\left(
\begin{array}{c}
    F\\
    D
\end{array}
\right)=\left(
\begin{array}{c}
    X_6\\
    J
\end{array}
\right)
\end{equation*}
Thus, we know that
\begin{equation*}
\shi(-Y(K),-\widehat{\Ga}_{n+1},S^{\tau})=\left(
\begin{array}{c}
    A\\
    B\\
    E\\
    F\\
    D
\end{array}
\right).
\end{equation*}
Comparing this expression with the expression of $\shi(-Y(K),-\widehat{\Ga}_{n},S^{\tau})$ in Example \ref{exmp_4: blocks}, we have
\beq
\dim_{\mathbb{C}}\mathcal{I}_+\wyk=&\dim_{\mathbb{C}}B\\
=&\dim_{\mathbb{C}}\shi(-Y(K),-\widehat{\Ga}_{n+1})-\dim_{\mathbb{C}}\shi(-Y(K),-\widehat{\Ga}_{n})\\
=&\dim_{\mathbb{C}}\shi(-\widehat{Y}(1),-\delta),
\eeq
where the last equality follows from Lemma \ref{lem_4: G_n is zero}.
\epf

\bprop\label{prop_4: F_n is an isomorphism}
Suppose $n\in\intg$ satisfies $q_{n+1}-q=q_n\geq q+2g(S)$. Then the map $F_n$ restricted to $\mathcal{I}_+\wyk$ is an isomorphism, \textit{i.e.},
$$F_n|_{\mathcal{I}_+\wyk}:\mathcal{I}_+\wyk\xra{\cong} \shi(-\widehat{Y}(1),-\delta)$$
\eprop

\bpf
It suffices to show that the restriction of $F_n$ is surjective. Since $q_n\geq q+2g(S)$, we have $q_{n-1}=q_n-q\geq 2g(S)$. By Lemma \ref{lem_4: G_n is zero}, we know that $G_{n-1}=0$. By exactness in (\ref{eq_2: surgery exact triangle}), the map $F_n$ is surjective. Then it suffices to show that $F_n$ remains surjective when restricted to ${\mathcal{I}_+\wyk}$. For any $x\in \shi(-\widehat{Y}(1),-\delta)$, let $y\in\shi(-Y(K),-\widehat{\Ga}_{n})$ be an element so that $F_n(y)=x.$ Suppose
$$y=\sum_{j\in\intg}y_j,\text{ where }y_j\in\shi(-Y(K),-\widehat{\Ga}_{n},S^{\tau},j).$$
For any $y_j$, we want to find $y'_j\in\mathcal{I}_+\wyk$ so that $F_n(y_j)=F_n(y_j').$

% $\psi_{-,n+1}^{n,j-mq}$,\dots, $\psi_{-,n+m}^{n,\hat{i}_{max}^{n+m}-\hat{i}_{max}^n+j-mq}$
To do this, we first assume that $j>\hat{i}_{max}^n-2g(S)$. Then there exists an integer $m$ so that
$$\hat{i}_{max}^n-2g(S)-q+1\leq j-mq\leq \hat{i}_{max}^n-2g(S).$$
We can take
\begin{equation}\label{eq: many bypasses}
    y_j'=(\psi_{-,n+1}^{n,j-mq})^{-1}\circ\cdots\circ(\psi_{-,n+m}^{n,\hat{i}_{max}^{n+m}-\hat{i}_{max}^n+j-mq})^{-1}\circ\psi_{+,n+m}^{n+m-1}\circ\dots\circ\psi_{+,n+1}^n(y_j).
\end{equation}
From Lemma \ref{lem_4: psi pm is isomorphism}, all the negative bypass maps involved in (\ref{eq: many bypasses}) are isomorphisms so the inverses exist. Also, we have $$y_j'\in\shi(-Y(K),-\widehat{\Ga}_{n},S^{\tau},j-mq)\subset \mathcal{I}_+\wyk.$$
Finally, from Lemma \ref{lem_4: surgery triangle for knots}, we know that $F_n(y'_j)=F_n(y_j).$

For $$j\in [ \hat{i}_{max}^n-2g(S)-q-1,\hat{i}_{max}^n-2g(S)],$$we can simply take $y_j'=y_j$.

For $j\leq\hat{i}_{max}^n-2g(S)-q-1$, we can pick $y_j'$ similarly as in (\ref{eq: many bypasses}), while switching the roles of $\psi_{+,*}^*$ and $\psi_{-,*}^*$ in (\ref{eq: many bypasses}).

In summary, we can take
$$y'=\sum_{j\in\intg}y_j'\in\mathcal{I}_+\wyk\text{ with }F_n(y')=F_n(y)=x.$$
Hence the restriction of $F_n$ is still surjective, and we obtain the desired result.
\epf

In Definition \ref{defn_4: essential component}, we use a large enough integer $n$ to define $\mathcal{I}_+\wyk$. We can also use a small enough integer $n$ to define a vector space similar to $\mathcal{I}_+\wyk$. Recall
$$\widehat{\Ga}_n=\ga_{(p_n,q_n)}$$
is defined as in Definition \ref{defn_4: Ga_n-hat sutures} and $q_n$ is chosen to be always non-negative.

\bdefn\label{defn_4: essential component, 2}
Suppose $n\in\intg$ satisfies $q_{n-1}-q=q_n\geq q+2g(S)$. Define
$$\mathcal{I}_-(-\widehat{Y},\widehat{K},i)=\shi(-Y(K),-\widehat{\Ga}_n,S^{\tau},\hat{i}_{max}^n-2g(S)-i),$$
and
$$\mathcal{I}_-\wyk=\bigoplus_{i=0}^{q-1}\mathcal{I}_-(-\widehat{Y},\widehat{K},i).$$
\edefn

The arguments for $\mathcal{I}_-\wyk$ are similar to those for $\mathcal{I}_+\wyk$. We sketch them as follows.

\blem\label{lem_4: psi pm is isomorphism, 2}
Suppose $n\in\intg$ satisfies $q_{n-1}-q=q_n$, \textit{i.e.} $n<-1$. Then the map
$$\psi_{+,n}^{n-1,i+\hat{i}_{max}^{n-1}-\hat{i}_{max}^{n}}:\shi(-Y(K),-\widehat{\Ga}_{n-1},S^{\tau},i+\hat{i}_{max}^{n-1}-\hat{i}_{max}^{n})\ra\shi(-Y(K),-\widehat{\Ga}_{n},S^{\tau},i)$$
is an isomorphism if $i\leq \hat{i}_{max}^{n}-2g(S)$. Similarly, the map
$$\psi_{-,n}^{n-1,i-\hat{i}_{min}^{n}+\hat{i}_{min}^{n-1}}:\shi(-Y(K),-\widehat{\Ga}_{n-1},S^{\tau},i-\hat{i}_{min}^{n}+\hat{i}_{min}^{n-1})\ra\shi(-Y(K),-\widehat{\Ga}_{n},S^{\tau},i)$$
is an isomorphism if $i\geq \hat{i}_{min}^{n}+2g(S)$.
\elem
\bpf
The proof is similar to the proof of Lemma \ref{lem_3: iso at particular gradings}.
\epf

\blem\label{lem_4: F_n is zero}
For any $n\in\intg$ so that $q_{n-1}-q=q_n\geq 2g$, the map
$$F_{n}:\shi(-Y(K),-\widehat{\Ga}_n)\ra \shi(-\widehat{Y}(1),-\delta)$$
defined as in (\ref{eq_2: surgery exact triangle}) is the zero map.
\elem
\bpf
If it is not, then let $j_{\max}\in\intg$ be the maximal index $j$ so that there exists $$x\in \shi(-Y(K),-\widehat{\Ga}_n,S^{\tau},j)$$ with
$F_n(x)\neq0.$ Since $q_n\geq 2g$, by Lemma \ref{lem_4: top and bottom non-vanishing gradings}, we know that either $$j_{max}\leq \hat{i}_{max}^{n}-2g(S)~{\rm or}~j_{max}\geq \hat{i}_{min}^{n}+2g(S).$$

Suppose, without loss of generality, that $j_{max}\geq \hat{i}_{min}^{n}+2g(S)$ and $$x\in \shi(-Y(K),-\widehat{\Ga}_n,S^{\tau},j_{max})$$ satisfying $F_n(x)\neq0.$ By Lemma \ref{lem_4: psi pm is isomorphism, 2}, $\psi_{+,n}^{n-1,j_{max}+\hat{i}_{max}^{n-1}-\hat{i}_{max}^{n}}$ is an isomorphism, and we can take
$$y=\psi_{-,n}^{n-1,j_{max}+\hat{i}_{max}^{n-1}-\hat{i}_{max}^{n}}\circ(\psi_{+,n}^{n-1,j_{max}+\hat{i}_{max}^{n-1}-\hat{i}_{max}^{n}})^{-1}(x).$$
By Lemma \ref{lem_4: surgery triangle for knots} we know that
$$F_n(y)=F_n(x)\neq 0{\rm~and~}y\in \shi(-Y(K),-\widehat{\Ga}_n,S^{\tau},j_{max}+q).$$
This is a contradiction.
\epf

\blem\label{lem_4: q-cyclic, 2}
Suppose $n\in\intg$ satisfies $q_{n-1}-q=q_n\geq q+2g(S)$, and suppose $i,j\in \intg$ satisfying
$$\hat{i}_{min}^n+2g(S)\leq i,j\leq \hat{i}_{max}^n-2g(S){\rm~and~}i-j=q.$$
Then we have
$$\shi(-Y(K),-\widehat{\Ga}_n,S^{\tau},i)\cong \shi(-Y(K),-\widehat{\Ga}_n,S^{\tau},j).$$
\elem

\bpf
The proof is similar to the proof of Lemma \ref{lem_4: q-cyclic}.
\epf

\blem\label{lem_4: the dimension equals, 2}
Let $\mathcal{I}_-\wyk$ be defined as in Definition \ref{defn_4: essential component, 2}. Then we have
$$\dim_{\mathbb{C}}\mathcal{I}_-\wyk=\dim_{\mathbb{C}}\shi(-\widehat{Y}(1),-\delta).$$
\elem

\bpf
The proof is similar to the proof of Lemma \ref{lem_4: the dimension equals}.
\epf

\bprop\label{prop_4: G_n is an isomorphism}
Suppose $n\in\intg$ satisfies $q_{n-1}-q=q_n\geq q+2g(S)$. Let $\Pi_n$ be the projection
$$\Pi_n: \shi(-Y(K),-\widehat{\Ga}_n)\ra \mathcal{I}_-\wyk.$$
Then we have an isomorphism
$$\Pi_n\circ G_n:\shi(-\widehat{Y}(1),-\delta)\xra{\cong} \mathcal{I}_-\wyk.$$
\eprop

\bpf
It suffices to show that $\Pi_n\circ G_n$ is injective. We assume it is not true and derive a contradiction. By assumption, there exists $$x\neq0\in \shi(-\widehat{Y}(1),-\delta) \text{ with } \Pi_n\circ G_n(x)=0.$$Write
$$y=G_n(x)=\sum_{j\in\intg}y_j,{\rm~where~}y_j\in \shi(-Y(K),-\widehat{\Ga}_n,S^{\tau},j).$$
From Lemma \ref{lem_4: surgery triangle for knots} and Lemma \ref{lem_4: F_n is zero}, we know that $G_n$ is injective, and hence $y\neq0$. From the assumption, we know that
$$y_j=0{~\rm for~}\hat{i}_{max}^n-2g(S)\geq j\geq \hat{i}_{max}^n-2g(S)-q+1.$$
Also write
$$z=G_{n-1}(x)=\sum_{j\in\intg}z_j,{\rm~where~}z_j\in \shi(-Y(K),-\widehat{\Ga}_{n-1},S^{\tau},j).$$
From Lemma \ref{lem_4: surgery triangle for knots}, we know that
$$\psi_{-,n}^{n-1}(z)=y=\psi_{+,n}^{n-1}(z).$$
Suppose $j_{min}$ is the minimal grading $j$ so that
$$j>\hat{i}_{max}^n-2g(S){\rm~and~}y_j\neq0.$$
Then we know that
$$y_{j_{min}-q}=0{\rm~and~}j_{min}-q\geq\hat{i}_{min}^n+2g(s).$$
Hence by Lemma \ref{lem_4: psi pm is isomorphism, 2}, we know that
\beq
y_{j_{min}}=&\psi_{-,n}^{n-1,j_{min}-\hat{i}^{n}_{min}+\hat{i}^{n-1}_{min}}(z_{j_{min}-\hat{i}^{n}_{min}+\hat{i}^{n-1}_{min}})\\
=&\psi_{-,n}^{n-1,j_{min}-\hat{i}^{n}_{min}+\hat{i}^{n-1}_{min}}\circ(\psi_{+,n}^{n-1,j_{min}-\hat{i}^{n}_{min}+\hat{i}^{n-1}_{min}})^{-1}(y_{j_{min}-q})\\
=&0.
\eeq
This implies that $y_j=0$ for all $j\geq \hat{i}_{max}^n-2g(S)-q+1$. Similarly we can prove that $y_j=0$ for all $j<\hat{i}_{max}^n-2g(S)-q+1$, and $y=0$, which contradicts the injectivity of $G_n$.
\epf

\subsection{Commutative diagrams for bypass maps}
\quad

In this subsection, we show there are some commutative diagrams for bypass maps. 
\begin{lem}[{\cite[Corollary 2.20]{li2019direct}}]\label{lem: commutative diagram 1}
For any surgery slope $q/p$, consider the bypass maps $\psi_{+,*}^*$ and $\psi_{-,*}^*$ in Proposition \ref{prop_4: bypass exact triangle for Ga_n-hat sutures}. For any integer $n\in\mathbb{Z}$, we have the following commutative diagram.
    \begin{equation}\label{eq: commutative diagram 1}
\xymatrix@R=3ex{
\shi(-Y(K),-\widehat{\Ga}_{n})\ar[rr]^{\psi_{-,n+1}^{n}}\ar[dd]_{\psi_{+,n+1}^n}&&\shi(-Y(K),-\widehat{\Ga}_{n+1})\ar[dd]^{\psi_{+,n+2}^{n+1}}\\
&&\\
\shi(-Y(K),-\widehat{\Ga}_{n+1})\ar[rr]^{\psi_{-,n+2}^{n+1}}&&\shi(-Y(K),-\widehat{\Ga}_{n+2})
}
\end{equation}
\begin{proof}
In Subsection \ref{subsec: general bypass}, we reinterpreted bypass maps by contact gluing maps. So the composition of bypass maps becomes the composition of contact gluing maps. To verify the commutative diagram, it suffices to verify that two contact structures coming from different bypasses are actually the same. Thus, it is free to change the basis of $H_1(T^2)$. It suffices to verify a special case $q/p=1/0$ and $n=0$. Then it follows from \cite[Lemma 4.14]{honda2000classification} that the contact structures are the same.

% Alternatively, to show the propoistion holds for $q/p=1/0$ and $n=0$, we can use the interpretation of bypass maps by contact handle decomposition. Explicitly, consider the proof of Proposition \ref{prop_4: bypass exact triangle for Ga_n-hat sutures}. We start with $\ga=\ga_{(2,-1)}$ and two bypass arcs $\al_{+}$ and $\al_{-}$ as depicted in the left subfigure of Figure \ref{commute1}. Let the suture $\ga^\p$ and the new bypass arc $\theta_{+}$ be obtained by applying the bypass attachment twice along the arc $\al_{+}$ (\textit{c.f.} Figure \ref{fig: the bypass triangle}). See the middle subfigure of Figure \ref{commute1}. The suture $\ga^\p$ is a stabilization of $\ga_{(1,-1)}$ and the bypass attachment along $\theta_{+}$ gives the map $\psi_{+,2}^1$. The arc $\al_-$ becomes a bypass arc on $\ga^\p$. Let the suture $\ga^\pp$ and the new bypass arc $\theta_{-}$ be obtained by applying the bypass attachment twice along the arc $\al_{-}$. See the right subfigure of Figure \ref{commute1}. The suture $\ga^\pp$ is a stabilization of $\ga_{(0,1)}$ and the bypass attachment along $\theta_{-}$ gives the map $\psi_{-,1}^0$. Thus,$$\psi_{+,2}^1\circ \psi_{-,1}^0=\psi_{\theta_+}\circ \psi_{\theta_-}.$$By Lemma \ref{lem_2: disjoint bypass commutes}, these two bypass maps commute. Similar to the above discussion, we have$$ \psi_{-,2}^1\circ \psi_{+,1}^0=\psi_{\theta_-}\circ \psi_{\theta_+}.$$Hence the lemma holds.
\end{proof}
% \begin{figure}[ht]
% \centering
% \includegraphics[width=0.8\textwidth]{fig0/commute1}
% \caption{Bypass arcs on sutures.}
% \label{commute1}
% \end{figure}
% \begin{cor}\label{cor: commutative diagram 1a}
% For any surgery slope $q/p$, consider the bypass maps $\psi_{+,*}^*$ and $\psi_{-,*}^*$ in Proposition \ref{prop_4: bypass exact triangle for Ga_n-hat sutures}. For $i,j\in\mathbb{Z}$ with $i>j$, suppose $$\Psi_{\pm,i}^j=\psi_{\pm,i}^{i-1}\circ \psi_{\pm,i-1}^{i-2}\circ\dots\circ\psi_{\pm,j+1}^j.$$Then we have the following commutative diagram.
%     \begin{equation}\label{eq: commutative diagram 1a}
% \xymatrix@R=3ex{
% \shi(-Y(K),-\widehat{\Ga}_{j})\ar[rr]^{\psi_{+,j+1}^{j}}\ar[dd]_{\Psi_{-,i}^{j}}&&\shi(-Y(K),-\widehat{\Ga}_{j+1})\ar[dd]^{\Psi_{-,i+1}^{j+1}}\\
% &&\\
% \shi(-Y(K),-\widehat{\Ga}_{i})\ar[rr]^{\psi_{+,i+1}^{i}}&&\shi(-Y(K),-\widehat{\Ga}_{i+1})
% }
% \end{equation}
% The similar commutative diagram holds if we switch the roles of $\psi_{+,*}^*$ and $\psi_{-,*}^*$, $\Psi_{+,*}^*$ and $\Psi_{-,*}^*$, respectively.
% \end{cor}
% \begin{proof}
% Combining the commutative diagrams from Lemma \ref{lem: commutative diagram 1} with $n=j,j+1,\dots,i-2$, this lemma is straightforward.
% \end{proof}
\end{lem}
\begin{lem}\label{lem: commutative diagram 2}
For any surgery slope $q/p$, consider the bypass maps $\psi_{+,*}^*$ and $\psi_{-,*}^*$ in Proposition \ref{prop_4: bypass exact triangle for Ga_n-hat sutures}. For any $n\in\mathbb{Z}$, we have two commutative diagrams
\begin{equation}\label{eq: commutative diagram 2a}
\xymatrix@R=6ex{
\shi(-Y(K),-\widehat{\Ga}_{n})\ar[rd]_{\psi_{+,\mu}^n}\ar[rr]^{\psi_{-,n+1}^n}&&\shi(-Y(K),-\widehat{\Ga}_{n+1})\ar[ld]^{\psi_{+,\mu}^{n+1}}\\
&\shi(-\widehat{Y}(K),-\widehat{\Ga}_{\mu})&
}
\end{equation}
and
\begin{equation}\label{eq: commutative diagram 2b}
\xymatrix@R=6ex{
\shi(-Y(K),-\widehat{\Ga}_{n})\ar[rr]^{\psi_{-,n+1}^n}&&\shi(-Y(K),-\widehat{\Ga}_{n+1})\\
&\shi(-\widehat{Y}(K),-\widehat{\Ga}_{\mu})\ar[ul]^{\psi_{+,n}^\mu}\ar[ur]_{\psi_{+,n+1}^\mu}&
}
\end{equation}
The similar commutative diagrams hold if we switch the roles of $\psi_{+,*}^*$ and $\psi_{-,*}^*$.
\end{lem}
\brem
The bypass maps in Lemma \ref{lem: commutative diagram 2} are from different bypass exact triangles.
\erem
% \begin{figure}[ht]
% \centering
% \includegraphics[width=0.8\textwidth]{fig0/commute2}
% \caption{Top left, bypass $\be_3$ from $\ga_{(1,1)}$; top middle, bypasses $\al_2$ and $\theta_1$ from $\ga_{(1,-1)}$; top right, after the bypass attachment along $\al_2$; bottom left, top left subfigure after isotopy; bottom middle, top middle figure after isotopy; bottom right, the suture $\ga$.}
% \label{commute2}
% \end{figure}
\begin{proof}[Proof of Lemma \ref{lem: commutative diagram 2}]
Similar to the proof of Lemma \ref{lem: commutative diagram 1}, this lemma follows from Honda's classification of tight contact structures on $T^2\times I$ \cite[Lemma 4.14]{honda2000classification}.

% We only give the proof of (\ref{eq: commutative diagram 2a}). The proof of (\ref{eq: commutative diagram 2b}) and other two commutative diagrams are similar. We use the same strategy as in the proof of Lemma \ref{lem: commutative diagram 2}. Hence it suffices to prove for $q/p=1/0$ and $n=0$. For $\psi_{+,\mu}^{0}$, we start with $\ga_{(1,1)}$. For $\psi_{-,1}^0\circ \psi_{+,\mu}^{1}$, we start with $\ga_{(1,-1)}$. See Figure \ref{commute2}, where $$\psi_{\theta_1}=\psi_{-,1}^0,\psi_{\al_2}=\psi_{+,\mu}^{1},~{\rm and}~\psi_{\be_3}=\psi_{+,\mu}^{0}.$$
% By Lemma \ref{lem_2: disjoint bypass commutes}, we have $$\psi_{\theta_1}\circ \psi_{\al_2}=\psi_{\al_2}\circ \psi_{\theta_1}.$$The bypass arcs $\al_2$ and $\be_3$ are isotopic. By Lemma \ref{lem_2: isotopic bypasses}, we have $$\psi_{\al_2}=\psi_{\be_3}.$$After the bypass attachment along $\al_2$, the bypass $\theta_1$ is a trivial bypass. From the discussion in the proof of Lemma \ref{lem_3: surgery triangle}, the corresponding bypass map is an isomorphism. Explicitly, the third suture $\ga$ in the bypass triangle associated to $\theta_1$ contains a component that bounds a disk. Hence the balanced sutured manifold $(-Y(K),-\ga)$ is not taut. By Theorem \ref{thm_2: SHI detects tautness}, $\shi(-Y(K),-\ga)=0$ and the bypass triangle implies the bypass map $\psi_{\theta_1}$ is an isomorphism.

% Thus, we conclude $$\psi_{-,1}^0\circ \psi_{+,\mu}^{1}=\psi_{+,\mu}^{0}.$$
\end{proof}
% \brem
% Indeed, Lemma \ref{lem: commutative diagram 2} also follows from Honda's classification of tight contact structures on $T^2\times I$ \cite{honda2000classification}.
% \erem
\begin{cor}\label{cor: commutative diagram 3}
For any surgery slope $q/p$, consider the bypass maps $\psi_{+,*}^*$ and $\psi_{-,*}^*$ in Proposition \ref{prop_4: bypass exact triangle for Ga_n-hat sutures}. For any $i,j\in\mathbb{Z}$, we have the following commutative diagrams related to $\psi_{+,*}^*$ and $\psi_{-,*}^*$, respectively.
    \begin{equation}\label{eq: commutative diagram 3}
\xymatrix@R=3ex{
\shi(-Y(K),-\widehat{\Ga}_{\mu})\ar[rr]^{\psi_{\pm,j}^{\mu}}\ar[dd]_{\psi_{\pm,i}^\mu}&&\shi(-Y(K),-\widehat{\Ga}_{j})\ar[dd]^{\psi_{\pm,\mu}^{j}}\\
&&\\
\shi(-Y(K),-\widehat{\Ga}_{i})\ar[rr]^{\psi_{\pm,\mu}^{i}}&&\shi(-Y(K),-\widehat{\Ga}_{\mu})
}
\end{equation}
\end{cor}
\begin{proof}
The commutative diagram related to $\psi_{+,*}^*$ follows from (\ref{eq: commutative diagram 2a}) and (\ref{eq: commutative diagram 2b}). Explicitly, for $i=j+1$, both compositions of maps are equal to $$\psi_{+,\mu}^{j+1}\circ\psi^n_{-,j+1}\circ\psi_{+,j}^\mu.$$The other commutative diagram follows from Lemma \ref{lem: commutative diagram 2} similarly.
\end{proof}
\begin{cor}\label{cor: commutative diagram 3b}
For any surgery slope $q/p$, consider the bypass maps $\psi_{+,*}^*$ and $\psi_{-,*}^*$ in Proposition \ref{prop_4: bypass exact triangle for Ga_n-hat sutures}. For any $n\in\mathbb{Z}$, we have $$\psi_{+,\mu}^{n}\circ\psi_{-,n}^\mu=\psi_{-,\mu}^{n}\circ\psi_{+,n}^{\mu}=0$$and$$\psi_{+,n}^{\mu}\circ\psi_{+,\mu}^n=\psi_{-,n}^{\mu}\circ\psi_{-,\mu}^{n}=0$$
\end{cor}
\begin{proof}
By Lemma \ref{lem: commutative diagram 2} and the exactness, we have $$\psi_{+,\mu}^{n}\circ\psi_{-,n}^\mu=\psi_{+,\mu}^{n+1}\circ\psi^n_{-,n+1}\circ\psi_{-,n}^\mu=0.$$
Other arguments follow from Lemma \ref{lem: commutative diagram 2} and the exactness similarly.
\end{proof}
% \begin{cor}\label{cor: commutative diagram 4}
% For any surgery slope $q/p$, consider the bypass maps $\psi_{+,*}^*$ and $\psi_{-,*}^*$ in Proposition \ref{prop_4: bypass exact triangle for Ga_n-hat sutures}. Consider $\Psi_{\pm,i}^j$ in Corollary \ref{cor: commutative diagram 1a}. For any $i,j\in \mathbb{Z}$ with $i>j$, we have two commutative diagrams.
% \begin{equation}\label{eq: commutative diagram 4a}
% \xymatrix@R=6ex{
% \shi(-Y(K),-\widehat{\Ga}_{j})\ar[rd]_{\psi_{+,\mu}^j}\ar[rr]^{\Psi_{-,i}^j}&&\shi(-Y(K),-\widehat{\Ga}_{i})\ar[ld]^{\psi_{+,\mu}^{i}}\\
% &\shi(-\widehat{Y}(K),-\widehat{\Ga}_{\mu})&
% }
% \end{equation}
% and
% \begin{equation}\label{eq: commutative diagram 4b}
% \xymatrix@R=6ex{
% \shi(-Y(K),-\widehat{\Ga}_{j})\ar[rr]^{\psi_{-,i}^j}&&\shi(-Y(K),-\widehat{\Ga}_{i})\\
% &\shi(-\widehat{Y}(K),-\widehat{\Ga}_{\mu})\ar[ul]^{\psi_{+,j}^\mu}\ar[ur]_{\psi_{+,i}^\mu}&
% }
% \end{equation}
% The similar commutative diagrams hold if we switch the roles of $\psi_{+,*}^*$ and $\psi_{-,*}^*$, $\Psi_{+,*}^*$ and $\Psi_{-,*}^*$, respectively.
% \end{cor}
% \begin{proof}
% Combining the commutative diagrams from Lemma \ref{lem: commutative diagram 1} with $n=j,j+1,\dots,i-1$, this lemma is straightforward.
% \end{proof}
\brem\label{curve invariant2}
The above commutative diagrams can be illustrated by the method described in Remark \ref{curve invariant}. The illustration of the special cases in the proofs is shown in Figure \ref{curve1}. Note that vector spaces are denoted by their sutures (we omit the minus signs), and all maps are bypass maps. They are grading preserving and commute with $F_*$ and $G_*$ by Proposition \ref{prop_4: graded bypass for Ga_n-hat sutures} and Lemma \ref{lem_4: surgery triangle for knots}, respectively.
\erem
\begin{figure}[ht]
\centering
\includegraphics[width=0.8\textwidth]{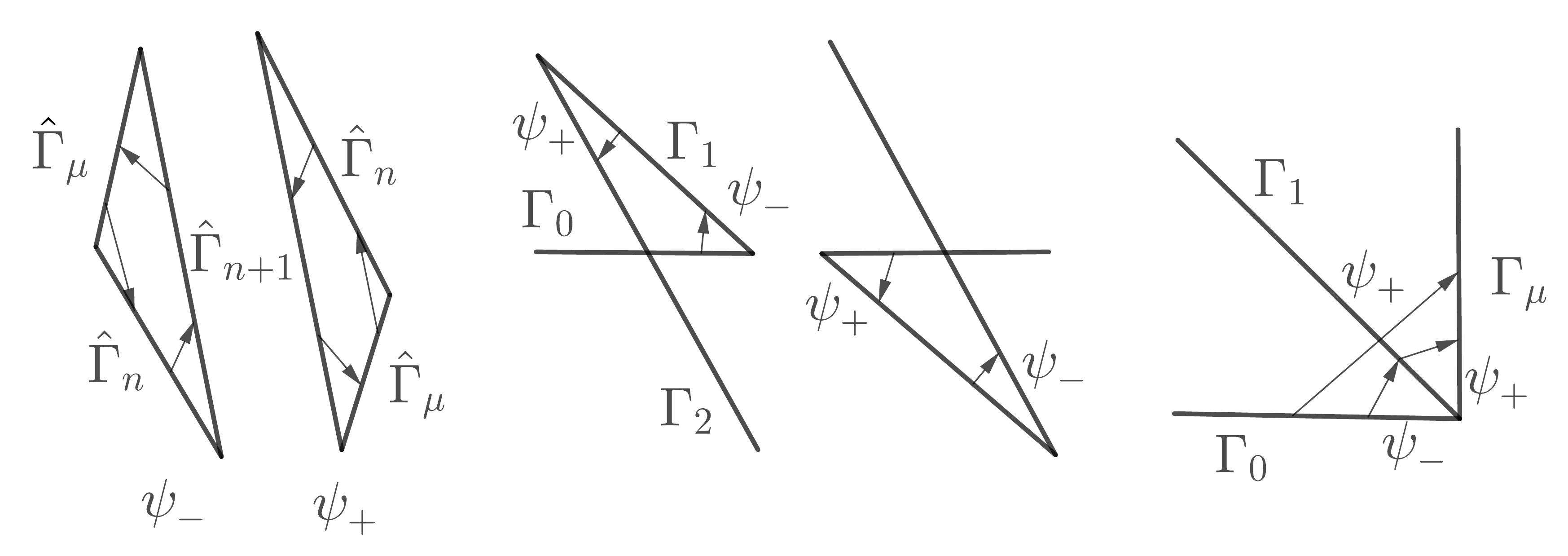}
\caption{Left, bypass maps; middle, illustration of (\ref{eq: commutative diagram 1}); right, illustration of (\ref{eq: commutative diagram 2a}).}
\label{curve1}
\end{figure}

\subsection{The stabilization of integral surgeries}\label{subsec: surgeries on knots}
\quad

Throughout this subsection, suppose $K\subset Y$ is a null-homologous knot and $S$ is a minimal genus Seifert surface of $K$. The stabilization $S^{\tau}$ of $S$ is chosen in Definition \ref{defn_4: S tau}. Since we might work with different surgery slopes, we will use the notation $\ga_{(p,q)}$ in Definition \ref{defn_4: Ga_n-hat sutures} to denote the suture on the knot complement. The maximal and minimal gradings of $\shi(-Y(K),-\ga_{(p,q)},S^{\tau})$ are described explicitly in Lemma \ref{lem_4: top and bottom non-vanishing gradings} and we write them as
$$i_{max}^q=\lceil\frac{q-1}{2}\rceil+g(S){\rm~and~} i_{min}^q=\lceil-\frac{q-1}{2}\rceil-g(S).$$
Note that  they are independent of $p$.

Since we will deal with different surgeries in the current subsection, we will write out the surgery slope explicitly: for the 3-manifold obtained by the $q/p$-surgery, we write $\widehat{Y}_{q/p}$. For the special class of sutures, we write $\widehat{\Ga}_n(q/p)$ instead of $\widehat{\Ga}_n$. For bypass maps, we write $\psi_{\pm,n}^{n+1}(q/p)$. However, if $q/p=1/0$, we still omit it from the notation, then in this case, we simply write $\widehat{\Ga}_n$ as $\Ga_n$.

In this subsection, we deal with large integral surgeries. In this case, we know that
$\Ga_n=\ga_{(-1,n)}$. Hence we have the following by Lemma \ref{lem_4: q-cyclic} and Proposition \ref{prop_4: F_n is an isomorphism}. (Note $(p,q)=(0,1)$.)

\blem\label{lem_4: structure of gamma_(-1,n)}
For any $n>2g(S)$ and $i\in\intg$ so that
$$\lceil-\frac{n-1}{2}\rceil+g(S)\leq i\leq \lceil\frac{n-1}{2}\rceil-g(S),$$
we have
$$\shi(-Y(K),-\Ga_{n},S^{\tau},i)\cong \shi(-Y(1),-\delta).$$
\elem

\brem
A more direct explanation of Lemma \ref{lem_4: structure of gamma_(-1,n)} is that, apart from the top $2g(S)$ and the bottom $2g(S)$ gradings, the vector spaces in all gradings are isomorphic to $\shi(-Y(1),-\delta)$.
\erem
Next, suppose we perform a $(-n)$-surgery. We can take
$$\hat{\lambda}=(0,-1)=-\mu{~\rm and}~\hat{\mu}=(-1,n)=\lambda-n\mu.$$
Then we compute
\begin{equation}\label{eq_4: Gamma_n-hat sutures for n=0,1,2}
\widehat{\Ga}_{\mu}(-n)=\Ga_{n},~\widehat{\Ga}_0(-n)=\Ga_{\mu},~\widehat{\Ga}_1(-n)=\ga_{(-1,n-1)}=\Ga_{n-1},~\widehat{\Ga}_2(-n)=\ga_{(-2,2n-1)},
\end{equation}
and also

\begin{equation}\label{eq_4: Gamma_n-hat sutures for n=-1}
    \widehat{\Ga}_{-1}(-n)=\ga_{(-1,n+1)}=\Ga_{n+1},~\widehat{\Ga}_{-2}(-n)=\ga_{(-2,2n+1)}.
\end{equation}
Observe that
\begin{equation}\label{eq_4: relating I_+ and I_- by sutures}
\widehat{\Ga}_{-2}(-n)=\widehat{\Ga}_2(-n-1).
\end{equation}

The following is the first part of \ref{prop: all but 2g-1 summands are trivial}.
\bprop[]\label{prop_4: structure of I_+(-Y-hat,K-hat)}
Suppose integer $n\geq 2g(S)+1$, then
$$\mathcal{I}_+(-\widehat{Y}_{-n},\widehat{K},i)\cong\shi(-Y(1),-\delta)$$
for any integer $i\in[0,n-2g(S)-1]$ and $i=n-1$.
\eprop
\bpf
When $0\leq i\leq n-1-2g(S)$, we know from equality (\ref{eq_4: Gamma_n-hat sutures for n=0,1,2}), Definition \ref{defn_4: essential component}, Lemma \ref{lem_4: psi pm is isomorphism}, and Lemma \ref{lem_4: structure of gamma_(-1,n)} that
\beq
\mathcal{I}_+(-\widehat{Y}_{-n},\widehat{K},i)=&\shi(-Y(K),-\widehat{\Ga}_{2}(-n),S^{\tau},i^{2n-1}_{max}-2g(S)-i)\\
\cong&\shi(-Y(K),-\widehat{\Ga}_{\mu}(-n),S^{\tau},i^{n}_{max}-2g(S)-i)\\
=&\shi(-Y(K),-{\Ga}_{n},S^{\tau},i^{n}_{max}-2g(S)-i)\\
\cong &\shi(-Y(1),-\delta).\\
\eeq

For $i=n-1$, we know similarly that
\beq
\mathcal{I}_+(-\widehat{Y}_{-n},\widehat{K},i)=&\shi(-Y(K),-\widehat{\Ga}_{2}(-n),S^{\tau},i^{2n-1}_{max}-2g(S)-n+1)\\
\cong&\shi(-Y(K),-\widehat{\Ga}_{\mu}(-n),S^{\tau},i^{n}_{max}-2g(S))\\
=&\shi(-Y(K),-{\Ga}_{n},S^{\tau},i^{n}_{max}-2g(S))\\
\cong &\shi(-Y(1),-\delta).\\
\eeq
\epf
The following is the second part of Proposition \ref{prop: all but 2g-1 summands are trivial}.
\bprop[]\label{prop_4: -Y_n-hat and -Y_n+1-hat}
Suppose integer $n\geq 2g(S)+1$, then for any integer $i\in[0,n-1]$, we have
$$\mathcal{I}_+(-\widehat{Y}_{-n-1},\widehat{K},i+1)\cong \mathcal{I}_+(-\widehat{Y}_{-n},\widehat{K},i).$$
\eprop
\bpf
For integer $i\in[0,n-1-2g(S)]$ and $i=n-1$, we know from Proposition \ref{prop_4: structure of I_+(-Y-hat,K-hat)} that
$$\mathcal{I}_+(-\widehat{Y}_{-n-1},\widehat{K},i+1)\cong\shi(-Y(1),-\delta)\cong \mathcal{I}_+(-\widehat{Y}_{-n},\widehat{K},i)$$
For the rest $(2g(S)-1)$ gradings, we use the following commutative diagram.

\begin{equation}\label{eq_3: commutative diagram}
\xymatrix{
\shi(-Y(K),-\Ga_{n-1})\ar[r]^{\psi_{+,n}^{n-1}}&\shi(-Y(K),-\Ga_{n})\\
\shi(-Y(K),-\Ga_{n})\ar[u]^{\psi_{-,1}^{\mu}(-n)}\ar[r]^{\psi_{-,n+1}^{n}}&\shi(-Y(K),-\Ga_{n+1})\ar[u]^{\psi_{-,1}^{\mu}(-n-1)}
}
\end{equation}
This is directly from facts (\ref{eq_4: Gamma_n-hat sutures for n=0,1,2}) and (\ref{eq_4: Gamma_n-hat sutures for n=-1}), and Corollary \ref{cor: commutative diagram 3} by taking $i=1$ and $j=-1$. Note that $$\psi_{+,n}^{n-1}=\psi_{-,\mu}^1(-n),\psi_{-,n+1}^{n}=\psi_{-,-1}^{\mu}(-n),~{\rm and}~\psi_{-,1}^{\mu}(-n-1)=\psi_{-,\mu}^{-1}(-n).$$
From Proposition \ref{prop_4: graded bypass for Ga_n-hat sutures} and Definition \ref{defn_4: essential component}, we obtain a graded commutative diagram
\begin{equation*}
\xymatrix{
\shi(-Y(K),-\Ga_{n-1}, S^{\tau},i_{min}^{n-1}+n-2-i)\ar[r]^{\psi_{+,n}^{n-1}}&\shi(-Y(K),-\Ga_{n}, S^{\tau},i_{min}^{n}+n-2-i)\\
\shi(-Y(K),-\Ga_{n},S^{\tau},i^{n}_{max}+n-2g-1-i)\ar[u]^{\psi_{-,1}^{\mu}(-n)}\ar[r]^{\psi_{-,n+1}^{n}}&\shi(-Y(K),-\Ga_{n+1},S^{\tau},i^{n+1}_{max}+n-2g-1-i)\ar[u]^{\psi_{-,1}^{\mu}(-n-1)}
}
\end{equation*}
For any fixed integer $i\in [n-2g(S),n-1]$, we have a graded exact triangle
\begin{equation*}
\xymatrix{
\shi(-Y(K),-\Ga_{n-1}, S^{\tau},i_{min}^{n-1}+n-2-i)\ar[rr]&&\mathcal{I}_+(-\widehat{Y}_{-n},\widehat{K},i)\ar[dll]\\
\shi(-Y(K),-\Ga_{n},S^{\tau},i^{n}_{max}+n-2g-1-i)\ar[u]^{\psi_{-,1}^{\mu}(-n)}&&
}
\end{equation*}
Hence we know that
\begin{equation*}
\dim_{\mathbb{C}}\mathcal{I}_+(-\widehat{Y}_{-n},\widehat{K},i)=\dim_{\mathbb{C}}{\rm ker}(\psi_{-,1}^{\mu}(-n))+\dim_{\mathbb{C}}{\rm coker}(\psi_{-,1}^{\mu}(-n)).
\end{equation*}
Similarly, we know that
\begin{equation*}
\dim_{\mathbb{C}}\mathcal{I}_+(-\widehat{Y}_{-n-1},\widehat{K},i+1)=\dim_{\mathbb{C}}{\rm ker}(\psi_{-,1}^{\mu}(-n-1))+\dim_{\mathbb{C}}{\rm coker}(\psi_{-,1}^{\mu}(-n-1))
\end{equation*}

Since the maps $\psi_{+,n}^{n-1}$ and $\psi_{-,n+1}^n$ are both isomorphisms in the corresponding gradings by Lemma \ref{lem_4: psi pm is isomorphism}, we conclude that
$$\dim_{\mathbb{C}}\mathcal{I}_+(-\widehat{Y}_{-n-1},\widehat{K},i+1)=\dim_{\mathbb{C}}\mathcal{I}_+(-\widehat{Y}_{-n},\widehat{K},i).$$By five lemma, there is indeed an isomorphism between $\mathcal{I}_+(-\widehat{Y}_{-n-1},\widehat{K},i+1)$ and $\mathcal{I}_+(-\widehat{Y}_{-n},\widehat{K},i)$.
\epf

\section{Some remarks and future directions}\label{sec: future directions}

In this section, we state some remarks and further directions.

First, the condition in Theorem \ref{thm: torsion spin c decomposition} that the boundary of the Seifert surface $S$ of $\widehat{K}$ is connected can be removed by modifying the hypothesis and the statement as follows.
\benu
	\item Suppose the order of $[\widehat{K}]\in H_1(\widehat{Y})$ is $a$, \textit{i.e.}, $a$ is the minimal positive integer so that $a[K]=0$. Suppose the number of the boundary components of $S$ is $b$, and $\lambda$ is a simple closed curve on $\partial \widehat{Y}(\widehat{K})$ so that $\partial S=b\lambda$.
	\item Choose another simple closed curve $\mu$ on $\partial \widehat{Y}(\widehat{K})$ so that $\mu\cdot\lambda=-1$. Suppose the meridian of $\widehat{K}$ has homology class $(q\mu+p\lambda)$. It can be shown that $q$ is independent of the choice of $\mu$. Indeed, we know that $q=a/b$.
	\item Let definitions of $i^{y}_{max}$ and $i^y_{min}$ in Definition \ref{defn_4: top and bottom non-vanishing gradings} be replaced by $$i^y_{max}=\lceil \frac{1}{2}(yb-\chi(S))\rceil~{\rm and}~i^y_{min}=\lceil -\frac{1}{2}(yb-\chi(S))\rceil.$$
	\item All proofs of Theorem \ref{thm: torsion spin c decomposition} and Proposition \ref{prop: I sharp and SHI} apply without essential change and we will obtain a decomposition associated to $\widehat{K}$$$I^{\sharp}(\widehat{Y})\cong \bigoplus_{i=0}^{a-1}I^{\sharp}(\widehat{Y},i).$$
\eenu
Note that in the original statement of Theorem \ref{thm: torsion spin c decomposition}, the integer $q$ is indeed the order of $[\widehat{K}]$.

Second, as mentioned in Remark \ref{rem: defect}, it is not known if the decomposition of $I^{\sharp}(\widehat{Y})$ is independent of the choice of $\widehat{K}$. Explicitly, we have the following conjecture.

\begin{conj}
Suppose $(\widehat{Y},\widehat{K})$ satisfies the hypothesis of Theorem \ref{thm: torsion spin c decomposition}. Suppose further that another knot $\widehat{K}^\p\subset \widehat{Y}$ satisfies the similar conditions to those of $\widehat{K}$, and $$[\widehat{K}]=[\widehat{K}^\p]\in H_1(\widehat{Y}).$$Then there exists a grading preserving isomorphism$$\mathcal{I}_+(\widehat{Y},\widehat{K})\cong\mathcal{I}_+(\widehat{Y},\widehat{K}^\p)$$up to a $\mathbb{Z}_q$ grading shift, where $\mathcal{I}_+$ is defined as in Definition \ref{defn_4: essential component}.
\end{conj}

% Moreover, the decomposition in Theorem \ref{thm: torsion spin c decomposition} can be generalized to a rationally null-homologous link, \textit{i.e.}, the union of rationally null-homologous knots, to recover more decompositions as a candidate for the counterpart of the torsion spin${}^c$ decompositions in monopole or Heegaard Floer homology. However, the statement would be much more complicated. We leave the discussion to the future.

\quad

Last, though we have had a decomposition of $I^{\sharp}(-\widehat{Y})$, we do not know if it works well with the cobordism maps. For example, let $K\subset S^3$ be a knot and let $S^3_{-n}(K)$ be obtained from $S^3$ by a $(-n)$-surgery along $K$. Then there is a natural cobordism $W$ from $S^3$ to $S^3_{-n}(K)$, which induces a cobordism map
$$I^{\sharp}(W_{-n}):I^{\sharp}(-S^3_{-n}(K))\ra I^{\sharp}(-S^3).$$
Baldwin and Sivek \cite[Section 7]{baldwin2019lspace} proved that the cobordism map decomposes in basic classes:
$$I^{\sharp}(W_{-n})=I^{\sharp}(W_{-n},t_i),$$
where $t_i: H_2(W_{-n})\ra\intg$ maps $[S^\p]$ to $(2i)$, where $[S^\p]$ is the homology class of the surface obtained by capping off of the Seifert surface of $K$. We have the following conjecture, basically saying that the decomposition of the cobordism map $I^{\sharp}(W_{-n})$ is compatible with the decomposition of $I^{\sharp}(-S^3_{-n}(K))$.
\begin{conj}
There exists an integer $N$ so that for any integer $i\in[0,n-1]$, under the identifications $I^{\sharp}(-S^3_{-n}(K))\cong\mathcal{I}_+(-S^3_{-n}(K),\widehat{K})$ and $I^{\sharp}(-S^3)\cong\mathcal{I}_+(-S^3_{1/0}(K),K)$, we have
$$I^{\sharp}(W_{-n},t_{i})=I^{\sharp}(W_{-n})|_{\mathcal{I}_+(-S^3_{-n}(K),\widehat{K},N-i)}$$and this cobordism map can be recovered by bypass maps.
\end{conj}
\bibliographystyle{abbrv}

\end{document}